\documentclass[11pt]{amsart}

\usepackage{amsmath,amssymb,graphicx,color,bbm}
\usepackage{amsthm}
\usepackage[left=1in,right=1in,top=1in,bottom=1in,a4paper]{geometry}
\usepackage{cite}
\usepackage[british]{babel}
\usepackage{hyperref} 
\usepackage{enumitem}
\usepackage{mathtools}
\usepackage{comment}
\usepackage{lipsum}    
\hypersetup{
    colorlinks=false,
    pdfborder={0 0 0},
}
\usepackage{tikz}
\theoremstyle{plain}
\newtheorem{thm}{Theorem}[section]
\newtheorem{lem}[thm]{Lemma}
\newtheorem{prop}[thm]{Proposition}

\theoremstyle{remark}
\newtheorem{rem}[thm]{Remark}

\newcommand{\ind}{{\mathbbm{1}}}
\newcommand{\1}[1]{{\ind\mkern -1.5mu}{\{#1\}}}
\DeclareMathOperator{\Tr}{Tr}

\DeclareMathOperator{\E}{\mathbb E}
\renewcommand{\P}{\mathbb P}
\renewcommand{\tilde}{\widetilde}

\renewcommand{\epsilon}{\varepsilon}

\newcommand{\tra}{{\scalebox{0.6}{$\top$}}}

\newcommand{\eps}{\varepsilon}
\newcommand{\re}{{\mathrm{e}}}

\newcommand{\ud}{{\mathrm d}}

\newcommand{\R}{{\mathbb R}}

\newcommand{\N}{{\mathbb N}}

\newcommand{\RP}{{\mathbb R}_+}

\newcommand{\as}{\ \text{a.s.}}
\newcommand{\cA}{{\mathcal A}}
\newcommand{\cB}{{\mathcal B}}

\newcommand{\cD}{{\mathcal D}}
\newcommand{\cE}{{\mathcal E}}
\newcommand{\cF}{{\mathcal F}}

\newcommand{\cH}{{\mathcal H}}

\newcommand{\cM}{{\mathcal M}}

\newcommand{\cP}{{\mathcal P}}

\newcommand{\cT}{{\mathcal T}}

\newcommand{\cZ}{{\mathcal Z}}
\newcommand{\lebm}{{\mathfrak{m}}}
\newcommand{\ocD}{\mathrm{Int}{(\mathcal D)}}
\newcommand{\aref}[1]{{\textup{\ref{#1}}}}

\newcommand{\taue}{\tau_{\mathcal{E}}}

\makeatletter
\def\namedlabel#1#2{\begingroup  
    #2%
    \def\@currentlabel{#2}%
    \phantomsection\label{#1}\endgroup
}
\makeatother

\newlist{myenumi}{enumerate}{10}
\setlist[myenumi]{leftmargin=0pt, labelindent=\parindent, listparindent=\parindent, labelwidth=0pt, itemindent=!, itemsep=1pt, parsep=4pt}

\newlist{thmenumi}{enumerate}{10}
\setlist[thmenumi]{leftmargin=0pt, labelindent=\parindent, listparindent=\parindent, labelwidth=0pt, itemindent=!}

\let\oldtocsection=\tocsection

\let\oldtocsubsection=\tocsubsection

\let\oldtocsubsubsection=\tocsubsubsection

\renewcommand{\tocsection}[2]{\hspace{0em}\oldtocsection{#1}{#2}}
\renewcommand{\tocsubsection}[2]{\hspace{1em}\oldtocsubsection{#1}{#2}}
\renewcommand{\tocsubsubsection}[2]{\hspace{2em}\oldtocsubsubsection{#1}{#2}}

\begin{document}

\title[Brownian motion with asymptotically normal reflection]{Brownian motion with asymptotically normal reflection in unbounded domains:
from  transience to stability}

\author{Miha Bre\v{s}ar} 
\address{Department of Statistics, University of Warwick, UK}
\email{miha.bresar@warwick.ac.uk}
\author{Aleksandar Mijatovi\'c} 
\address{Department of Statistics, University of Warwick, \and The Alan Turing Institute, UK}
\email{a.mijatovic@warwick.ac.uk}
\author{Andrew Wade}
\address{Department of Mathematical Sciences, Durham University, UK}
\email{andrew.wade@durham.ac.uk}

\maketitle

\begin{abstract}
We quantify the asymptotic behaviour of  multidimensional drifltess diffusions in 
 domains unbounded in a single direction, with asymptotically normal reflections from the boundary. 
We identify the critical growth/contraction rates of the domain that separate   stability, null recurrence and transience. In the stable case we prove existence and uniqueness of the invariant distribution and establish the polynomial rate of decay of its tail. We also establish matching polynomial upper and lower bounds on the rate of convergence to stationarity in total variation. 
All exponents are explicit in the model parameters that determine the asymptotics of the growth rate of the domain, the interior covariance, and the reflection vector field.

Proofs are probabilistic, and use upper and lower 
tail bounds for additive functionals up to return times to compact sets, 
for which we develop  
novel sub/supermartingale criteria, applicable to general continuous semimartingales. Narrowing domains fall outside of the standard literature, 
in part because boundary local time can accumulate arbitrarily rapidly. Establishing Feller continuity (essential for characterizing stability) thus requires an extension of the usual approach. 

Our recurrence/transience classification 
extends previous work on strictly normal reflections, and expands the range of phenomena observed across all dimensions. For all recurrent cases, we provide quantitative information through upper and lower bounds on tails of return times to compact sets 
(see~\cite{Presentation_AM} for a short \href{https://www.youtube.com/watch?v=oDDDWdbPx74}{YouTube} video describing the results).
\end{abstract}

\tableofcontents

%\medskip

\vspace{-10mm}


\smallskip


%\subjclass[2020]{60J05 (Primary); 60J65 (Secondary)}

\section{Introduction and main results}
We study the asymptotic behaviour of a multidimensional diffusion in an unbounded,
generalized parabolic domain, with asymptotically normal reflection from the boundary. Our model includes Brownian motion with normal reflection. 
We show that the phase transition between recurrence and transience occurs for asymptotically expanding domains. If the domain narrows asymptotically, we identify the phase transition between  null and positive recurrence.
In the recurrent case we characterise the asymptotic behaviour of the tails of the return times. Moreover, in the positive-recurrent case we prove the existence of the invariant distribution of the reflected diffusion and establish the polynomial rate of decay of its tail. Finally, we establish 
the polynomial rate of convergence to stationarity by proving matching 
upper and lower bounds on the total variation distance
between the marginal and the invariant distribution (see a short \href{https://www.youtube.com/watch?v=oDDDWdbPx74}{YouTube} video describing these results~\cite{Presentation_AM}).

All the aforementioned  phenomena depend on the asymptotic behaviour of the boundary and are exhibited by a normally reflected Brownian motion. In particular, this implies that a normally reflected Brownian motion in an unbounded domain of any dimension (greater than~1) may 
converge in total variation to its invariant distribution, which has heavy tails.  

Before stating our results, we briefly describe  our setting. For any $d\in\N$, define a closed domain  $\cD:= \{(x,y) \in \R_+\times\R^d:\| y\|_d \leq b(x)\}$ in $\R^{d+1}$, where 
$\|\cdot\|_d$ is the standard Euclidean norm on $\R^d$,
$\RP:=[0,\infty)$ and
$b: \R_+ \to \R_+$ is a smooth function with $b(0)=0$ and $b> 0$ on $(0,\infty)$.
Let $W = (W_t)_{t \in \RP}$ be a standard Brownian motion in $\R^{d+1}$ and $\Sigma$ a matrix-valued function on $\cD$, taking values in positive-definite square matrices  of dimension $(d+1)$.
Denote by $\Sigma^{1/2}$ the symmetric square root of $\Sigma$ and let $\phi: \partial \cD \rightarrow \R^{d+1}$ be a vector field  on $\partial \cD$.
Let 
the processes $Z = (Z_t)_{t \in [0,\taue)}$ 
and $L=(L_t)_{t\in[0,\taue)}$
with state spaces $\cD$ and $\RP$, %$:=[0,\infty)$, 
respectively, %(i.e. $Z_t \in \cD$ and $L_t\in\RP$ for $t < \taue$ a.s.) 
satisfy
the stochastic differential equation~(SDE)
\begin{equation}
\label{eq::SDE}
%\begin{aligned}
Z_t = z + \int_0^t \Sigma^{1/2}(Z_s) \ud W_s + \int_0^t \phi(Z_s) \ud L_s 
\quad\text{\&}\quad L_t = \int_0^t \mathbbm{1}\{Z_s \in \partial\cD\} \ud L_s,\quad  t \in [0,\taue),
%\end{aligned}
\end{equation}
where $\taue \in [0,\infty]$ is a possibly finite explosion time
and $L$ denotes the local time process of $Z$ at the boundary $\partial \cD$.
For any starting point $Z_0=z\in\cD$, 
by~\cite[Thm~A.1]{menshikov2022reflecting}, SDE~\eqref{eq::SDE} has a unique strong solution (with law denoted by $\P_z$)
under the assumptions in~\aref{ass:covariance1}, \aref{ass:domain1} and~\aref{ass:vector1} stated in Section~\ref{subsec::Modelling_assumptions} below. 
Informally,~\aref{ass:covariance1} requires $\Sigma$ to be bounded, Lipschitz and uniformly elliptic, \aref{ass:domain1} 
requires regularity of $b$ at zero to make the domain $\cD$ smooth,
and \aref{ass:vector1} stipulates that the smooth bounded vector field $\phi$ points into the interior of $\cD$.
Unlike in the case of oblique reflection~\cite{menshikov2022reflecting}, asymptotically normal reflection
does not exhibit explosions.  Theorem~\ref{thm:non_explosion_moments}, stated and proved below, asserts that $\taue=\infty$, $\P_z$-a.s. for all starting points $z\in\cD$,
which we assume in the remainder of the introduction. 

\subsection{The main results}
The asymptotic growth of the domain $\cD$ 
is described by the parameter 
\begin{equation}
\label{eq::beta}
\beta := \limsup_{x \rightarrow \infty}\frac{x b'(x)}{b(x)},   \end{equation}
typically  equal to the limit $\lim_{x \rightarrow \infty}x b'(x)/b(x)$
(e.g.~if  
$b(x) = ax^\beta$ for $x \geq x_0 > 0$ and some $a > 0$, or, more generally, if
$b$ is regularly varying and $b'$ is eventually monotone \cite[p.~59]{bingham1989regular}).
For local time in SDE~\eqref{eq::SDE} to influence the long-time behaviour of the reflected process $Z$, we have to assume in~\aref{ass:domain2} below 
that $\beta$  in~\eqref{eq::beta}  lies in the interval $(-\infty,1)$ (cf.~Section~\ref{subsec:heuristic} below), making the growth of $\cD$ sublinear (see Remark~\ref{rem:Ass2} below) and, possibly, asymptotically narrowing. 
Assumption~\aref{ass:covariance2} below
permits $\Sigma$ to vary smoothly with $z =(x,y)\in \cD\subset \RP\times\R^d$, but asserts 
that, as $x\to\infty$, the diagonal entry of $\Sigma$ in the $x$-direction and the sum of the remaining diagonal entries converge to positive
values $\sigma_1^2$ and $\sigma_2^2$, respectively.
Assumption~\aref{ass:covariance2} ensures that the process does not stop interacting with the boundary far away from the origin. 
%$\langle \Sigma(z)e_x,e_x\rangle \to \sigma_1^2 \in %(0,\infty)$ and
 %$\Tr \Sigma(z)-\sigma_1^2\to \sigma_2^2\in(0,\infty)$ as
%x \to \infty$,
%where $e_x$ is a unit vector in the $x$-direction in $\cD$ ($\langle\cdot,\cdot\rangle$ is the Euclidean inner product on %$\R^{d+1}$ and $\Tr$ denotes the trace). 
Assumption~\aref{ass:vector2} specifies linear factors $s_0, c_0 \in (0,\infty)$, which %modify normal reflection by 
scale the projections of the vector field $\phi$ in the $x$ and normal directions, so that, as $x\to\infty$, the former projection is
%: as $x\to\infty$, $\langle \phi(z), e_x \rangle$ is 
asymptotic to $s_0 b'(x)$ while the latter converges to 
$c_0$.
%$\langle \phi(z), n(z) \rangle \to c_0$ for every unit normal $n(z)$ at %$z\in\partial\cD$; 
Since $b'(x) \to 0$ (as $\beta < 1$), the assumption~\aref{ass:vector2} makes the vector field~$\phi$ \emph{asymptotically normal}. Strictly normal reflection has~$s_0 = c_0$.
%specifies 
%an example of $\phi$ we may take modified normal %reflection, where we multiply the reflection in the %horizontal and vertical direction, by positive constants %$s_0$ and $c_0$ respectively , \aref{ass:vector2}

% By ruling out explosions in our model, %Theorem~\ref{thm:non_explosion_moments} ushers in the %natural question of the transience/recurrence dichotomy for %the reflected process $Z$. Before giving a complete %haracterisation in Theorem~\ref{thm:rec_tran} below, we %recall the necessary definitions:
%Process $(A_t)_{t \in \RP}$, taking values in $\R^{d+1}$, 
The reflected process $Z$
is
\textit{transient} (resp.\ \textit{recurrent}) if $\lim_{t \rightarrow \infty} \|Z_t\|_{d+1} = \infty$ 
(resp.\ there exists $r_0 \in \RP$ satisfying $\liminf_{t \rightarrow \infty} \|Z_t\|_{d+1} \leq r_0$) $\P_z$-a.s.
The recurrence/transience 
transition occurs at the critical growth rate of the boundary 
\begin{equation}
\label{eq:beta_c}
    \beta_c := \frac{c_0\sigma_1^2}{s_0\sigma_2^2}.
\end{equation}

\begin{thm}
\label{thm:rec_tran}
Suppose that Assumptions~\aref{ass:domain2}, \aref{ass:covariance2}, \aref{ass:vector2} hold and the process 
$Z$ solves SDE~\eqref{eq::SDE}. Then the following statements hold for all starting points $z\in\cD$:
\begin{itemize}
    \item[(a)] if $\beta < \beta_c$, then $Z$ is recurrent;
    \item[(b)] if $ \beta > \beta_c$, then $Z$ is transient;
    \item[(c)] if $\beta = \beta_c$ and Assumptions~\aref{ass:vector2plus}, \aref{ass:domain2plus} and~\aref{ass:covariance2plus} are satisfied, then  $Z$ is recurrent.
\end{itemize}
\end{thm}

Write $Z = (X,Y) \in \cD$ in coordinates, so that $X_t \in \RP$ for $t\in\R_+$.  For any $r\in\RP$, let %Denote the return time %of $Z$ below level $r>0$ by 
\begin{equation}
\label{eq::varsigma}
\varsigma_r := \inf\{ t \in \RP: X_t \leq r\}%, \qquad\text{with $\inf \emptyset :=\infty$,}
\end{equation}
(with convention $\inf \emptyset :=\infty$) be the return time of $Z$ to $\cD\cap[0,r]\times \R^d$.
 %Recall $\beta$ and $\beta_c$ are defined  in~\eqref{eq::beta}
%and~\eqref{eq:beta_c}. 
Define
\begin{equation}
\label{eq::m0}
m_c  := \left(1-\beta/\beta_c\right)/2.
\end{equation}
Note that, if $\beta<\beta_c$, then $m_c>0$. In this case the following result
implies that the return time
$\varsigma_r$
is a.s.~finite, with $m_c$ being the critical moment exponent.

\begin{thm}
\label{thm:return_times}
Suppose that Assumptions~\aref{ass:domain2}, \aref{ass:covariance2}, \aref{ass:vector2} hold and the process 
$Z$ solves SDE~\eqref{eq::SDE}. Then the following statements hold. 
\begin{enumerate}
    \item[(a)] If $\beta \in(\beta_c,1)$, then for any level $r\in(0,\infty)$ and starting point $z\in \cD\cap(r,\infty)\times \R^d$, there is positive probability that component $X$ does not reach level $r$,
    %with positive probability, 
    i.e., 
    $\P_z(\varsigma_r = \infty) > 0$.
    \item[(b)] If $\beta \in (-\infty,\beta_c)$, then 
    for any $\eps>0$, $z=(x,y)\in\cD$ and $r\in(0,x)$, there exist constants  $C_1,C_2\in(0,\infty)$ for which
    $$
     C_1t^{-m_c-\eps}\leq \P_z(\varsigma_{r} \geq t) \leq C_2t^{-m_c+\eps}\quad \text{for all $t\in(1,\infty)$.}
    $$
\end{enumerate}
\end{thm}

We say that the process $Z$ is \textit{positive recurrent} if
$\E_z[\varsigma_r] < \infty$ holds for all $z \in \cD$ and sufficiently large $r$.
The process $Z$ is \textit{null recurrent} if it is recurrent but not positive recurrent. 
Theorem~\ref{thm:return_times}(b) implies that $Z$ is positive (resp.\ null) recurrent if $\beta < -\beta_c$
(resp.\ $-\beta_c<\beta<\beta_c$).

 By Theorem~\ref{thm:return_times}(b), positive recurrence of $Z$ occurs if 
 $\beta < -\beta_c$. In this case, we study properties of the invariant distribution of $Z$ on $\cD$.
Recall that, by~\cite[Thm~A.1]{menshikov2022reflecting}, the process $Z$ is strong Markov. An \textit{invariant distribution}
$\pi$ of $Z$ is a probability measure on the Borel $\sigma$-algebra $\mathcal{B}(\cD)$ generated by  the open subsets of $\cD$,
satisfying  $\int_{\cD}\E_z [f(Z_t)]\pi(\ud z)=\int_\cD f(z)\pi(\ud z)$ for all bounded measurable functions $f:\cD\to\RP$ and all $t\in\RP$.
A \textit{total variation} distance between two probability measures $\varphi_1$ and $\varphi_2$, defined on $\mathcal{B}$, is given by 
$\| \varphi_1-\varphi_2\|_{\mathrm{TV}}=\sup_{B\in\mathcal{B}}|\varphi_1(B)-\varphi_2(B)|$.
Define  $$M_c := -(1+\beta/\beta_c)/2,$$ 
and note that, in the case $\beta < -\beta_c$, we have $M_c>0$.

\begin{thm}
\label{thm:invariant_distributon}
Suppose~\aref{ass:domain2}, \aref{ass:covariance2}, \aref{ass:vector2} hold, 
$Z$ solves SDE~\eqref{eq::SDE}, and
$\beta < -\beta_c$.
Then the reflected process $Z$ possesses  a unique invariant distribution $\pi$. Moreover, for any $\epsilon>0$, the following statements hold:
\begin{itemize}
    \item[(a)] there exist  constants $c_\pi, C_\pi\in(0,\infty)$  such that 
\begin{equation*}
c_\pi  r^{-2M_c-\epsilon}\leq  \pi\left(\{z\in\cD:\|z\|_{d+1}\geq r\}\right)\leq C_\pi r^{-2M_c+\epsilon},\quad \text{for all $r\in[1,\infty)$;}
\end{equation*}
    \item[(b)] for any starting point $z\in \cD$ of $Z$,
there exist constants $c_{\mathrm{TV}},C_{\mathrm{TV}}\in(0,\infty)$ 
such that
\begin{equation*}
c_{\mathrm{TV}} t^{-M_c - \epsilon} \leq  \|\P_z(Z_{t} \in \cdot)-\pi\|_{\mathrm{TV}}\leq C_{\mathrm{TV}}t^{-M_c + \epsilon},\quad \text{for all  $t\in[1,\infty)$.}
\end{equation*}
\end{itemize}
\end{thm}

\begin{comment}
The \textit{total variation distance} between two probability measures $\tilde P$ and $\tilde Q$ on a $\sigma$-algebra $\tilde{ \cF}$ of the state space $\tilde \Omega$ is defined via
$$
\|\tilde P-\tilde Q\|_{TV} := \sup_{A \in \cF}\lvert \tilde P(A)-\tilde Q(A)\rvert.
$$
\end{comment}

%The following theorem gives the sub-exponential rate of convergence %in the total variation distance to the stationary measure of the %reflected process $Z$.
%Let $\cB(\cD)$ the Borel $\sigma$-algebra of $\cD \subset %\R^{d+1}$, $z \in \cD$, $A \in \cB(\cD)$, and $t\in \RP$ denote the %Markov transitional kernel by, $\cP_t(z,A) := \P(Z_t \in A \vert %Z_0 = z) =\P_z(Z_t \in A)$.

Theorem~\ref{thm:invariant_distributon} shows that, surprisingly, a reflected Brownian motion $Z$ in $\cD\subset \R^{d+1}$ with normal reflection 
(i.e., $Z$ in the class of models with $\sigma_1^2=1$, $\sigma_2^2=d$, and 
$c_0=s_0$, so that $\beta_c=1/d$)
on an unbounded domain may be polynomially ergodic if the domain narrows sufficiently fast, i.e.~$\beta<-1/d$.
In this case the tail 
$\pi\left(\{z\in\cD:\|z\|_{d+1}>r\}\right)$ 
of the invariant distribution $\pi$ decays asymptotically as $r^{1+d\beta}$ when $r\to\infty$. Note that if $\beta<-1/d$,   the domain $\cD$ has finite volume in $\R^{d+1}$. However, by modifying either the covariance matrix or the reflection vector field so that $\beta_c<1/d$, for any $\beta\in(-1/d,-\beta_c)$ we obtain polynomial stability of the reflected Brownian motion in a domain with infinite volume.  

 Theorem~\ref{thm:invariant_distributon}(a)
characterises the critical moment of the invariant distribution $\pi$ of the reflected process $Z$ in a domain $\cD$ with $\beta<-\beta_c$:
for any $\alpha$ in $[0,2M_c)$ (resp.\ $(2M_c,\infty)$), the moment $
\int_\cD \|z\|_{d+1}^\alpha \pi(\ud z)$  is finite (resp.\ infinite).
Moreover, by Remark~\ref{rem:no_mass_on_boundary} below, for every $z\in\cD$ we have $\P_z(Z_t\in\partial\cD)=0$ for Lebesgue almost every $t\in\RP$. By Theorem~\ref{thm:invariant_distributon}(b), this implies $\pi(\partial \cD)=0$.
%By Markov's inequality, the tails of $\pi$ are polynomial of order
%$$\pi\left(\{z\in\cD:\|z\|_{d+1}>r\}\right)\leq r^{-2M_c+\epsilon}C_\epsilon,\qquad\text{for all %$r\in(0,\infty)$,}$$
%where  $C_\epsilon:=\int_\cD \|z\|_{d+1}^{2M_c-\epsilon} \pi(\ud z)<\infty$.

The polynomial rate of convergence in total variation of $Z_t$ to stationarity, given in Theorem~\ref{thm:invariant_distributon}(b), is half of the rate of decay of the tail of its stationary measure. Differently put, by Theorem~\ref{thm:invariant_distributon}, we have
\begin{align*}
     \lim_{t \rightarrow \infty} \frac{\log\| \P_z(Z_{t} \in \cdot)-\pi\|_{\mathrm{TV}}}{\log t} =\frac{1}{2} \lim_{r\to\infty}\frac{\log\pi(\{z\in\cD:\|z\|_{d+1}>r\})}{\log r} = -M_c.
\end{align*}

\subsection{Discussion of the main results}
The fact that $\phi$ is an asymptotically normal reflection implies that, in the case with $\beta > 0$, the process accumulates a positive drift in the horizontal direction when it reaches the boundary. Here, we observe phase transitions between recurrence and transience depending on the model parameters. When $\beta < 0$, the process accumulates a negative drift in horizontal direction. In this case, the process is always recurrent, and in some cases the invariant distribution exists. 

We now comment on the structure of the proofs and discuss features of the theorems in the previous section. A key step in the proofs of 
Theorems~\ref{thm:rec_tran},~\ref{thm:return_times} and~\ref{thm:invariant_distributon} consists of reducing the problem to certain super/submartingale
conditions that can be verified. We stress that the processes involved (that turn out to be super/submartingales) in all non-critical cases, covering phenomena from transience to stability, are transformations of the reflected process in SDE~\eqref{eq::SDE} via a \textit{single} parametric family of Lyapunov functions. The class of functions we use are not, \textit{and should not be}, harmonic because the analysis of the return times and quantitative properties of the invariant distribution and rate of convergence require the presence of a sufficiently strong drift.

\subsubsection{Positive recurrence}
\label{subsubsec:positive_recurrence}
Theorem~\ref{thm:invariant_distributon} 
provides detailed information on the ergodicity of the reflected process $Z$, with lower bounds matching the upper bounds.  
To the best of our knowledge, this is the first  characterisation of the  rate convergence to stationarity in the context of reflecting diffusions, including those with drift. Upper bounds abound: for example~\cite{sarantsev2017reflected, Sarantsev21} provide upper bounds on the rate of convergence for various reflected diffusions with drift via drift conditions in~\cite{douc2009subgeometric} (due to the presence of the drift, the upper bounds in this case are sub-exponential). 

\begin{figure}[ht]
    \centering
    \includegraphics[width=80mm]{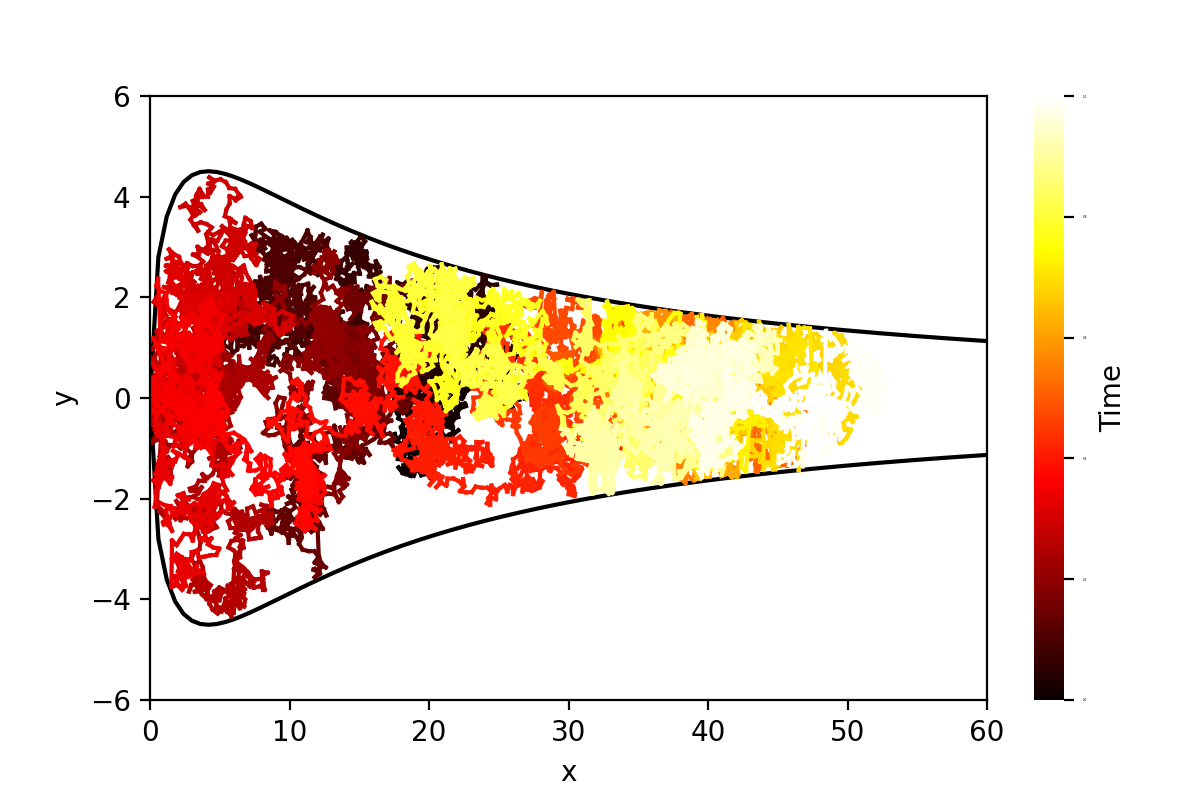}
    \caption{A positive-recurrent case ($\beta=-1.2<-1=-\beta_c$): simulation of the normally reflected Brownian motion in an unbounded domain, narrowing sufficiently fast so that (by Theorem~\ref{thm:invariant_distributon}(b)) the process converges to stationary in total variation with at the rate $t^{-0.1}$ as $t\to\infty$.}
    \label{fig:positive_rec}
\end{figure}

In contrast, the literature for lower bounds is scarce. 
Our approach to the lower bounds on the rate of convergence is purely probabilistic. It rests on 
novel continuous super/submartingale methods, based on Lemma~\ref{lem:return_times} below,
which provide a general setting where the full force of the  idea behind~\cite[Thm~5.1]{hairer2010convergence} (see Lemma~\ref{lem:lower_bound_convergence_rate} in Appendix~\ref{subsec:lower_bounds_convergence_to_stationarity} below)
can be exploited.  
The approach is robust to the underlying stochastic model and appears to be applicable to a general continuous ergodic Markov processes.

The first step in the proof of Theorem~\ref{thm:invariant_distributon}
consists of establishing the Feller continuity (Theorem~\ref{thm:Feller_continuity_of_Z} below) and 
irreducibility (Proposition~\ref{prop:marginal_equivalent_to_Lebesgue} below) for the reflected process $Z$. The key technical challenge in this step consists of controlling the growth of the local time (cf. Section~\ref{subsec:heuristic} below), which requires establishing Feller continuity of the stopped process first (see the proof of Theorem~\ref{thm:Feller_continuity_of_Z} for details).

Once Feller continuity and irreducibility of $Z$ have been established, the upper bounds in Theorem~\ref{thm:invariant_distributon} are proved 
using supermartingale conditions together with the classical convergence results in~\cite{douc2009subgeometric}, applicable in the 
subexponential case. This yields finiteness of moments of the invariant distribution $\pi$, which is then translated into upper bounds on the tails  in Theorem~\ref{thm:invariant_distributon} via Markov's  inequality. 

The lower bounds in Theorem~\ref{thm:invariant_distributon} require a lower bound on the tail of the invariant distribution $\pi$. In contrast to the upper bounds, characterising infinite moments of $\pi$
alone does not yield a lower bound on the tail of $\pi$. In order to circumvent this problem, we give a sufficient condition for $\int_\cD H \ud\pi=\infty$ for any  non-decreasing (not necessarily polynomial) function  $H$. This sufficient condition  relies on %Proposition~\ref{prop:bounded_set_excur%sions}. This criterion is established %using 
the lower bounds
on the tails of the return times to compact sets in Lemma~\ref{lem:return_times} below.
Once established, the criterion 
yields lower bounds on the tail of $\pi$ via elementary methods (see proof of Lemma~\ref{lem:lower_bound_invariant} below). The lower bound on the rate of convergence in total variation in Theorem~\ref{thm:invariant_distributon} follow from a supermartingale property of a transformed reflected diffusion in~\eqref{eq::SDE}, the lower bounds on the tail of the stationary distribution $\pi$
and a general result in~\cite{hairer2010convergence} (see Lemma~\ref{lem:lower_bound_convergence_rate} in Appendix~\ref{subsec:lower_bounds_convergence_to_stationarity} below) that converts the tails of the stationary distribution to a lower bound on the convergence rate in total variation. 

Finally we note that, in Theorem~\ref{thm:invariant_distributon}, the critical case $\beta=-\beta_c$ is omitted for brevity. Its analysis would require additional assumptions and a new Lyapunov function, analogous to the ones used in the proof of the critical case of Theorem~\ref{thm:rec_tran}(c). We expect that, under appropriate assumptions, such analysis would yield ergodicity of the normally reflected Brownian motion with logarithmic decay in any dimension.

\subsubsection{Return times}
If the reflected process is transient, return times to compact sets are not finite almost surely (we will thus discuss Theorem~\ref{thm:return_times}(a) in Section~\ref{subsubsection:rec_trans} below). In the recurrent case,
the upper bound on the tail of the return time in Theorem~\ref{thm:return_times}(b) is established via a supermartingale condition of a transformation of the reflected process, which (via~\cite{menshikov1996passage}) implies the finiteness of the moments of return times. The upper bounds in the theorem then follow by Markov's inequality. 

\begin{figure}[ht]
    \centering
    \includegraphics[width=80mm]{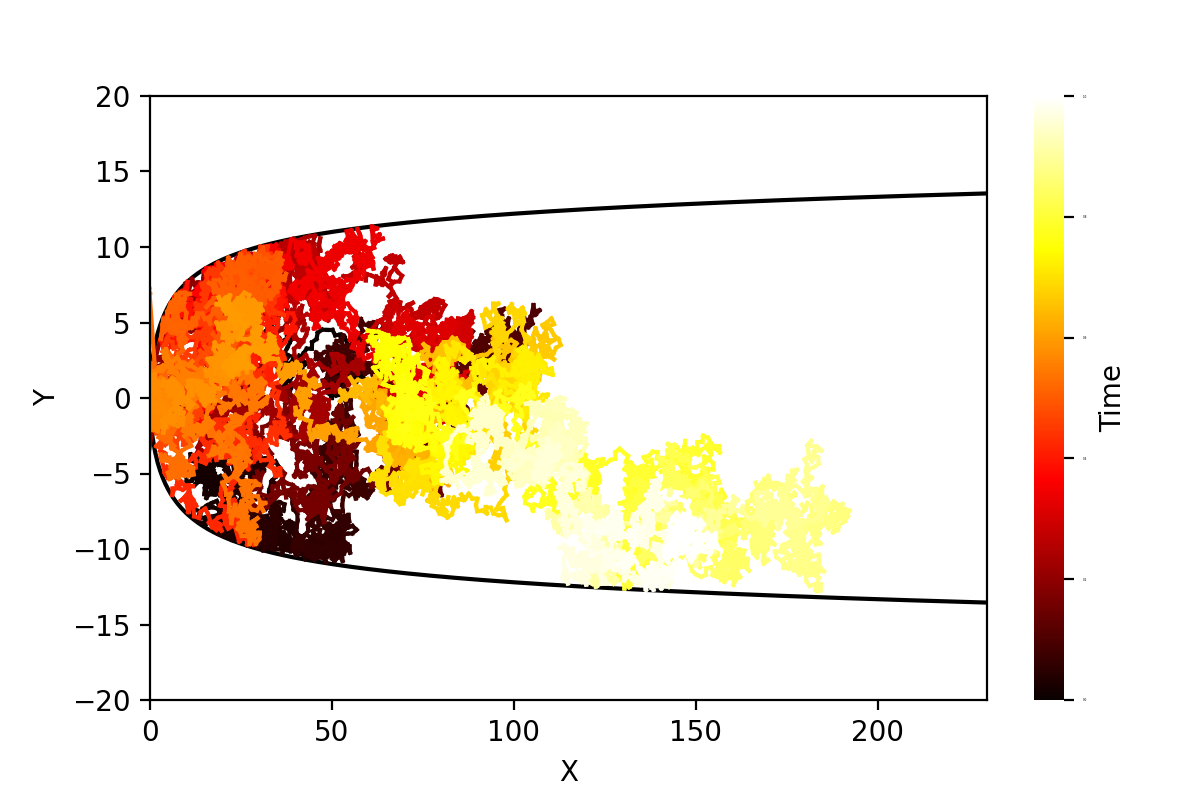}
    \caption{A null-recurrent case ($\beta=0.1<1=\beta_c$): simulation of the normally reflected Brownian motion in an unbounded expanding domain. By Theorem~\ref{thm:return_times}(b), the  tail of the return time decays with rate $t^{-0.45}$ as time $t\to\infty$, making the reflected process ``less'' recurrent than the modulus of the scalar Brownian motion.}
    \label{fig:null_rec}
\end{figure}

The lower bound in  Theorem~\ref{thm:return_times}(b) is established via certain submartingale conditions and 
lower bounds on the tails of return times to bounded sets in Lemma~\ref{lem:return_times}.
As in the proof of Theorem~\ref{thm:invariant_distributon} discussed above,
 Lemma~\ref{lem:return_times} is again critical here (as the infinite moment criterion of~\cite{menshikov1996passage} cannot be applied to obtain the lower bounds on the tail). Note that 
 Lemma~\ref{lem:return_times} is applicable in the entire recurrent regime.
 This is key in the 
 proof of Theorem~\ref{thm:return_times}(b), in contrast to the proof of Theorem~\ref{thm:invariant_distributon}, where Lemma~\ref{lem:return_times} is applied in the positive-recurrent case only. 
 %In the positive recurrent case, the %estimate the lemma provides is a speacial %case (for a constant function) of the %estimate used in the proof of the lower %bounds in %Theorem~\ref{thm:invariant_distributon}. 

\begin{comment}
 As our modelling assumptions \aref{ass:covariance2}, \aref{ass:domain2}, \aref{ass:vector2} are asymptotic (i.e., they only specify the limiting behaviour both of the coefficients of SDE~\eqref{eq::SDE}, the domain $\cD$ and the reflection vector field as $x \to \infty$, see Section~\ref{subsec::Modelling_assumptions} below for details), it is natural for Theorem~\ref{thm:return_times}(b) to assert the tail behaviour of return times for sufficiently large levels $r$ only. However, with some further work,
using an argument based on uniform ellipticity and
escape from the boundary, extending
the ideas of~\cite[\S 4]{menshikov2022reflecting} and Section~\ref{subsubsec:proof_of_thm_return_times} below
from first moment to higher moments estimates, this requirement could be removed.
\end{comment}
 As our modelling assumptions \aref{ass:covariance2}, \aref{ass:domain2}, \aref{ass:vector2} are asymptotic (i.e., they only specify the limiting behaviour both of the coefficients of SDE~\eqref{eq::SDE}, the domain $\cD$ and the reflection vector field as $x \to \infty$, see Section~\ref{subsec::Modelling_assumptions} below for details), it is natural for Theorem~\ref{thm:return_times}(b) to assert only the tail behaviour of return times, without information about the constants. However, we could provide some explicit constants, if we concentrated on the return times of the sufficiently large levels $r$
only (see Propositions~\ref{prop:return_time_upper_bound} and \ref{prop:return_time_lower_bound}).
%The bounds 
%in Theorem~\ref{thm:return_times}(b)
%on the tail of the distribution of the return time $\varsigma_r$, %defined in~\eqref{eq::varsigma}, 
%imply that for any $q\in(0,m_c)$ (resp.\  $q\in(m_c,\infty)$), there %exists $x_0\in(0,\infty)$, such that for all $r\in(x_0,\infty)$ and %$z\in\cD\cap(r,\infty)\times \R^d$ we have $\E_z[\varsigma_r^q]%%<\infty$
%(resp.\ $\E_z[\varsigma_r^q]=\infty$).
%The analysis of the critical moment exponent $q=m_c$ would require %additional assumptions, analogous to~\aref{ass:vector2plus}, %\aref{ass:domain2plus} and~\aref{ass:covariance2plus}
%in Theorem~\ref{thm:rec_tran}(c), with the corresponding moment %expected to be infinite. The analysis of the critical case is omitted %for brevity.
Note also that, by Assumption~\aref{ass:domain2}, the function $b$ is sublinear as $x \to \infty$, implying that there exist positive constants $c < C$, such that $c x \leq \|z\|_{d+1} \leq Cx$ for all $z=(x,y)\in\cD\cap (1,\infty)\times\R^d$.
Theorem~\ref{thm:return_times} may thus be restated for return times of $Z=(X,Y)$, given in terms of $\|Z\|_{d+1}$, instead of the scalar process $X$.

%\begin{rem}[Comparison with the return %times of scalar Brownian motion]
Recall that for the modulus of scalar Brownian motion, the critical exponent for return times equals $1/2$, with moments of order less (resp.\ greater) than $1/2$ being finite (resp.\ infinite).
By Theorem~\ref{thm:return_times}(b), for domains with asymptotically increasing 
(resp.\ decreasing) boundary function $b$, i.e.~$\beta > 0$ (resp.\ $\beta<0$), the critical exponent $m_c$ for the return times of the reflected process satisfies $0<m_c<1/2$ (resp.\ $1/2<m_c$). Differently put, the reflected process in an asymptotically expanding (resp.\ narrowing) domain is, due to the asymptotically positive (resp.\ negative) projection of the reflection vector field in the $x$-direction,
``less'' (resp.\ ``more'') recurrent than the modulus of the scalar Brownian motion.

In the case $\beta = 0$, the boundary function $b$ may (but need not, see Lemma~\ref{lem:oscilating_domain} below) be asymptotically constant, see the discussion in Remark~\ref{rem:Ass2} below. 
%If $\beta=0$, the component of the %$reflection in the $x$-direction decays %sufficiently fast, making the process $Z$ %recurrent with the critical moment %exponent  equal to that of the scalar %Brownian motion. Indeed, since %$\beta_c>0=\beta$, the process is %recurrent by Theorem~\ref{thm:rec_tran}(a)
%with the critical exponent for the return-
In this case, the projection of the reflection vector field converges to zero sufficiently fast that the critical moment exponent of the return time equals $m_c= 1/2$,
regardless of other model parameters.

\subsubsection{The recurrence/transience dichotomy}
\label{subsubsection:rec_trans}
Theorem~\ref{thm:rec_tran}  characterises the recurrence/transience dichotomy for the reflected process defined by the SDE in~\eqref{eq::SDE} above. 
Its proof relies on a generalisation to continuous time, given in  Section~\ref{subsec:Transience_Recurrence_criteria} below, of the classical Foster--Lyapunov criteria for transience and recurrence. We stress that our approach is purely probabilistic: our criteria are phrased in terms of continuous-time supermartingale conditions. In particular,  neither the Markov property nor any explicit knowledge of the infinitesimal characteristics of the process are required (cf.~discussion about the approach in~\cite{pinsky2009transcience} in Section~\ref{subsec:related_literature} below). Both of these features are crucial in the proof of Theorem~\ref{thm:rec_tran}.

\begin{figure}[ht]
    \centering
    \includegraphics[width=80mm]{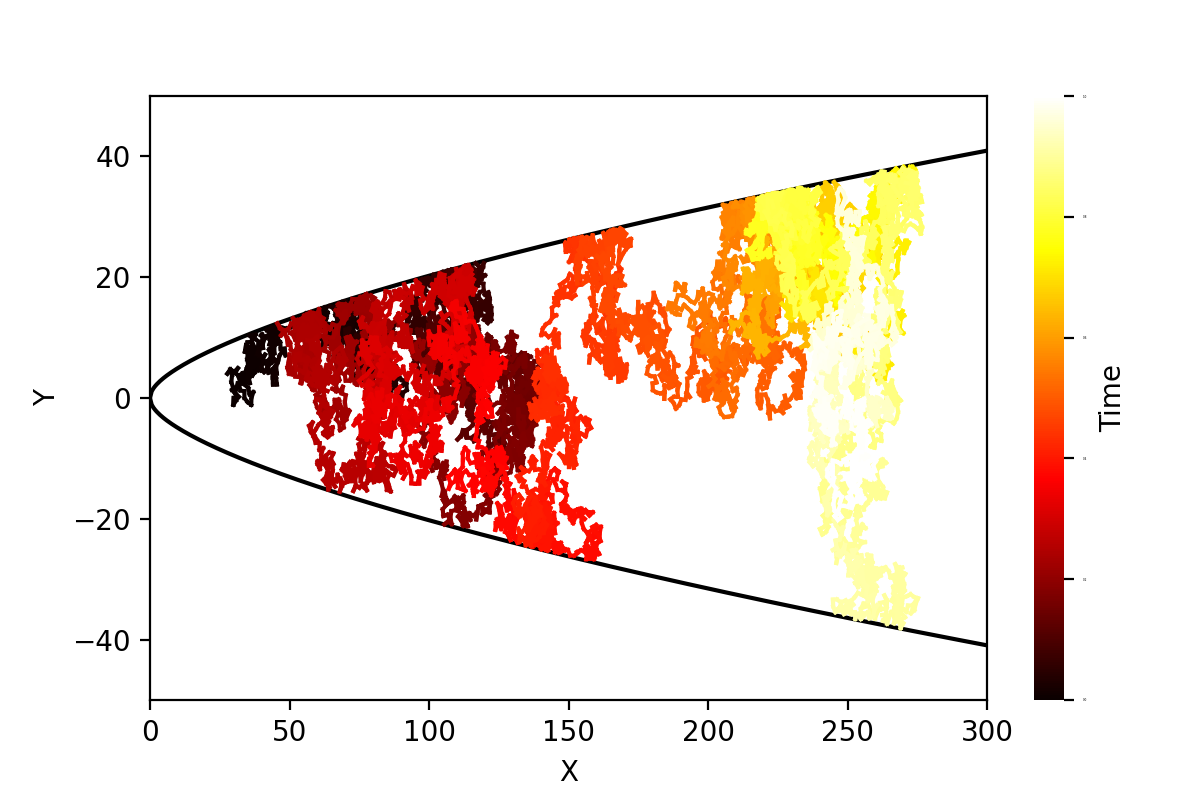}
    \caption{A transient case ($\beta=0.65>0.5=\beta_c$): the trajectory is a martingale in the interior of the domain, clearly being pushed away from the origin at boundary. The reflection is normal and $\sigma_1^2=0.5$, $\sigma_2^2=1$. By
    Theorem~\ref{thm:rec_tran}, normally reflected Brownian motion ($\sigma_1^2=\sigma_2^2=1$) in $\cD\subset \R^2$ is not transient.}
    \label{fig:transient}
\end{figure}

The critical asymptotic growth rate 
$\beta_c$ is always positive (see Assumptions~\aref{ass:covariance2} and~\aref{ass:vector2}). 
By Theorem~\ref{thm:rec_tran}, if $\beta_c \geq 1$, then $Z$ is recurrent 
for all boundary functions $b$ satisfying~\aref{ass:domain2}. 
This is for example the case for two-dimensional Brownian motion 
($\sigma_1^2=\sigma_2^2=1$ and $d=1$)
with normal reflection ($c_0=s_0$) in $\cD\subset\R^2$, as 
in this case we have 
$\beta_c = 1$.
Note that, in any dimension $d\in\N$,
the critical growth rate $\beta_c$ may be greater than one, implying recurrence for all boundary functions satisfying~\aref{ass:domain2}.

Theorem~\ref{thm:return_times}(a) strengthens transience 
of $Z$, stated in Theorem~\ref{thm:rec_tran}(b). Theorem~\ref{thm:return_times}(a) shows
that for any $r>0$ and  $z \in \cD\cap(r,\infty)\times \R^d$,
the process $Z$ does not visit the set 
$\cD\cap[0,r]\times \R^d$
with positive probability (under $\P_z$), even though $x-r>0$ can be arbitrarily small (where $z=(x,y)$). The proof of Theorem~\ref{thm:rec_tran}(b)
establishes only that, with positive probability, the process $Z$ does not return 
to $\cD\cap[0,r]\times \R^d$
 after reaching the set
 $\cD\cap(r_1,\infty)\times \R^d$
for a sufficiently large $r_1\in(r,\infty)$.
As our assumptions are asymptotic, this extension requires proving, using basic analytical techniques, that,  with positive probability, the reflected process reaches an arbitrarily high level before visiting a neighbourhood of the origin.

\subsection{A heuristic}
\label{subsec:heuristic}

Recall that $X$ denotes the $x$-component of the reflected process $Z$ in~\eqref{eq::SDE}.
An informative heuristic argument in \cite[pp.~679--680]{pinsky2009transcience}, based on the renewal theorem, 
estimates that the average local time 
%effective horizontal drift of $X$ 
accumulates as 
%via reflections to be equal to 
\begin{comment}
Suppose the process is at horizontal position~$x$.
Over short time-scales, 
imagine we may approximate the behaviour of the vertical coordinate by a diffusion on the interval $[-b(x),b(x)]$.
This diffusion has zero drift and infinitesimal variance about~$\sigma^2$.
The (vertical) reflection is effectively of magnitude~$c_0$. After a transformation, this is
equivalent to unit-magnitude reflection for Brownian motion on $[-b(x)/c_0,b(x)/c_0]$ with variance $\sigma^2/c_0^2$.
By an heuristic renewal argument similar to~\cite[pp.~679--680]{pinsky}, this process should accrue boundary local time on average at rate about $\frac{\sigma^2}{2c_0 b(x)}$.
\end{comment}
%infinitesimally the local time 
\begin{equation}
\label{eq:Pinsky_heuristic}
\ud L_t\approx 
\sigma_2^2/(2c_0b(x))\ud t,\qquad\text{when  $X_t$ is at level $x$.} 
\end{equation}
By SDE~\eqref{eq::SDE} and the definition of $\beta$ in~\eqref{eq::beta}, the total instantaneous drift of $X$
(when $X_t=x$)
is thus approximately equal to
$s_0 \sigma_2^2 b'(x)/(2c_0 b(x)) \approx s_0 \sigma_2^2 \beta/(2c_0 x)$ for large $x\in\RP$. 
Consequently, the large-scale behaviour of the horizontal coordinate $X$ of $Z$ is resembles that of the solution of the SDE
$ \ud \tilde X_t = s_0 \sigma_2^2 \beta/(2c_0 \tilde X_t)\ud t+\sigma_1\ud\tilde{W}_t$
for large times $t$ and values $\tilde X_t$,
where  $\beta\in(-\infty,1)$ and $\tilde{W}$ is a one-dimensional Brownian motion. After time-changing the SDE for $\tilde X$ 
by $t\mapsto t\sigma_1^2$, we obtain  a Bessel-type process 
%with
whose drift coefficient is determined by
the ratio $\beta/\beta_c$. For $\beta > 0$, we have a Bessel process of dimension 
$2\beta/\beta_c$ and
%In the case with $\beta > 0$, this results in, where
%$\beta_c = c_0\sigma_1^2/s_0\sigma_2^2$. 
the transition between recurrence and transience in the heuristic matches that of the result in Theorem~\ref{thm:rec_tran}. In the case $\beta<0$,  the standard literature on Langevin diffusions~\cite{fort2005subgeometric} implies that the invariant distribution exists if $\beta< -\beta_c$, the tail of invariant distribution decays at the rate $x^{\beta/\beta_c+1}$ (as $x\to\infty$) and the convergence to the invariant distribution is polynomial with the critical exponent $(1+\beta/\beta_c)/2$. Again the heuristic coincides with our results in~Theorem~\ref{thm:invariant_distributon}.

The definition of the asymptotically normal reflection requires the projection of the vector field $\phi$ in the $x$-direction to decay as $s_0 b'(x)\to0$, when $x\to\infty$. 
The heuristic in~\eqref{eq:Pinsky_heuristic} further motivates this definition: allowing a rate of decay 
of the projection of $\phi$ in the $x$-direction
different from that of~$b'(x)$ 
would miss phase transitions and other criticality phenomena.  
Moreover, heuristic~\eqref{eq:Pinsky_heuristic} and the definition of $\beta$ in~\eqref{eq::beta} suggest that, in the super-linear case  
 $\beta>1$, the drift of $X$ produced by the reflection at the boundary is asymptotically decays as $x^{-\beta}$ for large $x \in \RP$.  Such a drift is too weak to influence the long-term behaviour of the reflected process, because the process does not interact with the boundary sufficiently often. 

\begin{comment}
\begin{figure}
\includegraphics{LocalTimeGrowthPaperVersion.pdf}
\label{Fig:local_time_growth}
\caption{The growth of the local time $L$.}
\end{figure}
\end{comment}

Finally we note that the heuristic in~\eqref{eq:Pinsky_heuristic}
illustrates the difference between our domain $\cD$, satisfying~\aref{ass:domain2},
and the smooth domains studied in~\cite{stroock1971diffusion}. 
Recall from~\cite[Rem 2.7(f)]{menshikov2022reflecting} that  the domain $\cD$ with asymptotically narrowing boundary does not satisfy the assumptions of~\cite{stroock1971diffusion}. This difference is more than just a technical caveat: one of the crucial features of the domains considered in~\cite{stroock1971diffusion} is the availability of bounds, uniform in the starting point, on the increments of expected local time. In contrast, for a narrowing domain $\cD$
 (i.e., $b(x)\to 0$ when $x\to\infty$), by~\eqref{eq:Pinsky_heuristic},
an increment of local time over a short time period cannot be bounded uniformly in the starting point because its growth is proportional to $1/b(x)$.
%we expect the local time in $\cD$ to grow infinitesimally as $L_{\Delta %t} \approx \sigma_2^2/(2c_0b(x))\Delta t$ for starting point $z=(x,y)$  %and $\Delta t$ sufficiently small (see Figure~4 above). Thus, we are %lead to believe, that in cases where $\lim_{x\to\infty}b(x)\to 0$, 
We circumvent this issue via localisation, resulting in more involved proofs of fundamental properties such as the Feller continuity (see Theorem~\ref{thm:Feller_continuity_of_Z}).
Moreover, in the case of asymptotically oblique reflection, 
a narrowing domain $\cD$ may result in an explosive reflected 
Brownian motion and explosive local time~\cite[Thm~2.2(i)]{menshikov2022reflecting}, suggesting further that local time in $\cD$
can grow arbitrarily fast. The long-term growth of local time is discussed in Section~\ref{subsec:Concluding_rem} below.

\subsection{Related literature}
\label{subsec:related_literature}

The theory of reflecting diffusions began with~\cite{stroock1971diffusion}.
A large literature is dedicated to reflecting diffusions in bounded domains (see e.g.~\cite{lions1984stochastic,Burdzy17}). In the unbounded case, the classical domains are cones  (see e.g. \cite{franceschi2019integral,menshikov1996passage,williams1985recurrence,varadhan1985brownian,hobson1993recurrence}) and halfplanes~\cite{Burdzy1992,Burdzy1993}. Reflecting diffusions have been extensively studied due to their vast applications, including queueing theory~\cite{Harrison13,Peterson91,Ramanan03} and mathematical finance~\cite{Ichiba11,Banner05}.

Our domain $\cD$ in the case of (strictly)
normal reflection has been  studied in~\cite{pinsky2009transcience}.
Theorem~\ref{thm:rec_tran}  can be viewed as a generalisation
of a result in~\cite{pinsky2009transcience}, which 
considers the case where $\Sigma(z)$
equals the identity matrix and the reflection vector field is given by the unit normal on the entire boundary $\partial \cD$. In the context of our model, this setting is within the subclass $\sigma_1^2=1$, $\sigma_2^2=d$ 
and $c_0=s_0=1$ (recall that these constants specify only the limiting values of $\Sigma(z)$ and the reflection field $\phi$
on $\partial\cD$ as $x\to\infty$).

By~\cite{pinsky2009transcience}, the  $(d+1)$-dimensional Brownian motion with normal reflection has phase transitions between transience and recurrence at $\beta_c = 1/d$, cf.~\eqref{eq:beta_c}.
In particular, when dimension $d$ is large, recurrence occurs only when the boundary function grows at 
rate slower than $1/d$, i.e., very slowly.
In addition, Theorem~\ref{thm:rec_tran}(a) reveals that a $(d+1)$-dimensional reflected Brownian motion may be recurrent for a boundary function $b$ with growth close to linear if 
the projection of the reflection vector field $\phi$ in the $x$-direction decays as 
 $s_0 b'(x)\to0$, when $x\to\infty$, for a sufficiently small constant $s_0>0$ (note that the projection in the $x$-direction of the normal reflection 
decays precisely at the rate $b'(x)\to0$).
Moreover, Theorem~\ref{thm:rec_tran}(a) and the formula in~\eqref{eq:beta_c} imply that $Z$ 
with normal reflection (i.e.~$s_0=c_0$)
may be recurrent for the growth rate $\beta$ (of the boundary function $b$) arbitrarily close to one, if the instantaneous variance of $Z$ in the $x$-direction is greater than the sum of the variances in all other directions. 

 A general result in~\cite{pinsky2009transcience}, for domains satisfying existence and uniqueness conditions of~\cite{stroock1971diffusion},  states
that a $(d+1)$-dimensional Brownian motion with normal reflection is positive recurrent if and only if $\cD$ has finite volume. As in the case of Theorem~\ref{thm:rec_tran} above,
 Theorem~\ref{thm:return_times} and formula~\eqref{eq::m0}  contrast with the result of~\cite{pinsky2009transcience}. 
 It is easy to see that the right choice of parameters results in 
 positive (resp.\ null) recurrence in the domains with infinite (resp.\ finite) volume, e.g.~$\beta<-1/(s_0d)$ (resp.\ 
 $-1/(s_0d)<\beta<1/(s_0d)$) in the example in the previous paragraph.

As discussed in Section~\ref{subsec:heuristic} above, some of the technical difficulties in our paper arise due to the fact that our domain 
$\cD$ with the function $b$ decreasing to zero (e.g.~$\beta < 0$)
violates the smoothness condition of domains in~\cite{stroock1971diffusion}.
A general framework, via submartingale problems, for a large family of domains that fail to satisfy the assumptions~\cite{stroock1971diffusion} was developed in~\cite{kang2017submartingale}. Similar ideas were employed in~\cite{kang2014characterization} to characterise stationary distributions for a large family of reflecting diffusions with piecewise smooth boundaries.

The question of convergence in general domains, beyond assumptions in~\cite{stroock1971diffusion}, was studied in~\cite{Burdzy2006}. The main result of~\cite{Burdzy2006} gives the criteria for the uniform ergodicity of normally reflected planar Brownian motion. In particular,~\cite[Prop.~2.11]{Burdzy2006} yields that the process in the domain $\cD$ (with a boundary function $b$ of polynomial decay defined in~\eqref{eq::domain}) is not uniformly ergodic, a fact also implied by the lower bounds in Theorem~\ref{thm:invariant_distributon} above. Moreover, the proofs in~\cite{Burdzy2006} use analytical tools such as potential theory and conformal mappings (the latter  available in $\R^2$ only). In a domain $\cD\subset \R^2$ in~\eqref{eq::domain}, uniform ergodicity arises if the decay of the boundary function $b$ is superexponential~\cite[Prop.~2.11]{Burdzy2006}. It is feasible that the probabilistic methods developed in our paper could generalise the results of~\cite{Burdzy2006} to domains~$\cD$ of arbitrary dimensions.

Finally we note that the invariant distribution and the rate of convergence of the reflecting Brownian motion with drift have been studied in~\cite{sarantsev2017reflected,Sarantsev21}, motivated by applications in particle systems (see Section~\ref{subsubsec:positive_recurrence} above  for a brief discussion of the relation between our result and those in~\cite{sarantsev2017reflected,Sarantsev21}).

%\section{Proofs of the main results}
%\label{sec::3}
\section{Modelling assumptions and preliminary results}
\label{subsec::Modelling_assumptions}

\subsection{Modelling assumptions}
For any  $d \in \N := \{1,2,\dots\}$, let $\|\cdot\|_d$ denote the Euclidean norm on $\R^d$. 
%Let $\RP \coloneqq[0,\infty)$ and express $z \in \RP \times %\R^{d}$ in coordinates as $z = (x,y)$, where $x \in \RP$ and $y %\in \R^d$. 
Define the closed domain $\cD \subset \RP\times\R^d$, where  $\RP \coloneqq[0,\infty)$, by
\begin{equation}
\label{eq::domain}
\cD := \{ z =(x,y)\in \RP \times \R^d: \| y \|_d \leq b(x)\},\quad\text{where $b: \RP \rightarrow \RP$}
\end{equation}
is
strictly positive and differentiable on $(0,\infty)$ with $b(0)=0$.
Let 
$\partial\cD=\{z\in\cD:\|y\|_d =b(x)\}$ be the boundary of $\cD$ in $\R^{d+1}$ and denote the unit sphere in $\R^d$ by $\mathbb{S}^{d-1} := \{ u \in \R^d: \|u \|_d = 1\}$. Let $e_x := (1,0) \in \mathbb{S}^d\subset \RP\times\R^d$ denote the unit vector in the $x$-direction and, for any $u \in \mathbb{S}^{d-1}$, define $e_u := (0,u) \in \mathbb{S}^d$. 
We can express $z=(x,y)\in\cD$ %in ``polar coordinates''
as $z=xe_x+\|y\|_d e_{\hat y}$, 
where $\hat y := y / \| y \|_d\in\mathbb{S}^{d-1}$ for $\| y \|_d > 0$
(if 
$\| y \|_d =0$, we may choose $\hat y$ to be any vector in $\mathbb{S}^{d-1}$).

Recall that the functions $\Sigma: \cD \rightarrow \cM_{d+1}^+$
and $\phi:\partial\cD\to\R^{d+1}$
are the instantaneous variance and the reflection vector field at the boundary of the domain $\cD$ in SDE~\eqref{eq::SDE} above. Throughout we denote 
by 
$\cM_{d+1}^+$ 
the group of positive  definite square matrices  of dimension $(d+1)$.

%(see~\cite[Appendix A]{menshikov2022reflecting} for the %definition of space of trajectories of $Z$).  Accompanying $Z$ %will be a local time
%. The triplet   $(Z,L,\taue)$ is a solution of 

%started at any point $ Z_0 = z\in\cD$, and where the stopping time $\tau_\cE$ satisfies $\taue = \lim_{r \rightarrow \infty} \inf\{t \in \RP: \|Z_t\| \geq r \}$. We fix the triplet $(Z,L,\tau_\cE)$ and use it throughout the paper as the solution of the SDE~\eqref{eq::SDE}.
%We begin by formally stating the assumptions on functions $b$, %$\Sigma$ and $\phi$. Recall that
%\aref{ass:covariance1}, \aref{ass:domain1} and~\aref{ass:vector1} %are standard regularity assumptions, which are sufficient for the %existence and uniqueness of the solutions to %equation~\eqref{eq::SDE}. 

\begin{description}[align=left,leftmargin=2.2em,itemindent=0pt,labelsep=4pt,labelwidth=1.8em]
\item[\namedlabel{ass:domain1}{{(D1)}}]
Let $b$ be a continuous function on $\RP$, with $b(0) = 0$ and $b(x) > 0$ for $x > 0$. Suppose that $b$ is twice continuously differentiable  on $(0,\infty),$ such that (i) $ \liminf_{x \rightarrow 0} (b(x)b'(x)) > 0$, and (ii) $\lim_{x \rightarrow 0}(b''(x)/b'(x)^3)$ exists in $(-\infty,0]$.
\end{description}

 \begin{description}[align=left,leftmargin=2.2em,itemindent=0pt,labelsep=4pt,labelwidth=1.8em]
\item[\namedlabel{ass:covariance1}{{(C1)}}] 
Let $\Sigma: \cD \rightarrow \cM_{d+1}^+$ be bounded, globally Lipschitz, and uniformly elliptic, i.e., there exists $\delta_\Sigma > 0$ such that, for every $u  \in \mathbb{S}^d$ and all $z \in \cD$, we have $\langle\Sigma(z)u,u\rangle \geq \delta_\Sigma$.
 \end{description}
 
%Pick the boundary reflection vector field $\phi: \partial \cD \rightarrow %\R^{d+1}$. 
%If $z \in \partial \cD$, then $z = xe_x+b(x)e_u$ where $u=y/\|y\|_d$. 
%Denote $\phi(x,u) := \phi(z)$ and let
% $n(x,u)$ be the inwards-pointing unit normal vector for $\partial\cD$ at %$z=(x,ub(x))$. 
 
\begin{description}[align=left,leftmargin=2.2em,itemindent=0pt,labelsep=4pt,labelwidth=1.8em]
\item[\namedlabel{ass:vector1}{{(V1)}}] 
Suppose $\phi: \partial \cD \rightarrow \R^{d+1}$ is a $C^2$-vector field, satisfying $\sup_{z \in \partial \cD} \| \phi(z)\|_{d+1} < \infty$ and 
 \begin{equation*}
\inf_{x > 0}\inf_{\hat y \in \mathbb{S}^{d-1}} \langle \phi(x,b(x) \hat y),n(x,b(x) \hat y) \rangle > 0,
 \end{equation*}
 where $n(z)=n(x,b(x) \hat y)$ %with $y=b(x) \hat y$, 
 is the inwards-pointing unit normal vector at $z=(x,y)\in\partial \cD$.
 \end{description}
  
  \begin{rem}
  Assumption~\aref{ass:vector1} requires the vector field $\phi$ to be smooth, of bounded magnitude and have a uniformly positive component in the normal direction (throughout the paper, $\langle\cdot,\cdot\rangle$ denotes the standard inner product on $\R^{d+1}$).
  Assumption~\aref{ass:domain1} guarantees that the boundary $\partial \cD$ is sufficiently regular everywhere, including the origin (see~\cite[Lem.~4.3]{menshikov2022reflecting} for details). 
Assumption~\aref{ass:covariance1} ensures that the process $Z$ is not locally constrained in any direction. Throughout we use the  
%norm on 
%$\Sigma(z)\in \cM_{d+1}^+$ is bounded and Lipschitz, we use the (operator) 
matrix norm  $\|\Sigma(z)\|_{\text{op}} := \sup_{v\in \mathbb{S}^d}\|\Sigma(z) v\|_{d+1}$.
(In particular, since $\| \Sigma^{1/2}(z)\|_{\text{op}}^2 = \sup_{v \in \mathbb{S}^d}\langle\Sigma(z)v,v\rangle=\|\Sigma(z)\|_{\text{op}}$ equals the largest eigenvalue of $\Sigma$, the boundedness of $\Sigma$ implies the boundedness of $\Sigma^{1/2}$.)
  \end{rem}

Under assumptions \aref{ass:vector1}, \aref{ass:domain1} and \aref{ass:covariance1}, the process $Z$ may explode with positive probability. In fact, by~\cite[Thm~2.2]{menshikov2022reflecting}, we may have $\P(\taue<\infty)=1$. 
%there exists a strong solution $(Z,L,\tau_\cE)$ satisfying %\eqref{eq::SDE}, and there is pathwise %uniqueness~\cite[Thm~A1]{menshikov2022reflecting}. 
The following additional assumptions preclude explosion (i.e., as we shall see, imply $\P(\taue=\infty)=1$) and allow us to characterise transience and recurrence of the process $Z$;
see Theorem~\ref{thm:rec_tran} above for a detailed statement.

\begin{description}[align=left,leftmargin=2.2em,itemindent=0pt,labelsep=4pt,labelwidth=1.8em]
\item[\namedlabel{ass:domain2}{{(D2)}}] 
Suppose that~\aref{ass:domain1} holds, $\lim_{x \rightarrow \infty} b'(x) = \lim_{x \rightarrow \infty} b''(x) = 0$, the limit in~\eqref{eq::beta} exists and satisfies $\lim_{x\to\infty} xb'(x)/b(x)=\beta \in (-\infty,1)$.
 \item[\namedlabel{ass:covariance2}{{(C2)}}] 
 Suppose that \aref{ass:covariance1} holds and that there exist $\sigma_1^2,\sigma_2^2 \in (0, \infty)$ such that
\begin{align*}
 \langle \Sigma(z)e_x,e_x\rangle = \sigma_1^2(1+o_{\cD}(1)) \quad \text{and} \quad
  \Tr\Sigma(z) -\sigma_1^2 = \sigma_2^2(1+o_{\cD}(1))\quad\text{as $x\to\infty$.}
\end{align*}
\item[\namedlabel{ass:vector2}{{(V2)}}] 
Suppose that \aref{ass:vector1} holds and that there exist $s_0,c_0 \in (0,\infty)$ such that 
 \begin{align*}
  %  \label{eq::refx}
     \langle \phi(z),e_x \rangle = s_0b'(x)(1+o_{\partial\cD}(1))\quad \text{and} \quad
     \langle \phi(z),-e_{\hat y} \rangle =c_0(1+o_{\partial\cD}(1))\quad \text{as $x\to\infty$.}
 \end{align*}
\end{description}

Here and throughout,
for any $g:\RP\to(0,\infty)$
and $\cH\in\{\cD,\partial\cD\}$, 
$o_\cH(g(x))$ %(resp. $o_{\partial\cD}(g(x))$) 
as $x\to\infty$
denotes a function $f:\cH\to\R$ 
satisfying
%$f:\cD\to\R$ (respectively $f:\partial\cD\to\R$) and $g:\RP\to(0,\infty)$
%satisfy
%$f(z)=o_\cD(g(x))$ (resp. $f(z)=o_{\partial\cD}(g(x))$)
%as $x\to\infty$ 
%(where   $z=(x,y)\in\cD$)
%if 
$\lim_{x\to\infty}\sup_{y:(x,y) \in \cH}|f(x,y)|/g(x) = 0$.
%(resp. $\lim_{x\to\infty}\sup_{y:(x,y) \in \partial\cD}|f(x,y)|/g(x) = 0$).

\begin{rem}
\label{rem:Ass2}
Since Assumption \aref{ass:domain2} requires $\lim_{x\to\infty} xb'(x)/b(x)=\beta < 1$, 
for any $\beta' \in (\beta,1)$ we have
$b(x) < x^{\beta'}$ for all sufficiently large $x \in \RP$, implying that $b$ has sublinear growth as $x \to \infty$. Note however that, as $x \to \infty$, Assumption \aref{ass:domain2} allows $b$ to have any of the following limits: infinity (requiring $\beta \geq 0$), a positive finite limit (requiring $\beta = 0$) or a limit equal to $0$ (requiring $\beta \leq 0$). Interestingly, in the case $\beta=0$ the boundary function $b$  may exhibit a variety of different behaviours at infinity. For instance, $b$ may grow to infinity (e.g.~$b(x) \approx \log x$), converge to $0$ (e.g.~$b(x) \approx 1/\log x$), or be asymptotically constant. Furthermore, it is also possible for the function $b$ to oscillate, 
%greatly, 
i.e.~ $\limsup_{x \to \infty}b(x) = \infty$ and $\liminf_{x \to \infty} b(x) = 0$. For example, any function $b$ satisfying $$b(x) \approx (1 + (\log \log x)^{-2} + \sin\log \log x)\log \log x\quad\text{as $x\to\infty$,}$$
exhibits such behaviour (see Appendix~\ref{app:oscillating_b} for proof that such a function satisfies~\aref{ass:domain2}).

Assumption~\aref{ass:covariance2} ensures that the instantaneous covariance of the process $Z$ stabilises at a positive level in the $x$-direction as $x$ tends to infinity, without taking up all of the volatility of the process ($\Tr\Sigma$ denotes the sum of the diagonal elements of $\Sigma$). 
Since for the unit normal $n(x,y)$ at $(x,y)\in\partial \cD$, the inner products $\langle n(x,y),e_x\rangle$
and 
$\langle n(x,y),-e_{\hat y}\rangle$
are asymptotically equivalent to $b'(x)$ and $1$, respectively,
as $x\to\infty$, Assumption~\aref{ass:vector2} requires that the vector field $\phi$ has the same asymptotic behaviour as the unit normal $n(x,y)$ up to positive constants. As we shall see, it is precisely this property that precludes explosions of the process $Z$ and gives rise to the phenomena studied in this paper. 
\end{rem}

Theorem~\ref{thm:rec_tran} above gives rise to the critical exponent $\beta_c$,
defined in~\eqref{eq:beta_c} above,
at which the global behaviour of $Z$ transitions between recurrence and transience.  
Quantification of the limits in  Assumptions~\aref{ass:vector2}, \aref{ass:domain2} and \aref{ass:covariance2} are required to understand the behaviour of $Z$ if the boundary of the domain grows at the critical rate   
$\beta=\beta_c$ (see~\eqref{eq::beta} above for the link between the growth rate of the boundary and exponent $\beta$).

\begin{description}[align=left,leftmargin=2.2em,itemindent=0pt,labelsep=4pt,labelwidth=1.8em]
\item[\namedlabel{ass:domain2plus}{{(D2+)}}]
Assume \aref{ass:domain1} and that there exists  $\beta\in (0, 1)$  such that
\begin{equation*}
x b'(x) = \beta b(x)(1+o({b(x)^2x^{-2}})  ) \quad \text{as $x \rightarrow \infty$.} 
\end{equation*}
\item[\namedlabel{ass:covariance2plus}{{(C2+)}}]
Assume \aref{ass:covariance1} and that there exist $\sigma_1^2,\sigma_2^2 \in (0, \infty)$ and  $\epsilon > 0$ such that, as $x\to\infty$,
\begin{align*}
 \langle\Sigma(z)e_x,e_x\rangle = \sigma_1^2(1+o_{\cD}(x^{-\eps})) \quad\text{and}\quad
  \Tr\Sigma(z) - \sigma_1^2 = \sigma_2^2(1+o_{\cD}(x^{-\epsilon})).
\end{align*}
\item[\namedlabel{ass:vector2plus}{{(V2+)}}]
Assume \aref{ass:vector1} and that there exist $s_0,c_0 \in (0,\infty)$ such that, as $x\to\infty$, 
 \begin{align*}
    \label{eq::refx}
     \langle \phi(z),e_x \rangle = s_0b'(x)(1+o_{\partial\cD}(b(x)^2x^{-2})) \quad\text{and}\quad
     \langle \phi(x,y),-e_{\hat y} \rangle =c_0(1+o_{\partial\cD}(b(x)^2x^{-2})).
 \end{align*}
\end{description}

%Assumptions~\aref{ass:vector2plus}, \aref{ass:covariance2plus}, %\aref{ass:domain2plus} specify the rate of convergence in the %corresponding limits of the assumptions ~\aref{ass:vector2}, %\aref{ass:covariance2}, \aref{ass:domain2}.

\subsection{It\^o's formula for the reflected process
and Lyapunov functions}
\label{subsec:Ito_Lyapunov}
We start by noting that under Assumptions~\aref{ass:covariance1}, \aref{ass:domain1} and~\aref{ass:vector1},  by~\cite[Thm.~A.1]{menshikov2022reflecting}, SDE~\eqref{eq::SDE} has a unique strong solution $(Z,L,\taue)$ for any starting point in a generalised parabolic domain $\cD$ defined in~\eqref{eq::domain}.
In Section~\ref{subsec:non-explosion_trans_rec_proofs} we prove Theorem~\ref{thm:non_explosion_moments}
(which shows that $\taue=\infty$, a.s.) and Theorem~\ref{thm:rec_tran}. 
Sections~\ref{subsection:return_times_drift_conditions} and~\ref{subsection:main_proofs}
prove Theorems~\ref{thm:return_times} and~\ref{thm:invariant_distributon}, respectively.

A key step in each of these proofs consists of the application of It\^o's formula to an appropriate Lyapunov function. More precisely,
let $f:\cD \to \R$ denote a $C^2$-function on the open domain $\cD\setminus\partial\cD$, such that its gradient $\nabla f$ (i.e., the vector of the partial derivatives of $f$) has a continuous extension to the closed domain $\cD$ (e.g. if 
 $f$ has a $C^2$-extension to an open set in $\R^{d+1}$ containing $\cD$, which is typically the case in applications below).
By  It\^{o}’s formula~\cite[Thm.~3.3]{revuz2013continuous} we obtain
\begin{equation}
\label{eq::Ito}
    f(Z_t) = f(Z_0) + M_t 
    + \frac{1}{2}
    \int_0^t \Delta_\Sigma f(Z_s)\ud s + \int_0^t \langle \nabla f(Z_s),\phi(Z_s)\rangle \ud L_s,\quad\text{$0\leq t<\taue$,}
\end{equation}
where $\Delta_\Sigma f :=  \Tr\left(\Sigma^{1/2}H(f)\Sigma^{1/2}\right)=\Tr\left(\Sigma H(f)\right)$
%=\Tr\left(\Sigma H(f)\right)$
is the $\Sigma$-Laplacian of $f$ (recall that the Hessian matrix $H(f):\cD\to\cM_{d+1}^+$ consists of the second partial derivatives of $f$). The (scalar) process $M$ and its quadratic variation $[M]$ on the stochastic interval $[0,\taue)$ are given by 
\begin{equation}
\label{eq::Ito_QV}
M_t := \int_0^t \langle \nabla f(Z_s),\Sigma^{1/2}(Z_s)\ud W_s\rangle\>\>\&\>\>
[M]_t=\int_0^t \|\Sigma^{1/2}(Z_s)\nabla f(Z_s)\|_{d+1}^2 \ud s,\>\>\text{$0\leq t<\taue$.}
\end{equation}

The strategy of the proofs that follow consists of  applying the continuous semimartingale results of Section~\ref{sec:Trans_recurr_dichotomy} below
to the process $\kappa=f(Z)$ %and the family of stopping times $\cT$, with respect to the filtration $(\cF_t)_{t\in\RP}$ of the driving Brownian motion in SDE~\eqref{eq::SDE}, 
for suitable $C^2$-Lyapunov functions $f$.
In particular, we will 
use the representation of the quadratic variation $[M]$
in~\eqref{eq::Ito_QV} to conclude that the appropriately stopped process $M$ is a true martingale. 
%More specifically, %Subsection~\ref{subsec:Lyapunov} defines %the various $C^2$-functions $f$ we will %be using, while %Subsection~\ref{subsec:Proof_Thm_trans_r%ec} proves Theorem~\ref{thm:rec_tran}.

%In some cases we can show that $M$ is a (true) %martingale. Denote $\kappa_t = f(Z_t)$ and recall the %stopping times $\rho_{r,T}$, and $\lambda_{l,T}$ from %\eqref{eq::lambda},\eqref{eq::rho}.
%\begin{lem}
%\label{lem:local_mart}
%Assume the conditions of Theorem \ref{thm:rec_tran} and let %$f:\cD \to \R_+$ be a $C^2$ function: Then  if %$\|\nabla f(z)\|$ is bounded on $\cD$, $M_t = %\int_0^t \langle \nabla %f(Z_s),\Sigma^{1/2}(Z_s)dW_s\rangle$  is a true %martingale. Moreover, if $f(z) \to \infty$ as $x \to %\infty$, then for any stopping time $T \in \cT$, the %process $N_t =  (M_{(T+t) \wedge \rho_{r,T}}-M_T) %\mathbbm{1}\{T < \infty\}$ is a true martingale. 
%\end{lem}
%\begin{proof}
%For any right continuous semimartingale $X$, we %denote by $[X] := ([X_t])_{t \in \RP}$. The %corresponding quadratic variation process. Since, for %any $t \in \RP$, 
%$$
%[M]_t \leq \int_0^t \|\nabla %f(Z_s)\|\Sigma^{1/2}(Z_s)\|ds \leq Ct, 
%$$
%which follows from assumption~\aref{ass:covariance1}, %$M$ is a true martingale. Similarly, for any $t \in %\RP$,
%$$
%[N]_t \leq \int_T^{(T+t) \wedge} \|\nabla f(Z_s)\| %\|\Sigma^{1/2}(Z_s)\|ds \leq Ct,
%$$
%since $\|\nabla f(z)\| \leq C_r$ on compact sets $\{z %\in \cD; f(z) \leq r \}$ 
%\end{proof}

Pick $w \in \R\setminus\{0\}$.
Define
$k_w:=1+ \sup_{x \in \RP}\left(|w(1-w)|^{1/2}b(x)-x\right)$.
%for any $w \in \R\setminus\{0\}$.
Note that, under~\aref{ass:domain2}, the function $b$ has sublinear growth at infinity (see Remark~\ref{rem:Ass2} above), implying $1 \leq k_w < \infty$. 
For any $z=(x,y)\in\cD$ and parameter $\gamma \in \R$ define
\begin{equation}
\label{eq:Lyapunov_polynomial}
f_{w,1}(z):=(x+k_w)\left(1  + w(1-w)\frac{\|y \|_d^2}{2(x+k_w)^2}\right)^{1/w}\quad\&\quad f_{w,\gamma}(z):=f_{w,1}(z)^\gamma.
\end{equation}
Since, for $(x,y) \in \cD$, we have $\|y\|_d \leq b(x)$, the definition of $k_w$ implies  
\begin{equation}
\label{eq:Lyapunov_fun_ingrediant_inequlaity}
1/2\leq 1  + w(1-w)\|y \|_d^2/(2(x+k_w)^2)\leq 3/2\quad\text{ for all %$(x,y) \in \cD$ and 
$w\in\R\setminus\{0\}$.}
\end{equation}
Thus
$(x+k_w)2^{-1/\lvert w\rvert }\leq f_{w,1}(x,y)\leq (x+k_w)2^{1/\lvert w\rvert}$ for all parameter values $w\in\R\setminus\{0\}$. Moreover, 
for any $(x,y) \in \cD$ we have 
\begin{equation}
\label{eq:bound_Lyapunov_f}
    (x+k_w)^\gamma2^{-|\gamma|/\lvert w\rvert }\leq f_{w,\gamma}(x,y)\leq (x+k_w)^\gamma2^{|\gamma|/\lvert w\rvert}
    \qquad\text{for all  $\gamma\in\R$. }
\end{equation}

Note that, for $\gamma>0$
(resp.~$\gamma <0$), the function 
$f_{w,\gamma}$ tends to infinity (resp.~zero) as $x\to\infty$, 
%(resp.~$x\to0$),
making it suitable for the application of the results in Section~\ref{sec:Trans_recurr_dichotomy}.
Moreover, it is clear that (for all choices of parameters $\gamma$ and $w$) 
$f_{w,\gamma}$ is a $C^2$-function on the open domain $\cD\setminus\partial\cD$ and its gradient $\nabla f_{w,\gamma}$ has a continuous extension to the closed domain $\cD$.

The Lyapunov function $f_{w,\gamma}$ is inspired by a generalisation of a polynomial approximation of the $2$-dimensional harmonic function $h_w(z) = r^w\cos(w\theta)$ (given in polar coordinates $z=(r,\theta)$ of the plane), previously used in the analysis of the reflected processes in wedges~\cite{menshikov2021reflecting,varadhan1985brownian,menshikov1996passage}.

The following lemma provides asymptotic properties of the relevant derivatives of $f_{w,\gamma}$.

\begin{lem}
\label{lem3.1}
Let assumptions~\aref{ass:vector2}, \aref{ass:domain2} and \aref{ass:covariance2} hold and fix %consider $f_{\gamma,w}$ with 
$\gamma\in\R$ and $w\in\R\setminus\{0\}$. Then
%The following holds for the function $f_{w,\gamma}(z)$: 
\begin{equation}
\label{eq:Sigma_Laplacian_f}
\Delta_\Sigma f_{w,\gamma}(z) = \gamma f_{w,1}(z)^{\gamma-2}(\sigma_1^2(\gamma-1)+\sigma_2^2(1-w)  +o_{\cD}(1))\quad\text{as $x\to\infty$.} \quad%\text{ for all } z =(x,y)\in \cD.
\end{equation}
There exists a constant $C>0$, such that
\begin{equation}
\label{eq:grad_f_gamma_w_bound}
\|\nabla f_{w,\gamma}(z)\|_{d+1}^2 \leq C (x+k_w)^{2(\gamma-1)} \text{ for all } z=(x,y)
\in \cD.
\end{equation}
Moreover, 
if $\gamma (s_0 \beta/c_0 - 1 + w)<0$ (resp.~$>0$), then 
there exists a positive $x_0$ such that for all $z=(x,y) \in \partial \cD \cap [x_0,\infty)\times \R^d$  we have
$
\langle \nabla f_{w,\gamma}(z),\phi(z)\rangle < 0$ (resp.~$> 0$).
%) \text{ according to whether } \gamma \left(\frac{s_0 \beta}{c_0} - 1 + %w\right) < 0 \text{ or }( > 0).
%$$
\end{lem}

\begin{rem}
Note that the constants in $o_\cD(1)$, as well as $C$ and $x_0$, in Lemma~\ref{lem3.1}
depend on the values of the parameters $\gamma\in\R$ and $w\in\R\setminus\{0\}$.
\end{rem}

\begin{proof}[Proof of Lemma~\ref{lem3.1}]
Denote 
$v(x,y):=y/(x+k_w)\in\R^d$ and define the scalar
%$h(z):=1  + w(1-w)\|y \|_d^2/(2(x+k_w)^2)$.
$h(z):=1  + w(1-w) \|v(z)\|_d^2/2$.
For any $z=(x,y)\in\cD$ we have $\|y\|_{d} \leq b(x)$ 
and, by~\aref{ass:domain2} (see also Remark~\ref{rem:Ass2} above), $b(x)=o(x^{\beta'})$ as $x\to\infty$ for any $\beta'\in(\beta,1)$.
Hence $\|v(z)\|_d=o_\cD(1)$ and, for any $r\in\R$, $h(z)^r = 1 + o_{\cD}(1)$ as $x\to\infty$.
Since
$h(z)>0$ 
(by~\eqref{eq:Lyapunov_fun_ingrediant_inequlaity})
and 
$f_{w,1}(z)=(x+k_w)h(z)^{1/w}$
for any $z=(x,y)\in\cD$,  we obtain 
\begin{align}
\nonumber
\nabla f_{w,1}(z)&=h(z)^{1/w}
\left(
e_x\left(1-(1-w)\frac{\|v(z)\|_d^2}{h(z)}\right) + e_{\hat y} (1-w)\frac{\|v(z)\|_d}{h(z)}
\right)
%\begin{bmatrix}
%1-(1-w)\|y\|_d^2/\left((x+k_w)^2 h(z)\right)\\
%(1-w)y/\left((x+k_w) h(z)\right)
%\end{bmatrix}
\\
&=
(1+ o_\cD(1))\left(e_x+e_{\hat y}(1-w)\|v(z)\|_d\right)
%\begin{bmatrix}
%51\\
%(1-w)y/(x+k_w)
%\end{bmatrix}
\quad\text{as $x\to\infty$.}
\label{eq:grad_f_w_1}
\end{align}
(See the first paragraph of Section~\ref{subsec::Modelling_assumptions}
for the definition of $\hat y$, $e_x$ and $e_{\hat y}$.)
Moreover, since
$\|v(z)\|_d=o_\cD(1)$ as $x\to\infty$
and
$\nabla f_{w,\gamma}(z)=\gamma f_{w,\gamma-1}(z)\nabla f_{w,1}(z)$, 
by~\eqref{eq:bound_Lyapunov_f}
there exists a positive constant $C$ satisfying
\begin{align*}
    \| \nabla f_{w,\gamma}(z)\|_{d+1}^2 &= \gamma^2f_{w,\gamma-1}(z)^2 \left(1 + o_{\cD}(1)\right) \leq C(x+k_w)^{2(\gamma-1)}\quad\text{for all $z=(x,y)\in\cD$.}
\end{align*}

%Every point on the boundary $z=(x,y)\in \partial \cD$
%satisfies $y=b(x) \hat y$, where $\hat y=y/\|y\|_d \in \mathbb{S}^{d-1}$
%if $\|y\|_d>0$ and $\hat y$ is arbitrary in $\mathbb{S}^{d-1}$
%if $\|y\|_d=0$ (in the latter case we have $x=b(x)=0$).
Recall that $f_{w,\gamma-1}(z)>0$ by~\eqref{eq:bound_Lyapunov_f} for all $z\in\cD$.
Thus, for $z\in\partial \cD$,
the signs of 
$\langle \nabla f_{w,\gamma}(z),\phi(z)\rangle$
and
$\gamma\langle \nabla f_{w,1}(z),\phi(z)\rangle$
are equal.
Assumption~\aref{ass:vector2} implies
$\langle e_x,\phi(z)\rangle=s_0b'(x)(1+o_{\partial\cD}(1))$
and 
$\langle e_{\hat y},\phi(z)\rangle=-c_0(1 +o_{\partial\cD}(1))$ 
%for any $u\in\mathbb{S}^{d-1}$
as $x\to\infty$. 
Note that for $z=(x,y)\in\partial \cD$
we have $\|v(z)\|_d=b(x)/(x+k_w)$.
By~\eqref{eq:grad_f_w_1} we thus obtain
%$ \in \partial \cD$ with $x > x_0$ we have
%$\langle \nabla f_{w,\gamma}(z),\phi(z)\rangle$
\begin{align*}
\gamma\langle \nabla f_{w,1}(z),\phi(z)\rangle
&= \gamma (s_0b'(x) -c_0 (1-w)\|v(z)\|_d)\left(1 + o_{\partial\cD}(1)\right)\\
%\langle \nabla f_{w,\gamma}(x,\hat y b(x)),\phi(x,\hat y b(x))\rangle \\
&=\gamma  b(x) (x+k_w)^{-1}( s_0 (x+k_w) b'(x)/b(x) - c_0 (1-w)) (1+ o_{\partial\cD}(1)) \\
 &= \gamma b(x)(x+k_w)^{-1}(  s_0 \beta - c_0 (1-w) + o_{\partial\cD}(1)) (1+ o_{\partial\cD}(1))\\
 &=  c_0 b(x)(x+k_w)^{-1} \left[\gamma \left(s_0 \beta/c_0 - 1 + w \right)+ o_{\partial\cD}(1)\right],
\end{align*}
where the third equality in the display follows from the definition of $\beta$ in~\aref{ass:domain2}.
By~\aref{ass:vector2}, the model parameter $c_0$ is positive. Thus for
$z=(x,y)\in\partial \cD$ with $x$ sufficiently large,
the sign of $\gamma\langle \nabla f_{w,1}(z),\phi(z)\rangle$
equals that of $\gamma \left(s_0 \beta/c_0 - 1 + w \right)$ as claimed in the lemma.

%To estimate the $\Delta_\Sigma$, note that 
By definition we have $f_{w,\gamma}(z) = f_{w,1}(z)^\gamma$. Hence, for any $z \in \cD$, the Hessian takes the form
$$H(f_{w,\gamma})(z)=\gamma(\gamma-1)f_{w,1}(z)^{\gamma-2}\nabla f_{w,1}(z)(\nabla f_{w,1}(z))^\tra 
+\gamma f_{w,1}(z)^{\gamma-1}H(f_{w,1})(z),
$$
where $(\nabla f_{w,1}(z))^\tra $ denotes the $(d+1)$-dimensional row vector with coordinates given by the first partial derivatives of $f_{w,1}(z)$.
Since $\Delta_\Sigma f_{w,\gamma} = \Tr\left(\Sigma^{1/2}H(f_{w,\gamma})\Sigma^{1/2}\right)=\Tr\left(\Sigma H(f_{w,\gamma})\right)$, for $z = (x,y)\in\cD$ we have
\begin{align}
\label{eq:sigma_Laplacian_f_w_gamma}
\Delta_{\Sigma} f_{w,\gamma}(z) =\gamma(\gamma-1)f_{w,1}(z)^{\gamma-2} \langle\Sigma(z)\nabla f_{w,1}(z),\nabla f_{w,1}(z)\rangle
+\gamma f_{w,1}(z)^{\gamma-1} \Delta_{\Sigma} f_{w,1}(z).
\end{align}
By~\aref{ass:covariance2} and~\eqref{eq:grad_f_w_1}, we have
$\langle\Sigma(z)\nabla f_{w,1}(z),\nabla f_{w,1}(z)\rangle  =\langle \Sigma(z)e_x,e_x\rangle + o_{\cD}(1) = \sigma_1^2 (1+ o_{\cD}(1))$. 

Note that $\partial_x h(z) = -w(1-w)\|y\|_d^2(x+k_w)^{-3}$
and
$\partial_{y_i} h(z) = w(1-w)y_i (x+k_w)^{-2}$
for any $z=(x,y)\in\cD$ ($y_i$ is the $i$-th coordinate of $y\in\R^d$).
An elementary (but tedious) calculation, based on the representation $f_{w,1}(z)=(x+k_w)h(z)^{1/w}$, yields
\begin{align*}
\partial_{y_i}^2f_{w,1}(z) &=(1-w)(x+k_w)^{-1}h(z)^{1/w-1} +o_{\cD}((x+k_w)^{-1})\\
&=(1-w)(x+k_w)^{-1}(1+o_\cD(1))\qquad\text{as $x\to\infty$}
\end{align*}
for every $i\in\{1,\dots,d\}$
(recall that $h(z)^{1/w-1}=1+o_\cD(1)$).
Moreover, all other elements of the Hessian 
$H(f_{w,1})(z)$ are of order $o_{\cD}((x+k_w)^{-1})$
as $x\to\infty$.
Thus, by definition
$\Delta_{\Sigma} f_{w,1}(z)=\Tr(\Sigma(z)H(f_{w,1})(z))$ and the fact that $\Sigma$ is bounded by~\aref{ass:covariance1} (contained in~\aref{ass:covariance2})
we get
$$\Delta_{\Sigma} f_{w,1}(z)=\left(\Tr(\Sigma(z))-\langle \Sigma(z)e_x,e_x\rangle\right)(1-w)(x+k_w)^{-1}(1+o_\cD(1))\quad
\text{as $x\to\infty$.}$$
By Assumption~\aref{ass:covariance2},
it thus follows that 
$\Delta_{\Sigma} f_{w,1}(z)=\sigma_2^2(1-w)(x+k_w)^{-1}(1+o_\cD(1))$ as $x\to\infty$.
The expression in~\eqref{eq:Sigma_Laplacian_f} is now a direct consequence of~\eqref{eq:sigma_Laplacian_f_w_gamma}.
\end{proof}

By Lemma~\ref{lem3.1}, the function $f_{w,\gamma}$ 
controls the sign of the inner product
$\langle \nabla f_{w,\gamma}(z),\phi(z)\rangle$
for $z=(x,y)\in\partial\cD$
with sufficiently large $x\in\RP$.
Controlling the sign of
$\langle \nabla f_{w,\gamma}(z),\phi(z)\rangle$
for \textit{all} $z=(x,y)\in\partial\cD$
is crucial for analysing the moments of $Z_t$ at a fixed time $t$ (see the proof of Theorem~\ref{thm:rec_tran}) as well as establishing drift conditions in the case of positive recurrence (see Lemma~\ref{Lem:drift_conditions} in Section~\ref{subsection:return_times_drift_conditions} below). This requires a slight modification of the function $f_{w,\gamma}$, which we now describe.

Fix arbitrary $x_0,x_1\in(0,\infty)$, satisfying $x_0<x_1$, and define
the function $m:\RP\times\R^d\to[0,1]$ as follows: for  $z=(x,y)\in\RP\times\R^d$ let 
\begin{align*}
m(z) &:= \exp\left((x_1-x_0)^{-2}-((x_1-x_0)^{2}-(x_1-x)^2)^{-1}\right)\mathbbm{1}\{x_0<x<x_1\}+\mathbbm{1}\{x_1\leq x\}.
\end{align*}
The function $m$ is smooth, $\partial_x m(z)\geq 0$ for all $z\in\RP\times\R^d$, 
and the following holds: for any 
 $z_i=(x_i,y)\in\RP\times\R^d$, $i\in\{0,1\}$, we have
$m(z_0) = \partial_x m(z_0) = \partial_x^2 m(z_0)  = 0$ and $m(z_1) = 1$, $\partial_x m(z_1) = \partial_x^2 m(z_1) = 0$.
For any constant  $k\in(0,\infty)$, define
\begin{equation}
\label{eq:F_w,gamma}
    F_{w,\gamma}(z) := f_{w,\gamma}(z)m(z)+k(1-m(z)),\quad z\in\cD.
\end{equation}
The function $F_{w,\gamma}$ 
is clearly a $C^2$-function on the open domain $\cD\setminus\partial\cD$ and its gradient $\nabla F_{w,\gamma}$ has a continuous extension to the closed domain $\cD$
(for all  parameters $\gamma\in\R$ and $w\in\R\setminus\{0\}$). 

\begin{lem}
\label{lem:F_w,gamma}
Let Assumptions~\aref{ass:vector2}, \aref{ass:domain2} and \aref{ass:covariance2} hold and fix %consider $f_{\gamma,w}$ with 
$\gamma\in\R$ and $w\in\R\setminus\{0\}$. Then, if $\gamma(\beta s_0/c_0 -1 + w) < 0$, there exist $k\in(0,\infty)$  and $0<x_0<x_1$ such that 
$$
\langle \nabla F_{w,\gamma}(z),\phi(z)\rangle \leq 0 \quad \text{ for all } z\in\partial\cD.  
$$
\end{lem}

\begin{proof}
Note that for any constants $0<k<\infty$ and $0<x_0<x_1$, for all $z\in\cD$
we have
\begin{align}
\label{eq:grad_F}
    \langle\nabla F_{w,\gamma}(z),\phi(z) \rangle
    %\langle \nabla \left(m(z)f_{w,\gamma}(z) + M(1-m(z))\right),\phi(z)\rangle\\
    &= (f_{w,\gamma}(z)-k)\partial_x m(z)\langle e_x,\phi(z)\rangle + m(z)\langle \nabla f_{w,\gamma}(z),\phi(z)\rangle.
\end{align}
Since $\gamma(\beta s_0/c_0 -1 + w) < 0$, by Lemma~\ref{lem3.1}, there exists $x_0'>0$ such that $\langle\nabla f_{w,\gamma}(z),\phi(z)\rangle < 0$  for all $z\in\partial\cD\cap [x_0',\infty)\times\R^d$.
By Assumption~\aref{ass:vector2}, 
$\langle \phi(z), e_x\rangle = s_0 b'(x)(1+f(z))$ for all $z\in\partial\cD$,
where $f:\partial\cD\to\R$ satisfies $\sup_{y:(x,y)\in\partial\cD} |f(x,y)|\to0$ as $x\to\infty$.
Pick $x_0''\in\RP$ such that $\sup_{y:(x,y)\in\partial\cD} |f(x,y)|<1/2$ for all $x\in[x_0'',\infty)$.
%as $x\to\infty$ and thus there exists $x_0'\in\RP$
%such that $\langle \phi(z), e_x\rangle<0$ for all %$z\in\partial\cG\cap[x_0',\infty)\times \R^d$.
There are two possibilities. 

\smallskip

\noindent \underline{(I) There exists $x_0\in[\max\{x_0',x_0''\},\infty)$ such that  $|b'(x_0)|>0$.}
If 
$b'(x_0)>0$ (resp.\ $b'(x_0)<0$), by the continuity of $b'$, there exists $x_1\in(x_0,\infty)$, such that $b'(x)>0$ (resp.\ $b'(x)<0$) for all $x\in[x_0,x_1]$.
Since $s_0>0$ by Assumption~\aref{ass:vector2}, for any $z=(x,y)\in\partial \cD$ with $x\in[x_0,x_1]$, we have 
$\langle \phi(z), e_x\rangle = s_0 b'(x)(1+f(z))>s_0b'(x)/2>0$
(resp.\ $\langle \phi(z), e_x\rangle = s_0 b'(x)(1+f(z))<s_0b'(x)/2<0$).
By~\eqref{eq:bound_Lyapunov_f},  $f_{w,\gamma}$ is a positive bounded
function on the set $\cD\cap[x_0,x_1]\times\R^d$.
Thus we may pick $k$ in the interval $(\sup_{(x,y)\in\partial\cD:x_0\leq x\leq x_1} f_{w,\gamma}(x,y),\infty)$ (resp.\ $(0,\inf_{(x,y)\in\partial\cD:x_0\leq x\leq x_1} f_{w,\gamma}(x,y))$).
Since the function $m$, defined above, satisfies $\min\{\partial_x m(z),m(z)\}\geq0$  for all $z\in\cD$, $\partial_x m(z)=0$ for all $z=(x,y)$ with $x\in[x_1,\infty)$
and $\partial_x m(z)=m(z)=0$ for all $z=(x,y)$ with $x\in[0,x_0]$,
by~\eqref{eq:grad_F} it follows $\langle\nabla F_{w,\gamma}(z),\phi(z) \rangle\leq 0$ for all $z\in\partial\cD$.

%Hence, the sign of $\langle e_x,\phi(z)\rangle$ agrees with the sign of $b'(x)$ for %all $z=(x,y)\in\partial\cD$ with $x$ sufficiently large. Also recall that by %Lemma~\ref{lem3.1},  with $x$ sufficiently large.

% For an arbitrary choice of $x_0$ in the definition of $F_{w,\gamma}$, the claim %holds on  $z=(x,y)\in\cD$ with $x < x_0$, where $F_{w,\gamma}(z)$ is a constant.

%To show the result on $z=(x,y)\in\cD$ with $x_0<x<x_1$, first note that 

%Assume that for any $n\in\N$, there exists $x > n$ satisfying $|b'(x)| > 0$. In %this case, we can choose $x_0$ such that the following holds, $| b'(x_0)| > 0$, and %on $z=(x,y)\in\partial\cD$ with $x > x_0$ the sign of $\langle e_x,\phi(z)\rangle$ %coincides with the sign of $b'(x)$ as well as $\langle\nabla %f_{w,\gamma}(z),\phi(z)\rangle < 0$. By continuity of $b'$, there exists $x_1$ such %that $b'(x) > 0$ (resp.~$< 0$) holds on the interval $[x_0,x_1]$. Moreover, since %$\partial_xm(z)\geq 0$, it follows that $\langle \nabla m(z),\phi(z)\rangle = %\partial_x m(z) \langle e_x,\phi(z)\rangle \geq 0$ (resp.~$\leq 0$) holds on %$z=(x,y)\in\partial\cD$ with $x_0<x<x_1$.
%By~\eqref{eq:bound_Lyapunov_f}, we can choose $\infty>M > %\sup_{(x,y)\in\partial\cD:x_0\leq x\leq x_1} f_{w,\gamma}(x,y)$ (resp.~$0<M %<\inf_{(x,y)\in\partial\cD:x_0 \leq x \leq x_1} f_{w,\gamma}(x,y)$).
%Hence, for $z=(x,y)\in\partial\cD$ with $x_0<x<x_1$, 
%we have 

\smallskip

\noindent \underline{(II) $|b'(x)|=0$ for all $x\in[\max\{x_0',x_0''\},\infty)$.}
%We are left with the case where $b'(x)=0$ for all $x$ sufficiently large. 
Thus, for any $x_0, x_1\in[\max\{x_0',x_0''\},\infty)$ with $x_0<x_1$, 
we have $\partial_x m(z)\langle e_x,\phi(z)\rangle= \partial_x m(z)s_0 b'(x)(1+f(z))=0$ for all $z\in\cD$.
Hence, for any $k\in(0,\infty)$, by~\eqref{eq:grad_F}  we have  $\langle \nabla F_{w,\gamma}(z),\phi(z)\rangle = m(z)\langle \nabla f_{w,\gamma}(z),\phi(z)\rangle \leq 0$ for all $z\in\partial \cD$.
\end{proof}

The function  $f_{w,\gamma}$ suffices  to establish Theorem~\ref{thm:rec_tran}(a)--(b) when the asymptotic exponent $\beta$ is away from the critical value $\beta_c$ defined in~\eqref{eq:beta_c}. In the critical case $\beta=\beta_c$,
logarithmic (rather than polynomial) growth %rate 
of the Lyapunov function is required. 
%(see function $g_\delta$ defined in~\eqref{eq:g_delta} %below).
%In the  critical case $\beta=\beta_c$, the polynomial %growth of  $f_{w,\gamma}$ is too fast to classify %recurrence. We are lead to consider a Lyapunov function %with logarithmic growth, 
The function we now define for this purpose is inspired by the analysis of the reflecting Brownian motion and random walk in a wedge in~\cite{menshikov2021reflecting,menshikov1996passage}.
Pick an arbitrary constant $\delta\in(0,\infty)$ and
let $g_\delta:\cD\to (1,\infty)$ be a continuous function, twice differentiable in the interior of $\cD$, satisfying
\begin{equation}
\label{eq:g_delta}
g_\delta(z) = \log (x) -x^{-\delta} + \frac{\sigma_1^2}{\sigma_2^2}\frac{\|y\|_{d}^2}{2x^2} (1+\delta x^{-\delta})+1\quad
\text{for $z = (x,y) \in\cD$ with $x \in (\re,\infty)$,}
%\eta \frac{\|y\|_{d}^2}{2x^{\delta + 2}}, %\text{ for all } z=(x,y) \in \cD
\end{equation}
with $\sigma_1^2$ and $\sigma_2^2$ given in~\aref{ass:covariance2plus}.
Since for any $(x,y)\in\cD$ we have $\|y\|_{d} \leq b(x)$ and $b$ is sublinear by~\aref{ass:domain2plus} (cf. Remark~\ref{rem:Ass2}), there exists a positive constant $C_\delta \in (0,\infty)$ such that\begin{equation}
\label{eq:bound_Lyapunov_g}
    %\log x \leq 
    \log x \leq g_\delta(z) \leq C_\delta + \log x, \quad
    \text{ for $z=(x,y)\in\cD$ with $x > \re$.}
\end{equation}
The relevant asymptotic properties of the derivatives of $g_\delta$ are in the next lemma. 

\begin{lem}
\label{lem:g_delta}
Assume 
$\beta = \beta_c$,
where $\beta$ (resp.\ $\beta_c$) is defined in~\eqref{eq::beta} (resp.~\eqref{eq:beta_c}).
Let Assumptions~\aref{ass:vector2plus}, \aref{ass:domain2plus} and \aref{ass:covariance2plus} hold and choose $\delta\in(0, \min\{\epsilon, 1 - \beta\})$, where $\epsilon > 0$ is the rate of decay in Assumption~\aref{ass:covariance2plus}.
Then there exists an $x_0 > 0$ such that 
\begin{equation}
%\begin{aligned}
\label{eq:sigma_lap_scalar_prod_g_delta}
    \Delta_\Sigma g_{\delta} < 0\text{ on $\cD\cap[x_0,\infty)\times\R^d$} \quad
    \text{and} \quad \langle \nabla g_{\delta},\phi\rangle 
    < 0\text{ on $\partial \cD\cap[x_0,\infty)\times\R^d$.} 
%\end{aligned}
\end{equation}
%Moreover, we have
%$\sup_{z\in\cD\cap(\re,\infty)\times\R^d}\| \nabla %g_{\delta}(z) \|_{d+1}<\infty$.
%for any $z=(x,y)\in\cD$ with $x > \re$, there exists a %positive constant $C$ satisfying
%$$
%\| \nabla g_{\delta}(z) \|_{d+1}^2 \leq C.
%$$
\end{lem}

\begin{proof}
Recall that $\beta=\beta_c=c_0\sigma_1^2/(s _0\sigma_2^2)$. By~\eqref{eq:g_delta}, for $z = (x,y)\in\cD$ with $x> \re$, we have
\begin{align}
    \label{eq::g_delta_gradient}
    \nabla g_{\delta}(z)  & =  e_x\left(x^{-1}(1+\delta x^{-\delta}) -\frac{\sigma_1^2}{\sigma_2^2} \frac{\|y\|_d^2}{x^3}(1-\delta x^{-\delta}(1+\delta/2)) \right) \\
    & + e_{\hat y}\left(\frac{\sigma_1^2}{\sigma_2^2}\frac{\|y\|_d}{x^2}(1+\delta x^{-\delta} )\right);
    \nonumber
\end{align}
see the first paragraph of Section~\ref{subsec::Modelling_assumptions}
for the definition of $\hat y$, $e_x$ and $e_{\hat y}$. 

%Thus, by~\eqref{eq::g_delta_gradient}, the gradient of %$g_\delta$ is bounded as claimed in the lemma.

By~\aref{ass:vector2plus} we have
$\langle e_x,\phi(z)\rangle=s_0b'(x)(1+o_{\partial\cD}(b(x)^2x^{-2}))
%=s_0\beta b(x)x^{-1}(1+o(b(x)^2x^{-2}))(1+o_{\partial\cD}(b(x)^2x^{-2}))
=s_0\beta b(x)x^{-1}(1+o_{\partial\cD}(b(x)^2x^{-2}))$ as $x\to\infty$,
where the second equality follows from~\aref{ass:domain2plus},
and 
$\langle e_{\hat y},\phi(z)\rangle=-c_0(1+o_{\partial\cD}(b(x)^2x^{-2}))$ 
%for any $u\in\mathbb{S}^{d-1}$
as $x\to\infty$. 
Since $z = (x,y)\in\partial\cD$ satisfies $\|y\|_d =b(x)$ as $x\to\infty$, by~\eqref{eq::g_delta_gradient} we obtain
\begin{align*}
    \langle \nabla g_{\delta}(z),\phi(z)\rangle &=s_0\beta b(x)x^{-2} (1+\delta x^{-\delta})o_{\partial\cD}(b(x)^2x^{-2})
    -s_0^2\beta^2c_0^{-1} b(x)^3x^{-4}(1+o_{\partial\cD}(b(x)^2x^{-2})) \\
    &= -s_0^2c_0^{-1}\beta^2 b(x)^3x^{-4}+ o_{\partial\cD}(b(x)^3x^{-4})\quad\text{as $x\to\infty$.}
\end{align*}
Thus there exists $x_0>0$, such that $\langle \nabla g_{\delta}(z),\phi(z)\rangle<0$ for all $z=(x,y)\in\partial \cD \cap [x_0,\infty)\times\R^d$.

By Remark~\ref{rem:Ass2} above, for any $\beta'\in(\beta,1)$ it holds 
$b(x)=o(x^{\beta'})$ as $x\to\infty$, implying 
$\|y\|_d/x=o_\cD(1)$ as $x\to\infty$. 
By assumption $\delta\in(0,1-\beta)$. Thus $b(x)=o(x^{\beta+\delta})$, implying  $b(x)^2x^{-4} = o(x^{-2-\delta})$ and $b(x)x^{-3} = o(x^{-2-\delta})$ as $x\to\infty$. 
Recall from~\eqref{eq::g_delta_gradient} that,
for any $i \in \{1,\dots,d\}$, % and all large $x\in\RP$,
the $i$-th partial derivative $\partial_{y_i} g_\delta$ along the coordinate $y_i$ of $y$ equals 
$\partial_{y_i} g_\delta(z)=\beta s_0c_0^{-1}y_i x^{-2}(1+\delta x^{-\delta})$ for $x>\re$,
implying 
$\partial^2_{y_i} g_\delta(z)=\beta s_0c_0^{-1}x^{-2}(1+\delta x^{-\delta})$.
By the representation of the gradient~\eqref{eq::g_delta_gradient} and the fact $\|y\|_d\leq b(x)$ for every $z=(x,y)\in\cD$, we have $\partial_x^2 g_\delta(z) = -x^{-2}(1+\delta(1+\delta)x^{-\delta})+o_\cD(x^{-2-\delta})$, while 
all mixed derivatives in the Hessian $H(g_\delta)(z)$ are of order $b(x)x^{-3} = o_\cD(x^{-2-\delta})$ as $x \to \infty$.
Thus, by the definition
$\Delta_{\Sigma} g_\delta(z)=\Tr(\Sigma(z)H(g_\delta)(z))$ and the fact that
$\partial^2_{y_i} g_\delta$
does not depend on the index $i \in \{1,\dots,d\}$ and is bounded for $x>\re$,
we get
%$\Sigma$ is bounded by~\aref{ass:covariance1} (contained %in~\aref{ass:covariance2plus}), and the assumption~\aref{ass:covariance2plus} 
%we get
\begin{align*}
\Delta_{\Sigma} g_\delta(z) =& \langle \Sigma(z)e_x,e_x\rangle\partial_x^2 g_\delta(z)  +\left(\Tr(\Sigma(z))-\langle \Sigma(z)e_x,e_x\rangle\right) \partial^2_{y_1} g_\delta(z) 
%-\sum_{i=1}^d\langle \Sigma(z)e_{y_i},e_{y_i}\rangle\partial_{y_i}^2g_\delta(z) 
+o_{\cD}(x^{-2-\delta}) \\
=& -x^{-2} (1+\delta(1+\delta)x^{-\delta}) (\sigma_1^2 + o_{\cD}(x^{-\eps}))
+ (\sigma_2^2+o_{\cD}(x^{-\eps}))\beta s_0c_0^{-1} x^{-2}(1+\delta x^{-\delta})\\
 & + o_{\cD}(x^{-2-\delta}) 
=  -\sigma_1^2\delta^2x^{-2-\delta}  +o_{\cD}(x^{-2-\delta})\quad\text{as $x\to\infty$,}
\end{align*}
where the last equality follows form the identity  $s_0\beta\sigma_2^2/c_0=\sigma_1^2$ and 
the fact $\delta<\epsilon$. Hence, the sign of $\Delta_{\Sigma} g_\delta(z) $ is negative for $z\in(x,y)\in\cD$ with $x$ sufficiently large as claimed in lemma.
\end{proof}

\section{Non-explosion, recurrence/transience criteria, and return times of continuous semimartingales}
\label{sec:Trans_recurr_dichotomy}

This section develops certain semimartingale tools for classifying asymptotic behaviour via Foster--Lyapunov criteria.
The general theory developed in this section is expected to have broad applicability. In the present paper, it will be applied to study the reflected process $Z$, given by SDE~\eqref{eq::SDE}, via the Lyapunov functions constructed and analysed in Section~\ref{subsec:Ito_Lyapunov} above.

Fix a probability space $(\Omega,\cF,\P)$ and a filtration $(\cF_t)_{t \in \RP}$ satisfying the usual conditions. Consider an $(\cF_t)$-adapted  continuous process $\kappa = (\kappa_t)_{t \in \RP}$, taking values in $[0,\infty]$.
%where reaching the state $\infty$ in finite time represents explosion. 
Let $\cT$ denote the set of all $[0,\infty]$-valued stopping times with respect to $(\cF_t)_{t\in \RP}$.
For any $\ell,r \in \RP$ and stopping time $T \in \cT$, define the first entry times (after $T$) by
\begin{align}
\label{eq::lambda}
    \lambda_{\ell,T} :=  T + \inf\{s \in \RP: T < \infty,~ \kappa_{T+s} \leq \ell\}, \\
    \rho_{r,T} :=  T + \inf\{s \in \RP: T < \infty, ~\kappa_{T+s} \geq r\},
    \label{eq::rho}
\end{align}
where we adopt the convention $\inf \emptyset := +\infty$. If $T = 0$, we write $\lambda_\ell := \lambda_{\ell,0}$ and $\rho_r := \rho_{r,0}$. Almost sure limits $\rho_{\infty} := \lim_{r \rightarrow \infty}\rho_r$ and $\rho_{\infty,T} = \lim_{r \rightarrow \infty}\rho_{r,T}$ exist by monotonicity. 
\textit{Explosion} of the process $\kappa$ occurs
if the event $\{\rho_\infty < \infty\}$ has positive  probability. 
Since $\rho_{r,T} = \rho_r$ on the event $\{T < \rho_r \}$, we have $\rho_\infty = \rho_{\infty,T}$ on the event $\{ T < \rho_\infty\}$.  
For $r_0\leq r$,
we define the first exit time from the interval $[r_0,r]$ after some stopping time $T \in \cT$ by 
\begin{equation}
\label{eq::exit}
S_{r,T} := \lambda_{r_0,T} \wedge \rho_{r,T}.
\end{equation}
Here and throughout we denote $x\wedge y:=\min\{x,y\}$ and 
$x\vee y:=\max\{x,y\}$
for any $x,y\in[0,\infty]$.

\subsection{Non-explosion}
We first establish criteria for $\kappa$ not to explode.
%i.e. $\P(\rho_\infty<
%\infty)=0$.
The main application of this result in the present paper is to prove that the reflected process with asymptotically normal reflection cannot explode. 

This should be contrasted with the case of the asymptotically oblique reflection, where explosion may occur, see the characterisation in~\cite[Thm~2.2]{menshikov2022reflecting}.
The non-explosion criteria in~\cite[Thm~3.4]{menshikov2022reflecting} are more delicate than the ones in Lemma~\ref{lem2.1}, but require transience of the underlying semimartingale, making them inapplicable to the entire class of processes considered here. The following result is more robust (i.e.~with simpler assumptions), has an elementary proof and covers all the models analysed in this paper. 

\begin{lem}
\label{lem2.1}
Let $\kappa  = (\kappa_t)_{t \in \RP}$ be an $[0,\infty]$-valued $(\cF_t)$-adapted continuous process 
%with $X_0 = x$ for some $x \in \R^d$. 
and $V: \RP \rightarrow (0,\infty)$ a continuous function with $\lim_{x \rightarrow \infty} V(x) = \infty$. Suppose there exist $r_0,\eta \in \RP$, such that for all $r \in (r_0,\infty)$ and any  $T \in \cT$, such that $\E[V(\kappa_T)\mathbbm{1}\{T<\rho_\infty\}]<\infty$, the process $\zeta^{T,r} = (\zeta^{T,r}_t)_{t \in \RP}$,
defined by
\begin{equation}
\label{eq::zeta}
\zeta^{T,r}_t := (V(\kappa_{(t+T)\wedge S_{r,T}})-\eta(t \wedge (S_{r,T}-T)))\mathbbm{1}\{T < \rho_{\infty}\},
\end{equation}
is an $(\cF_{t+T})$-supermartingale, i.e., for $0 \leq s \leq t < \infty$, 
\begin{equation}
\label{eq::exp}
    \E[\zeta^{T,r}_t\vert \cF_{s + T}] \leq \zeta^{T,r}_s , \as
\end{equation}
Then $\P(\rho_\infty = \infty) = 1$.
\begin{comment}
Furthermore, if for the (possibly random) starting value $\kappa_0$, we assume $\E[V(\kappa_0)]<\infty$ then $\E[V(\kappa_t)] < \infty$ for all $t\in\RP$.
\end{comment}
\end{lem}

\begin{proof}
Choose $r_1\in(r_0,\infty)$ %such that $V(r_0) < V(r_1)$, 
and define recursively the upcrossing and downcrossing times of the process $\kappa$ over the interval $[r_0,r_1]$ as follows: $\overline{\theta}_1 := 0$, and if $\overline \theta_k$ has been defined for some $k\in\N$, then 
$\underline{\theta}_k := \lambda_{r_0,\overline\theta_k}$ and  $\overline\theta_{k+1} := \rho_{r_1,\underline\theta_k}$. Thus we have $0=\overline{\theta}_1\leq \underline{\theta}_1\leq \cdots \leq \overline{\theta}_k \leq \underline{\theta}_k \leq \overline\theta_{k+1} \leq \cdots$.
Moreover, for any $t \in \RP$, we denote $$
D(t,r_0,r_1) :=\sup\{ k \in \N;~ \underline\theta_k \leq t\},
$$ 
the number of downcrossings of the interval $[r_0,r_1]$ up to time $t$ for the process $\kappa$. Since any continuous function on the compact interval $[0,t]$ %\rightarrow [0,\infty]$ 
crosses an interval of positive length at most finitely many times, we have $D(t,r_0,r_1) < \infty$ a.s.%~\cite[Ch II.]{revuz2013continuous}.

Assume  now $\P(\rho_\infty < \infty) > 0$. Then there exists $t_0 \in \RP$ such that $\P(\rho_\infty < t_0) > 0$.  We will prove by induction that $\{\overline\theta_k<\rho_\infty\}\cap\{ \rho_\infty < t_0\}=\{ \rho_\infty < t_0\}$ a.s.~holds for all $k \in \N$.
The induction hypothesis holds for $k=1$ since
$\overline\theta_1=0 < \rho_\infty$ a.s.
Assume the almost-sure equality of events holds for some $k\in\N$.
%On the event $\{ \rho_\infty < t_0\}$, we may assume %that for some $k \in \N$, $\overline\theta_k < %\rho_\infty$ a.s. since  clearly holds. 
Set $T := \overline\theta_k$ and note 
$V(\kappa_T)\mathbbm{1}\{T<\rho_\infty\}= V(r_1)\mathbbm{1}\{T<\rho_\infty\}$, since the paths of $\kappa$ are continuous and, on the event $\{T<\rho_\infty\}$, we have $T<\infty$. In particular, $V(\kappa_T)\mathbbm{1}\{T<\rho_\infty\}$ is bounded and hence integrable.
%From the supermartingale property of the process %$\zeta^{T,r}$ and the fact that on the event $\{ %T < \rho_\infty < t_0\}$, the process %$(\kappa_{t+T})_{t \in \RP}$ exits the interval %$[r_0,r_1]$ before the time $t_0-T$, it follows %that
%\begin{align*}
%&\P(\lambda_{r_0,T} < \rho_{r,T} \vert \{ T < %\rho_\infty < t_0\})f(r_0) + \P(\rho_{r,T} < %\lambda_{r_0,T}\vert \{T < \rho_\infty < %t_0\})f(r) - \eta t_0 \\
%&\leq \E[\zeta_{t_0-T}^{T,r}\vert \{T < %\rho_\infty < t_0\}] \leq \E[\zeta^{T,r}_0\vert %\{ T < \rho_\infty < t_0\}] = E[f(\kappa_T)\vert %\{ T < \rho_\infty < t_0\}] < \infty.
%\end{align*}
Pick any $r \in (r_1,\infty)$ and consider 
 the supermartingale $\zeta^{T,r}$ defined in~\eqref{eq::zeta}.
 Note that $(T\vee t_0)-T$ is a bounded ($\cF_{T+t}$)-stopping time since,
 for any $s\in\RP$, we have $\{T \vee t_0 -T \leq s\} = \{t_0 \vee T\leq  T+s \}\in\cF_{T+s}$ as both $t_0\vee T$ and $s+T$ are $(\cF_t)$-stopping times.
 Applying the optional sampling theorem to $\zeta^{T,r}$
 at $(T\vee t_0)-T$ yields:
\begin{align*}
\P(\rho_{r,T}<\lambda_{r_0,T}\wedge (T\vee t_0), T<\rho_\infty \vert\cF_{T}) V(r)- \eta t_0 & \leq 
\E[\zeta_{(T\vee t_0)-T}^{T,r}\vert \cF_{T}] \\
& \leq \zeta^{T,r}_0 = V(r_1)\mathbbm{1}\{T<\rho_\infty\}.
\end{align*}
Taking expectations on both sides, we obtain the following inequality for every $r\in(r_1,\infty)$:
$$
\P(\rho_{r,T}<\lambda_{r_0,T}\wedge (T\vee t_0), T<\rho_\infty)\leq  (V(r_1)+ \eta t_0)/V(r).
$$
Thus, by the monotone convergence theorem, we obtain
\begin{align}
\nonumber
0\leq \P(\rho_\infty\leq\lambda_{r_0,T}\wedge (T\vee t_0),T<\rho_\infty ) & =
\lim_{r\to\infty}\P(\rho_{r,T}<\lambda_{r_0,T}\wedge (T\vee t_0),T<\rho_\infty ) \\
& \leq \limsup_{r\to\infty}(V(r_1)+t_0\eta)/V(r)=0,
\label{eq::prob_zero}
\end{align}
implying
$\{T<\rho_\infty \}=%\{\rho_\infty\leq\lambda_{r_0,T}\wedge (T\vee t_0),T<\rho_\infty\}\cup
\{\rho_\infty>\lambda_{r_0,T}\wedge (T\vee t_0)\}\cap \{T<\rho_\infty\}$
a.s.
We hence %on the event $\{T<\rho_\infty<t_0\}$
obtain 
\begin{align*}
\{ \rho_\infty < t_0\} & =\{T<\rho_\infty\}\cap\{\rho_\infty<t_0\}
=
\{\rho_\infty>\lambda_{r_0,T}\wedge (T\vee t_0)\}\cap\{T<\rho_\infty\}\cap \{\rho_\infty<t_0\}\\
& =
\{\lambda_{r_0,T}<\rho_\infty\}\cap\{\rho_\infty<t_0\}\text{ a.s.}
\end{align*}
The first equality in this display holds by the induction hypothesis, the second holds by~\eqref{eq::prob_zero},  and the third equality follows from the fact that $T\leq \lambda_{r_0,T}$ by definition~\eqref{eq::lambda}.
%Thus, since $\underline\theta_{k}=\lambda_{r_0,T%}$ and $T=\overline\theta_{k}$, we have 
%$\P(\rho_\infty<
%\underline\theta_{k}\wedge (T\vee %t_0)\vert\cF_{T})=\P(\rho_\infty<
%\overline\theta_{k}\wedge (T\vee %t_0)\vert\cF_{T})$
%By assumption on $f$ we get that  $\P(\rho_{r,T} %< \lambda_{r_0,T}\vert \{ T < \rho_\infty < %t_0\}) \rightarrow 0$ as $r \rightarrow \infty$, %and since  on the event $\{T < \rho_\infty \}$, %for every $n \in \N$, $\rho_{n,T} \leq %\rho_\infty$, a.s. 
Since the 
%following 
equality %holds
$\{\underline\theta_k = \lambda_{r_0,T} < \rho_\infty\}=\{\rho_{r_1,\underline\theta_k} =\overline\theta_{k+1} < \rho_\infty\}$ holds almost surely,
%This implies that $\underline\theta_k = %\lambda_{r_0,T} < \rho_\infty$, a.s., which further %implies $\rho_{r_1,\underline\theta_k} %=\overline\theta_{k+1} < \rho_\infty$, a.s. 
we proved that $\{\overline\theta_{k+1}<\rho_\infty\}\cap\{ \rho_\infty < t_0\}=\{ \rho_\infty < t_0\}$ almost surely,
thus verifying the induction step.

We conclude that 
$$\{\overline\theta_1<\ldots<\overline\theta_k<\rho_\infty<t_0\}= \bigcap_{i=1}^k\{\overline\theta_i < \rho_\infty < t_0\} = \{\rho_\infty < t_0\} \quad\text{a.s.~for all $k\in\N$.}$$  
Since $\{\overline\theta_i<\overline\theta_{i+1}\}=\{\overline\theta_i<\underline \theta_i<\overline\theta_{i+1}\}$ for every $i\in\N$ and
$D(t_0,r_0,r_1) < \infty$ a.s., it follows that $\P(\rho_\infty < t_0) = 0$, contradicting our assumption $\P(\rho_\infty<\infty)>0$.
\end{proof}

\subsection{Transience and recurrence criteria for continuous semimartingales}
\label{subsec:Transience_Recurrence_criteria}
Lemmas~\ref{lem2.2} and~\ref{lem2.3} of the present subsection  provide sufficient conditions for recurrence and transience, respectively, for a continuous semimartingale  $\kappa$. They are continuous-time analogues to the Foster-Lyapunov criteria for discrete-time processes discussed in  e.g.~\cite{menshikov2016non}.

\begin{lem}
\label{lem2.2}
Let
$V: \RP \rightarrow (0,\infty)$ be a continuous function with $\lim_{x \rightarrow \infty} V(x) = \infty$ and
 $\kappa  = (\kappa_t)_{t \in \RP}$ an $\RP$-valued $(\cF_t)$-adapted continuous process satisfying $\limsup_{t \rightarrow \infty} \kappa_t  = \infty$ a.s. If there exists such $r_0 > 0$, such that for all $t_0 \in \RP$ and $r \in (r_0,\infty)$, the process $(V(\kappa_{(t+t_0)\wedge S_{r,t_0}}))_{t \in \RP}$ is an $(\cF_{t+t_0})$-supermartingale (recall that $r_0$ features in $S_{r,t_0}=\lambda_{r_0,t_0} \wedge \rho_{r,t_0}$ by definition~\eqref{eq::exit}), i.e., 
 $\E[V(\kappa_{t_0})]<\infty$ and
 for $0 \leq s \leq t < \infty,$
\begin{equation*}
\label{eq::rec}
    \E[V(\kappa_{(t + t_0) \wedge S_{r,t_0}})\vert \cF_{s + t_0}] \leq V(\kappa_{(s + t_0) \wedge S_{r,t_0}}),
\end{equation*}
then $\P(\liminf_{t \rightarrow \infty}  \kappa_t  \leq r_0)  =  1$.
\end{lem}

\begin{rem}
Up to requiring verification over a smaller class of stopping times, the hypotheses of Lemma~\ref{lem2.2}
essentially imply those of Lemma~\ref{lem2.1} with $\eta=0$. Thus, a Lyapunov function that implies recurrence will also often yield non-explosion. In the case of a transient process, however, $\eta>0$ is typically needed for Lemma~\ref{lem2.1} to be applicable.
\end{rem}

\begin{proof}[Proof of Lemma~\ref{lem2.2}.]
%The supermartingale property in~\eqref{eq::rec} implies  $\E[f(\kappa_t)] < %\infty$ for all $t \in \RP$. 
Pick $t_0 \in \RP$, $r \in (r_0,\infty)$ and consider the $(\cF_{t+t_0})$-supermartingale $ (V(\kappa_{(t+t_0)\wedge S_{r,t_0}}))_{t\in\RP}$.
%By assumptions $\zeta^{t_0,r}$ is a non-negative %$(\cF_{t+t_0})$-supermartingale. 
Note that the assumption $\limsup_{t \rightarrow \infty} \kappa_t = \infty$ a.s. implies $S_{r,t_0} < \infty$ a.s. and hence 
%the following equality holds
it holds that $\lim_{t \rightarrow \infty}V(\kappa_{(t+t_0) \wedge S_{r,t_0}}) = V(\kappa_{S_{r,t_0}})$ a.s. 
%Since $f$ is bounded from below and $\zeta^{r,0}$ is a
The supermartingale property, Fatou's lemma, definition~\eqref{eq::exit} and the continuity of $\kappa$ imply
\begin{eqnarray*}
\infty>\E[V(\kappa_{t_0})]  &\geq & \liminf_{t \rightarrow \infty}\E[V(\kappa_{(t + t_0) \wedge S_{r,t_0}})] \geq  \E[\liminf_{t \rightarrow \infty}V(\kappa_{(t + t_0) \wedge S_{r,t_0}})] \\
& = & \E[V(\kappa_{S_{r,t_0}})] \geq \P(\rho_{r,t_0} < \lambda_{r_0,t_0}) V(r).
\end{eqnarray*}
Thus, for all $r\in(r_0,\infty)$, we obtain
$$
\P\Bigl(\inf_{t \geq t_0}  \kappa_t \leq r_0\Bigr) \geq \P(\lambda_{r_0,t_0} < \infty) \geq \P(\lambda_{r_0,t_0} <\rho_{r,t_0}) \geq 1 - \E[V(\kappa_{t_0})]/V(r).
$$
Since  $\lim_{r \rightarrow\infty}V(r) =\infty$ by assumption, it follows that $\P(\inf_{t \geq t_0}  \kappa_t \leq r_0) = 1$ for any fixed $t_0 \in \RP$, implying $\P(\liminf_{t \rightarrow \infty}  \kappa_t  \leq r_0)=\P(\cap_{t_0\in\N}\{\inf_{t \geq t_0}  \kappa_t \leq r_0\})=1$.
\end{proof}

%A sufficient condition for transience of a %continuous-time process with continuous paths is given %in Lemma~\ref{lem2.3}.
%The corresponding discrete-time result can be found %in~\cite[Theorem 3.5.6]{menshikov2016non}.

\begin{lem}
\label{lem2.3}
Let
$V: \RP \rightarrow (0,\infty)$ be a continuous function with $\lim_{x \rightarrow \infty} V(x) = 0$ and
 $\kappa  = (\kappa_t)_{t \in \RP}$ an $\RP$-valued $(\cF_t)$-adapted continuous process satisfying $\limsup_{t \rightarrow \infty} \kappa_t  = \infty$ a.s.
If there exists $r_0 > 0$, such that for all $T \in \cT$, satisfying  $T < \infty$ a.s., and $r \in (r_0,\infty)$, the process $(V(\kappa_{(t+T)\wedge S_{r,T}}))_{t \in \RP}$, is an $(\cF_{t+T})$-supermartingale, i.e., for $0 \leq s \leq t < \infty,$
\begin{equation}
\label{eq::tran}
    \E[V(\kappa_{(t + T) \wedge S_{r,T}})\vert \cF_{s + T}] \leq V(\kappa_{(s + T) \wedge S_{r,T}}),
\end{equation}
then 
% $\P_{r_1}(\lambda_l = \infty\vert \cF_) > 0$ for all $l \geq r_0$ and $r_1 > r$ for $r$ sufficiently large and 
$\P(\lim_{t \rightarrow \infty} \kappa_t = \infty) = 1$.
\end{lem}

\begin{rem}
\label{rem:trans_integrability}
Note that the function $V$ in Lemma~\ref{lem2.3} is assumed to be continuous and have limit zero at infinity, making it bounded. Thus, for any $T \in \cT$, satisfying  $T < \infty$ a.s., we have $\E[V(\kappa_T)]<\infty$. Once we have such a candidate function $V$, in order to apply Lemma~\ref{lem2.3}, we only need to check the $(\cF_{t+T})$-supermartingale property in~\eqref{eq::tran}.
\end{rem}

\begin{proof}
%Since $V$ is bounded, $\E[V(\kappa_t)] < \infty$ for all $t \in \RP$. 
Pick an arbitrary $\ell\in(r_0,\infty)$. For any
$r, r_1\in \RP$, satisfying $\ell<r_1<r$,  define $T := \rho_{r_1}$. 
Since $\limsup_{t \rightarrow \infty} \kappa_t = \infty$ a.s., we have $T = \rho_{r_1}<\infty$ a.s. and $\rho_{r,T}<\infty$ a.s., implying further by definition~\eqref{eq::exit} that 
$S_{r,T}=\lambda_{r_0,T} \wedge \rho_{r,T}\leq \rho_{r,T}<\infty$ a.s.

Define the process $\xi = (\xi_t)_{t \in \RP}$ by 
$\xi_t:= V(\kappa_{(t + T)\wedge \lambda_{\ell,T} \wedge \rho_{r,T}})$,
$t \in \RP$.
%As $\xi_0=V(\kappa_T)$ and, 
Since $\lambda_{\ell,T} \wedge \rho_{r,T}\leq S_{r,T}$ a.s., 
the process $\xi$ equals the $(\cF_{t+T})$-supermartingale 
$(V(\kappa_{(t+T)\wedge S_{r,T}}))_{t \in \RP}$ stopped 
at the $(\cF_{t+T})$-stopping time 
$\lambda_{\ell,T} \wedge \rho_{r,T}-T$.
Thus, by Remark~\ref{rem:trans_integrability} and~\cite[Ch. II, Thm 3.3]{revuz2013continuous}, the process $\xi$ is a positive  $(\cF_{t+T})$-supermartingale. 
%Thus, from \cite[Ch. II, Thm 3.3]{revuz2013continuous} it follows that %$\xi$ is a non-negative supermartingale. 
Hence, for all $t \in \RP$, we have 
$$
V(\kappa_T) = \xi_0 \geq \E[\xi_t \vert \cF_T] \geq \E[\xi_t \mathbbm{1}\{\lambda_{\ell,T} \leq \rho_{r,T}\}\vert \cF_T].
$$
%Moreover by \cite[Ch. II, Corollary 2.11]{revuz2013continuous} it follows %that 
Since $\kappa$ is continuous, 
it holds
$\lim_{t\to\infty}\xi_t \mathbbm{1}\{\lambda_{\ell,T} \leq \rho_{r,T}\}%=V(\kappa_{\lambda_{\ell,T}})\mathbbm{1}\{\lambda_{\ell,T} \leq \rho_{r,T}\}
=V(\ell)\mathbbm{1}\{\lambda_{\ell,T} \leq \rho_{r,T}\}$
and $V(\kappa_T)=V(r_1)$ a.s. %Thus, for any sequence $t_k \uparrow \infty$, 
The (conditional) Fatou's lemma yields
\begin{align*}
\P(\lambda_{\ell,T} \leq \rho_{r,T}\vert \cF_T)
V(\ell) &= \E[\liminf_{t \rightarrow \infty}\xi_{t}\mathbbm{1}\{\lambda_{\ell,T} \leq  \rho_{r,T}\}\vert \cF_T] \\&\leq \liminf_{t \rightarrow \infty}\E[\xi_{t}\mathbbm{1}\{\lambda_{\ell,T} \leq \rho_{r,T}\}\vert \cF_T]\leq \xi_0 = V(\kappa_T)= V(r_1).
\end{align*}
Since, by assumption,
$\kappa_t\in\RP$ a.s. for all $t\in\RP$, $\kappa$ does not explode. Thus %By the fact that $T < \infty$ and $\rho_{r,T}<\infty$ a.s. for all $r>r_1$ and that 
$\lim_{r \rightarrow \infty} \rho_{r,T} = \infty$ a.s. and, since
$\rho_{r,T}<\infty$ a.s. for all $r\in(r_1,\infty)$,
we get (recall $V(\ell)>0$)
\begin{equation}
    \label{eq:upper_bound_lower_exit}
\P(\lambda_{\ell,T} < \infty)=\P(\cup_{r\in\N\cap(r_1,\infty)}\{\lambda_{\ell,T} \leq \rho_{r,T}\}) =\lim_{r\to\infty}\P(\lambda_{\ell,T} \leq \rho_{r,T}) \leq V(r_1)/V(\ell).
\end{equation}
%which tends to $0$ as $r_1 \rightarrow \infty$, by our hypothesis on $f$, %and implies the first part of the lemma. 
Recall $T=\rho_{r_1}<\infty$ a.s. and 
note $\{\liminf_{t \rightarrow \infty}  \kappa_t\leq \ell\}\subset \{\lambda_{\ell,\rho_{r_1}} < \infty\}$ for all $r_1\in(\ell,\infty)$.
The inequality
$\P(\liminf_{t \rightarrow \infty}  \kappa_t\leq \ell) \leq \P(\lambda_{\ell,\rho_{r_1}} < \infty)$
for all $r_1\in(\ell,\infty)$,
the upper bound in~\eqref{eq:upper_bound_lower_exit} and the hypotheses $V(r_1)\to0$ as $r_1\to\infty$ imply
$$
\P \Bigl(\liminf_{t \rightarrow \infty}  \kappa_t\leq \ell \Bigr) \leq \limsup_{r_1 \rightarrow \infty} \P(\lambda_{\ell,\rho_{r_1}} < \infty) = 0.
$$
Thus $\liminf_{t \rightarrow \infty}  \kappa_t>\ell$ a.s.
Since $\ell \in(r_0,\infty)$ was arbitrary, transience follows.
\end{proof}

\subsection{Lower bounds on the tails of return times and associated additive functionals}
\label{subsec:lower_bounds_semimartingale_return_Times}

In this subsection we establish \textit{lower} bounds for the tails of the return times of continuous semimartingales and associated additive functionals. In the Markovian setting, there exists a rich theory providing upper bounds for additive functionals considered here, in the context of establishing related upper bounds on the tails of the invariant distribution (see~\cite{douc2009subgeometric} and the references therein). Moreover, in the continuous semimartingale setting,~\cite{menshikov1996passage} establishes upper bounds on the return-time moments. Comparatively, the literature dedicated to lower bounds on the tails of the return times is scarce. Some results in this direction can be found in~\cite{menshikov1996passage}, however the assumptions are too restrictive to be used in our model. 
Our approach, based on the maximal inequality in Proposition~\ref{prop:maximal} below, is inspired by the discrete-time results in~\cite{hryniv2013excursions}.

The link between between additive functionals studied in the present subsection and the invariant distributions in the Markovian setting, described in~\cite{tweedie1993generalized},
enables the application of Lemma~\ref{lem:return_times} below
in the proofs of lower bounds on the tails of the invariant distribution and the sub-exponential convergence rate
in total variation (stated in  Theorem~\ref{thm:invariant_distributon}).
Lemma~\ref{lem:return_times} is also crucial for 
establishing the lower bounds on the tails of return times in Theorem~\ref{thm:return_times}.

\begin{prop}[Maximal inequality]
\label{prop:maximal}
Let $\xi = (\xi_t)_{t\in\RP}$ be an $\RP$-valued $(\cF_t)$-adapted continuous process and $f:\RP^2 \to \RP$ a measurable function. For some $r > 0$, let $\tau_r:=\inf\{t\in\RP:\xi_t\geq r\}$
(with $\inf\emptyset =\infty$) and assume
 %$W=(W_t)_{t\in\RP}$,
$(\xi_{t\wedge \tau_r}- \int_0^{t\wedge \tau_r} f(u,\xi_u)\ud u)_{t\in\RP}$
is an $(\cF_t)$-supermartingale. 
Then, for any $s\in(0,\infty)$, we have $\P(\sup_{0\leq t \leq s}\xi_t\geq r\vert\cF_0)\leq r^{-1}(\xi_0 + \E[\int_0^{s\wedge \tau_r} f(u,\xi_u)\ud u\vert\cF_0])$.
\end{prop}

\begin{proof}
Pick $s\in(0,\infty)$ and consider the  stopping time $\tau_r \wedge s$, bounded above by $s$. By assumption we have $\E[\xi_{\tau_r \wedge s} - \int_0^{\tau_r \wedge s} f(u,\xi_u)\ud u\vert \cF_0] \leq \xi_0$, which implies that
\begin{equation}
\label{eq:maximal_supermart_inequality}
\E[\xi_{\tau_r \wedge s}\vert \cF_0] \leq \xi_0 + \E\left[\int_0^{\tau_r \wedge s} f(u,\xi_u)\ud u \big\vert \cF_0\right].
\end{equation}
Moreover, by the definition of $\tau_r$ in the proposition we have $\{\tau_r \leq s\}=\{\sup_{u\in[0,s]}\xi_u\geq r\}$
a.s. Since the equality $\xi_{\tau_r \wedge s} = r $ holds on this event, by~\eqref{eq:maximal_supermart_inequality}
we have
$$
\P\left(\sup_{0\leq t \leq s}\xi_t \geq r\Big\vert\cF_0\right) = r^{-1}\E[\xi_{\tau_r \wedge s}\mathbbm1\{\tau_r \leq s\}\vert \cF_0] \leq r^{-1}\left(\xi_0 + \E\left[\int_0^{\tau_r \wedge s} f(u,\xi_u)\ud u\Big\vert \cF_0\right]\right),
$$
implying the proposition.
\end{proof}

Note that, in the case 
$\E[\int_0^{s\wedge \tau_r} f(u,\xi_u)\ud u]=\infty$, both the statement and the proof of Proposition~\ref{prop:maximal}
are formally correct, but not informative. In particular, 
when applying Proposition~\ref{prop:maximal} we need 
$\E[\int_0^{s\wedge \tau_r} f(u,\xi_u)\ud u]<\infty$, which follows easily 
if, for example, the function $f$ is continuous (and hence bounded on $[0,s]\times[0,r]$).

\begin{lem}
\label{lem:return_times}
Let $\kappa = (\kappa_t)_{t\in\RP}$ be an $\RP$-valued $(\cF_t)$-adapted continuous process satisfying $\limsup_{t\to\infty}\kappa_t =\infty$ a.s. Suppose that there exist  $p\in(0,\infty)$, $\ell\in(0,\infty)$, and $C \in (0,\infty)$, such that the following hold for all $r\in(\ell,\infty)$:
\begin{enumerate}[label=(\alph*)]
\item \label{assumption:Lem:return_times_a} the process $(\kappa_{t\wedge \lambda_\ell\wedge \rho_r}^p)_{t\in\RP}$ is an $(\cF_t)$-submartingale;

\item \label{assumption:Lem:return_times_b} for any  $q\in(0,1)$ and $r_q := (1-q)^{-1}r$, the process 
$$
\left(\kappa_{(\rho_{r_q}+t)\wedge \lambda_{r,\rho_{r_q}}}^{-2} - C\int_{\rho_{r_q}}^{(\rho_{r_q} + t) \wedge \lambda_{r,\rho_{r_q}}} \kappa_{u}^{-4}\ud u\right)_{t\in\RP},
$$
is an $(\cF_{\rho_{r_q}+t})$-supermartingale.
\end{enumerate}
Let $h:\RP\to\RP$ be a non-decreasing measurable function. Then 
for all $r\in(\ell,\infty)$,  $q\in(0,1)$
and $\eps\in(0,C^{-1}q(1-q)]$,
we have 
$$
\P\left(\int_0^{\lambda_\ell} h(\kappa_s)\ud s \geq \eps h(r)r^2\Big\vert \cF_0\right)\geq q\min\{(\kappa_0^p-\ell^p) (1-q)^{p}r^{-p}, 1\}, \quad \text{on the event $\{ \kappa_0 > \ell\}$}.
$$
In particular, for all $t\in(\ell,\infty)$ (with $h\equiv1$  and $\eps=C^{-1}q(1-q)$), we have
$$
\P(\lambda_\ell \geq t\vert\cF_0) \geq q\min\{(\kappa_0^p-\ell^p)((1-q)^3q/C)^{p/2}t^{-p/2}, 1\},\quad \text{on the event $\{ \kappa_0 > \ell\}$}.$$
\end{lem}

\begin{proof}
Pick $q\in (0,1)$, $r \in(\ell,\infty)$ and note that it suffices to prove the lemma for $\eps=C^{-1}q(1-q)$. We start by establishing the following inequality: 
\begin{equation}
    \label{eq:r_squared_time_to_come_back}
\P(\lambda_{r,\rho_{r_q}} > \rho_{r_q } + \eps r^2\vert \cF_{\rho_{r_q}}) \geq q\quad\text{a.s.}
\end{equation}
%Note that the definition of $\lambda$ in~\eqref{eq::lambda} gives the equality
%$$
%\P\left(\lambda_{r,\rho_{r_q }} \leq \rho_{r_q }+t\vert \cF_{\rho_{r_q }}\right) =\P\left(\sup_{0\leq s \leq t} \kappa_{(\rho_{r_q} + t) \wedge \lambda_{r,\rho_{r_q}}}^{-2} \geq r^{-2}\vert \cF_{\rho_{r_q }}\right).
%$$
Define $(\xi_t)_{t\in\RP}$ by $\xi_t := \kappa_{\rho_{r_q} + t}^{-2}$ and note that~\eqref{eq::lambda} yields
$\tau_{r^{-2}}:=\inf\{t>0:\xi_t \geq r^{-2}\}=\lambda_{r,\rho_{r_q}}-\rho_{r_q}$.
% and let $f(u) := Cu^2$ for $u\in\RP$. 
By Assumption~\ref{assumption:Lem:return_times_b}, the process $(\xi_{t\wedge\tau_{r^{-2}}} - \int_0^{t\wedge\tau_{r^{-2}}} C\xi_u^2\ud u)_{t\in\RP}$ is an $(\cF_{\rho_{r_q}+t})$-supermartingale.  By Proposition~\ref{prop:maximal} applied to $\xi$ and the stopping time $\tau_{r^{-2}}$,  we obtain
\begin{align*}
    \P(\lambda_{r,\rho_{r_q }} \leq \rho_{r_q }+t\vert \cF_{\rho_{r_q }}) &= \P(\tau_{r^{-2}}\leq t\vert \cF_{\rho_{r_q}} )=\P(\sup_{0\leq u \leq t} \xi_u \geq r^{-2}\vert \cF_{\rho_{r_q}}) \\
    &\leq r^2\left( \xi_0 + \E\left[\int_0^{t\wedge\tau_{r^{-2}}} C\xi_u^2\ud u\Big\vert \cF_{\rho_{r_q}}\right] \right) \\
    &= r^2\left(\kappa_{\rho_{r_q}}^{-2} +\E\left[ C\int_{\rho_{r_q}}^{(\rho_{r_q}+t)\wedge \lambda_{r,\rho_{r_q}}}\kappa_u^{-4}\ud u\Big\vert \cF_{\rho_{r_q }}\right]\right) \\
    &\leq r^2(r_q^{-2} + Ctr^{-4}) \leq (1-q)^2 + Cr^{-2}t
    \quad\text{for any $t\in(0,\infty)$.}
\end{align*}
It follows that $\P(\lambda_{r,\rho_{r_q }} > \rho_{r_q } + t) \geq 1 - ((1-q)^2 + Cr^{-2}t)$. Taking $t = \eps r^{2}$, we obtain~\eqref{eq:r_squared_time_to_come_back}.
%with $\eps=C^{-1}q(1-q)$, we obtain $(1-q)^2 + C r^{-2}t =(1-q)$ for all $r\in(l,\infty)$. Thus, 
%$\P(\lambda_{r,\rho_{r_q }} > \rho_{r_q } + \eps r^2) \geq q$
%implying~\eqref{eq:r_squared_time_to_come_back}.
%Observe that 
%$$
%\P\left(\{\rho_{r_q  } < \lambda_l\} \cap \{\lambda_{r,\rho_{r_q  %}} > \rho_{r_q  } + \eps r^2 \}\right) = %\E\left[\mathbbm{1}\{\rho_{r_q  } < %\lambda_l\}\P\left(\lambda_{r,\rho_{r_q }} > \rho_{r_q  } + \eps %r^2\big\vert \cF_{\rho_{r_q}}\right)\right],
%$$
%and 

Note that on the event $\{\lambda_{r,\rho_{r_q }} > \rho_{r_q} + \eps r^2\}$, we have $h(\kappa_{\rho_{r_q} + t})\geq h(r)$ for all $t\in[0,\eps r^2]$ and measurable non-decreasing functions $h:\RP\to\RP$.
Since $r>\ell$, the inclusion
$$\left\{\int_0^{\lambda_\ell} h(\kappa_t) \ud t \geq \eps h(r)r^2 \right\}\supset \{\rho_{r_q } < \lambda_\ell\}\cap\{\lambda_{r,\rho_{r_q }} > \rho_{r_q } + \eps r^2\}$$ holds and,
by the inequality in~\eqref{eq:r_squared_time_to_come_back}, we obtain 
\begin{align}
\nonumber
    \P\left( \int_0^{\lambda_\ell} h(\kappa_t) \ud t \geq \eps h(r)r^2\Big \vert\cF_0\right) &\geq \E\left[\mathbbm{1}\{\rho_{r_q } < \lambda_\ell\}\P\left(\lambda_{r,\rho_{r_q }} > \rho_{r_q } + \eps r^2\vert \cF_{\rho_r}\right)  \big\vert\cF_0\right] \\
    \label{eq:return_times_excursion}
    &\geq q\P(\rho_{r_q }<\lambda_\ell\vert\cF_0).
\end{align}
%holds for all $r\in(l,\infty)$. 
By assumption
$\limsup_{t\to\infty}\kappa_t =\infty$ a.s. and thus 
$\rho_{r_q }\wedge \lambda_\ell<\infty$ a.s. 
Since $(\kappa_{t\wedge \rho_{r_q}\wedge\lambda_\ell}^p)_{t\in\RP}$ is a continuous $(\cF_t)$-submartingale by Assumption~\ref{assumption:Lem:return_times_a}, dominated convergence implies
$$
\kappa_0^p \leq \lim_{t\to\infty} \E[\kappa_{t\wedge \lambda_\ell\wedge \rho_{r_q }}^p\vert \cF_0] = \E[\kappa_{\lambda_\ell\wedge \rho_{r_q }}^p\vert\cF_0] \leq \ell^p + \P(\rho_{r_q }<\lambda_\ell\vert \cF_0)r_q^p.
$$
On the event $\{\kappa_0 > \ell\}$ we obtain $\P(\rho_{r_q}<\lambda_\ell\vert \cF_0) \geq ((\kappa_0^p-\ell^p)r_q^{-p})\wedge 1$. Combining this result with~\eqref{eq:return_times_excursion} implies the  general case of the lemma. The special case follows by setting $h\equiv1$.
\end{proof}

\section{Non-explosion and recurrence/transience dichotomy for the reflected process} \label{subsec:non-explosion_trans_rec_proofs}

\subsection{Diffusivity and non-explosion under asymptotically normal reflection}
\label{subsec:Proof_non_explosion_moments}
The non-explosion of $Z$ is essentially due to Assumption~\aref{ass:vector2}, which stipulates that the horizontal projection of the reflection  vanishes sufficiently fast when $Z_t$ is far from the origin. In contrast, in the asymptotically oblique case~\cite{menshikov2022reflecting}, the horizontal projection of the reflection vector field
has
%may have 
a strictly positive limit as $\|Z_t\|_{d+1}\to\infty$, which  in a domain $\cD$
that narrows sufficiently fast, may lead to explosive behaviour~$\tau_\cE < \infty$ a.s.

\begin{thm}
\label{thm:non_explosion_moments}
Suppose that Assumptions~\aref{ass:domain2}, \aref{ass:covariance2}, \aref{ass:vector2} hold and
let $(Z,L)$ satisfy SDE~\eqref{eq::SDE} in the domain $\cD$ on the stochastic interval $[0,\taue)$.
For any starting point $z\in\cD$, the process $Z$ 
does not explode,  i.e.~$\P_z(\tau_\cE=\infty)=1$, and 
%the $p$-th moment of $Z_t$  is finite and diffusive, 
the second moment of $Z_t$  is finite and diffusive,
i.e.,
$\sup_{t\in\RP} \E_z \|Z_t\|_{d+1}^2/(1+t)<\infty$. 
%for  any $p\in(0,\infty)$.
\end{thm}
 
The finiteness of the first moment of $\|Z_t\|_{d+1}$, implied by Theorem~\ref{thm:non_explosion_moments}, is crucial in the proof of recurrence in Theorem~\ref{thm:rec_tran}(a) and Theorem~\ref{thm:rec_tran}(c) (see Section~\ref{subsec:Proof_Thm_trans_rec} below for details).
A minor modification of the final step in the proof of Theorem~\ref{thm:non_explosion_moments}
would imply that the $p$-th moment of $Z_t$ is also diffusive for any $p\in(0,\infty)$,
i.e., $\sup_{t\in\RP} \E_z \|Z_t\|_{d+1}^p/(1+t)^{p/2}<\infty$
(see Remark~\ref{rem:modification_general_p}, after the proof of Theorem~\ref{thm:non_explosion_moments} below).

\begin{proof}[Proof of Theorem~\ref{thm:non_explosion_moments}]
    Assume~\aref{ass:vector2}, \aref{ass:domain2} and \aref{ass:covariance2} are satisfied.
Consider the process  $f_{w,\gamma}(Z)$
for any $\gamma\in(-\infty,1]$ and $w\in\R\setminus\{0\}$,
where the function $f_{w,\gamma}$ is 
given in~\eqref{eq:Lyapunov_polynomial}.
It\^o's formula applied to $f_{w,\gamma}(Z)$
on the stochastic interval $[0,\taue)$ is given in~\eqref{eq::Ito}. The quadratic variation of the local martingale $M$ in~\eqref{eq::Ito}
grows at most linearly in time. Indeed, 
the representation in~\eqref{eq::Ito_QV} and the bound in~\eqref{eq:grad_f_gamma_w_bound} of Lemma~\ref{lem3.1} (recall that $\gamma\leq 1$) imply
\begin{equation}
\label{eq::f_QV_bound}
[M]_t \leq \int_0^t \|\Sigma(Z_s)\|_{\text{op}}\|\nabla f_{w,\gamma}(Z_s)\|_{d+1}^2 \ud s\leq C_\gamma t, \text{ for }0 \leq t < \tau_{\cE},
\end{equation}
for a   constant $C_\gamma \in \RP$. 
The inequality in~\eqref{eq::f_QV_bound} relies on the norm
$\|\Sigma\|_{\text{op}}=\|\Sigma^{1/2}\|_{\text{op}}^2$
being bounded via
Assumption~\aref{ass:covariance1}. 

Our first task is to prove $\P_z(\taue=\infty)=1$ for any $z\in\cD$. Pick $w \in(-\infty, 1-\beta s_0/c_0)\setminus\{0\}$, ensuring by Lemma~\ref{lem3.1} that $\langle \nabla f_{w,1}(z),\phi(z)\rangle < 0$
for all $z=(x,y)\in\cD$ with sufficiently large $x$, and define the process $\kappa := f_{w,1}(Z)$. 
Recall the definition of the stopping times $\rho_r$ (for any $r\in\R_+$ and $T=0$) in~\eqref{eq::rho} and $\rho_\infty=\lim_{r\to\infty}\rho_r$, both
given in Section~\ref{sec:Trans_recurr_dichotomy}
for the process $\kappa$.
By~\eqref{eq:bound_Lyapunov_f}, the function $f_{w,1}$ has linear growth at infinity, implying the equality of the events
$\{\taue<\infty\}=\{\rho_\infty<\infty\}$.

We will apply Lemma~\ref{lem2.1} with the identity function $V(r) = r$ for all $r \in (2^{-1/|w|-1},\infty)$
to conclude $\P(\rho_\infty<\infty)=0$. Since $\gamma=1$, by~\eqref{eq:Sigma_Laplacian_f} in Lemma~\ref{lem3.1},
we have 
$$
\Delta_\Sigma f_{w,1}(z) = f_{w,1}(z)^{-1} \sigma_2^2(1-w) + o_\cD(1)\quad\text{as $x\to\infty$.}
$$
Thus, by~\eqref{eq:bound_Lyapunov_f}, there exists $\eta>0$ satisfying $|\Delta_\Sigma f_{w,1}(z)|<\eta$ for all $z\in\cD$.
Moreover,
by Lemma~\ref{lem3.1}, there exists $r_0>0$ such that 
$\langle \nabla f_{w,1}(z),\phi(z)\rangle < 0$
for all $z=(x,y)\in\cD$ satisfying $x\geq r_02^{-1/|w|}-k_w$.
Note that the upper bound on $f_{w,1}$ in~\eqref{eq:bound_Lyapunov_f} 
implies that 
any $z\in\cD$ with 
$f_{w,1}(z)\geq r_0$
must satisfy 
$x\geq r_02^{-1/|w|}-k_w$ (and hence 
$\langle \nabla f_{w,1}(z),\phi(z)\rangle < 0$).
%which satisfies the conditions of . Since  $V(\kappa_t) = %f_{w,1}(Z_t)$, the process $(V(\kappa_t))_{t \in %[0,\rho_{\infty})}$ can be represented by the %identity~\eqref{eq::Ito}. Furthermore Lemma~\ref{lem3.1} %implies that in the identity~\eqref{eq::Ito} reflection is %asymptotically negative and the drift is bounded. Thus, for %some $\eta \in \RP$, $r > r_0$ with large enough $r_0$ and for 

For any stopping time $T \in \cT$,
recall the definition in~\eqref{eq::exit} of the exit time $S_{r,T}$
of the process $\kappa$ from the interval $(r_0,r)$ after time $T$.
In order to apply Lemma~\ref{lem2.1}, 
assume that the stopping time $T$ is such that $\E[\kappa_T\mathbbm{1}\{T < \rho_\infty\}] < \infty$.
Then, by It\^o's formula in~\eqref{eq::Ito} and the choice of the constants $\eta$ and $r_0$, the process $\zeta^{T,r}=(\zeta_t^{T,r})_{t\in\RP}$, defined for any $t\in\RP$ by  
$\zeta_t^{T,r} := (\kappa_{(t+T) \wedge S_{r,T}} - \eta ( t \wedge ( S_{r,T} - T ) ) ) \mathbbm{1}\{T < \rho_\infty \}$,
 satisfies %has a negative drift and reflection. This further implies that 
 $$
 \zeta_t^{T,r}-\zeta_s^{T,r}\leq \left(M_{(t+T) \wedge S_{r,T}}-M_{(s+T) \wedge S_{r,T}}\right)\mathbbm{1}\{T < \rho_{\infty}\}
 \quad\text{a.s. for any $0\leq s\leq t$,}
 $$ 
  where $M$ is the local martingale arising in~\eqref{eq::Ito} (for the function $f_{w,1}$). Since $\kappa$ is continuous, on the event $\{T < \rho_{\infty}\}$ it holds that $S_{r,T} < \rho_\infty$. Thus, by~\eqref{eq::f_QV_bound}, we get
 $$
 [M]_{(t+T) \wedge S_{r,T}}-[M]_T \leq C_1((t+T) \wedge S_{r,T}-T)\leq C_1t \quad\text{for all $t\in\RP$,}
 $$
 ensuring that $(M_{(t+T) \wedge S_{r,T}})_{t\in\RP}$ is a true $(\cF_{t+T})$-martingale and implying 
 $\E[\zeta_t^{T,r}-\zeta_s^{T,r}\vert \cF_{s+T}] \leq 0$
 for all $0\leq s\leq t$.
 Since $\zeta^{T,r}$ is an $(\cF_{t+T})$-supermartingale
 for all $r\in(r_0,\infty)$, Lemma~\ref{lem2.1}
 yields $\P(\rho_\infty=\infty)=1$. Thus, 
 $\P_z(\tau_\cE = \infty)=1$
 for all $z\in\cD$. 
 In the remainder of the section we assume that $Z$ satisfies SDE~\eqref{eq::SDE}
on the entire time interval $\RP$. 

%We claim that for every $p \in \ZP$, there exists a constant $C_p < \infty$ such %that, for any $z \in \cD$,
%\begin{equation}
%\label{eq:Z-moments-bound}
%\E_z [ \| Z_t \|^{2p} ] \leq C_p t^p + C_p \| z \|^{2(p-1)^+} t^{(p-1)^+}   +  \| z %\|^{2p}  , \text{ for all } t \in \RP. \end{equation}
%We prove~\eqref{eq:Z-moments-bound} by induction on $p\in\ZP$. The base case %($p=0$) is trivial. Assume that~\eqref{eq:Z-moments-bound} holds for a given $p \in %\ZP$. First note that, by the lower bound in~\eqref{eq:bound_Lyapunov_f} and the %fact that $F_{w,\gamma} (z)$ is bounded away from $0$ on compact sets, $\| z\|^{2p} %\leq C_{p,w} F_{w,2p} (z)$ for all $z \in \cD$. (MAYBE $w=1$ is fine)

We now prove there exists a constant $C_2>0$ such that $\E_z\|Z_t\|_{d+1}^2\leq C_2 (t+\|z\|_{d+1}^2+1)$ holds for any $z\in\cD$ and $t\in\RP$.
Pick $w\in(-\infty,1-\beta s_0/c_0)\setminus\{0\}$, note $\beta s_0/c_0-1+w<0$ and apply Lemma~\ref{lem:F_w,gamma} to find the constants $0<x_0<x_1$ and $k\in(0,\infty)$ such that the corresponding function  $F_{w,2}$, defined in~\eqref{eq:F_w,gamma}, satisfies $\langle\nabla F_{w,2}(z),\phi(z)\rangle \leq 0$ for all $z\in\partial \cD$. Moreover, by~\eqref{eq:F_w,gamma}, we have $\Delta_\Sigma F_{w,2}(z)=\Delta_\Sigma f_{w,2}(z)$ for all $z=(x,y)\in\cD \cap [x_1,\infty)\times \R^d$. Since $F_{w,2}$ is smooth on a neighbourhood of $\cD$
and, by~\eqref{eq:Sigma_Laplacian_f} in  Lemma~\ref{lem3.1},
the function $z\mapsto|\Delta_\Sigma f_{w,2}(z)|$ is bounded on $\cD$,
there exists a constant 
$C_0'\in(0,\infty)$ satisfying $|\Delta_\Sigma F_{w,2}(z)|<C_0'$ for all $z\in\cD$.
By Assumption~\aref{ass:domain2}, the boundary function $b$ is sublinear (cf. Remark~\ref{rem:Ass2} above): there exist constants $C_1',C_2'\in(0,\infty)$ such that $b(x)^2<C_1'x^2+C_2'$ for all $x\in\RP$.
Thus, by the definition of $F_{w,2}$ in~\eqref{eq:F_w,gamma}
and the lower bound on $f_{w,2}$ in~\eqref{eq:bound_Lyapunov_f}, there exist positive constants $C_i'\in(0,\infty)$, $i\in\{3,4,5,6\}$,  satisfying 
\begin{equation}
\label{eq:basic_dereministic_bound}
\|z\|_{d+1}^2\leq x^2+b(x)^2\leq 
(C_1'+1)x^2+C_2'\leq
C_3' F_{w,2}(z)+C_4'\leq C_5' \|z\|_{d+1}^2 +C_6'\quad\text{for all $z\in\cD$.}
\end{equation}

Recall that the coordinates of $Z=(X,Y)$, taking values in $\cD$, satisfy $X_t\in\RP$ and $Y_t\in\R^d$ for all $t\in\RP$. Define the passage time of the level $r\in\RP$ for the process $X$ by
\begin{equation}
    \label{eq::varrho}
    \varrho_r := \inf\{t\in\RP: X_t\geq r\},
\end{equation}
(with $\inf \emptyset =\infty$).
Fix $r$ and assume that the starting point $z\in\cD$ of the process $Z$
lies in $[0,r)\times\R^d$ or, equivalently, $\P_z(\varrho_r>0)=1$.
The definition of $F_{w,2}$ in the previous paragraph and
It\^o's formula in~\eqref{eq::Ito} applied to the process
$(F_{w,2}(Z_{t\wedge \varrho_r}))_{t\in\RP}$
yield the following inequalities for all $t\in\RP$:
\begin{align*}
F_{w,2}(Z_{t\wedge \varrho_r})\leq F_{w,2}(z)+ \frac{1}{2} \int_0^{t\wedge \varrho_r} \Delta_\Sigma F_{w,2}(Z_s)\ud s + M_{ \varrho_r\wedge t}+\leq 
F_{w,2}(z)+ tC_0' + M_{ \varrho_r\wedge t}.
%\quad\text{}
\end{align*}
%Moreover, by~\eqref{eq:basic_dereministic_bound}, we have
%$\|Z_{t\wedge \varrho_r}\|_{d+1}^2\leq (C_1'+1)r^2+C_2'$
%for all $t\in\RP$,
%implying that
Moreover, the quadratic variation of $M_{\varrho_r\wedge\cdot}$ in~\eqref{eq::Ito_QV} is almost surely bounded 
and hence integrable,
since the gradient 
$\nabla F_{w,2}$ is bounded on compact sets. Thus 
$\E_zM_{t\wedge \varrho_r}=0$ for all $t,r\in\RP$, implying the inequality 
$\E_z  F_{w,2} (Z_{t\wedge \varrho_r}) \leq tC_0' + F_{w,2}(z)$
for all $t,r\in\RP$ and $z\in\cD\cap [0,r)\times\R^d$.
%By~\eqref{eq:bound_Lyapunov_f} and~\eqref{eq:F_w,gamma} there exist constants %$C_5',C_6'>0$ satisfying $F_{w,2}(z)\leq C_5' \|z\|_{d+1}^2 +C_6'$ for all %$z\in\cD$. 
Since $Z$ does not explode in finite time and has continuous paths,
we have $\varrho_r\to\infty$ a.s. as $r\to\infty$ and hence 
$Z_{t\wedge \varrho_r}\to Z_t$
a.s. as $r\to\infty$.
By Fatou's lemma and the inequalities in~\eqref{eq:basic_dereministic_bound},
for all $t\in\RP$ and $z\in\cD$, we obtain
$$
\E_z\|Z_t\|_{d+1}^2= \E_z \liminf_{r\to\infty}\|Z_{t\wedge \varrho_r}\|_{d+1}^2\leq \liminf_{r\to\infty}\E_z\|Z_{t\wedge \varrho_r}\|_{d+1}^2
%\leq tC_0'C_3'+ C_5'C_3' \|z\|_{d+1}^2 +C_6'C_3'+C_4',
\leq tC_0'C_3'+ C_5'  \|z\|_{d+1}^2 +C_6',
$$
concluding the proof of Theorem~\ref{thm:non_explosion_moments}.
\end{proof}

\begin{rem}
\label{rem:modification_general_p} 
The proof of Theorem~\ref{thm:non_explosion_moments} can be modified to get
$\sup_{t\in\RP} \E_z \|Z_t\|_{d+1}^{2p}/(1+t)^{p}<\infty$
for any $p\in(0,\infty)$.
For $p\in\N$, we note that $F_{w,2p}$
can be constructed by Lemma~\ref{lem:F_w,gamma} as above, so that
$\langle\nabla F_{w,2p}(z),\phi(z)\rangle \leq 0$ for all $z\in\partial \cD$.
By~\eqref{eq:Sigma_Laplacian_f} in  Lemma~\ref{lem3.1} there exists
$C_0'\in(0,\infty)$ satisfying $|\Delta_\Sigma F_{w,2p}(z)|<C_0'F_{w,2(p-1)}(z)$ for all $z\in\cD$.
This inequality, combined with the modification of~\eqref{eq:basic_dereministic_bound} for $\|z\|_{d+1}^{2p}$
and $F_{w,2p}(z)$, and induction on $p\in\N$ yields the diffusive property of moments for all even powers. 
The statement for all positive real powers $p$ can be deduced from the even powers via 
Lyapunov's
%H\"older's 
inequality. As the generalisation 
of the $p=1$ is not essential for the development in the present paper, the details are omitted for brevity.  
\end{rem}

\subsection{Proof of recurrence/transience classification}
\label{subsec:Proof_Thm_trans_rec}
By Theorem~\ref{thm:non_explosion_moments}, in the remainder of the section we may assume without loss of generality that $Z$ satisfies SDE~\eqref{eq::SDE}
on the entire time interval $\RP$. 
The proof of Theorem~\ref{thm:rec_tran}
starts with a lemma about non-confinement, essential  for the applications of Lemmas~\ref{lem2.2} and~\ref{lem2.3} in the proof of recurrence/transience dichotomy.

\begin{lem}
\label{lem:non_cofinment_f_F}
Under~\aref{ass:vector1}, \aref{ass:domain1} and \aref{ass:covariance1}, the process $Z=(X,Y)$ defined by SDE~\eqref{eq::SDE}, started at any $z\in\cD$, satisfies
$$\limsup_{t\to\infty}X_t=\limsup_{t\to\infty}f_{w,1}(Z_t)=\limsup_{t\to\infty}F_{w,1}(Z_t)=\limsup_{t\to\infty}g_\delta(Z_t)=\infty\quad\text{ $\P_z$-a.s.}$$ 
for any $w\in\R\setminus\{0\}$ and $\delta\in(0,\infty)$, where the functions $f_{w,1}$, $F_{w,1}$ and $g_\delta$ are given  in~\eqref{eq:Lyapunov_polynomial},~\eqref{eq:F_w,gamma} and~\eqref{eq:g_delta}, respectively.
\end{lem}

\begin{proof}
Under~\aref{ass:vector1}, \aref{ass:domain1}, \aref{ass:covariance1}, \cite[Thm~4.1]{menshikov2022reflecting} implies that the process $Z=(X,Y)$ cannot be confined to a compact set:
%with positive probability: 
$\limsup_{t \to \infty} X_t = \infty$ $\P_z$-a.s.~for all $z\in\cD$. 
%This is essentially due to the uniform %ellipticity of $\Sigma$
%in~\aref{ass:covariance1} and the %non-vanishing nature of the reflection vector %field $\phi$ in the normal direction %in~\aref{ass:vector1}.
In particular, by the inequality in~\eqref{eq:bound_Lyapunov_f} (resp.~\eqref{eq:bound_Lyapunov_g}), for any parameter value $w\in\R\setminus\{0\}$
(resp.\ $\delta\in(0,\infty)$)
the process $f_{w,1}(Z)$ (resp.\ $g_\delta(Z)$) is also not confined: $\limsup_{t \to \infty}f_{w,1}(Z_t) = \infty$ (resp.\ $\limsup_{t\to\infty}g_\delta(Z_t)=\infty$) a.s. Since, by~\eqref{eq:F_w,gamma}, the functions $f_{w,1}$ and $F_{w,1}$ coincide on the complement of a neighbourhood of the origin, the lemma follows.    
\end{proof}

\begin{proof}[Proof of Theorem~\ref{thm:rec_tran}]
    Since Assumptions~\aref{ass:vector1}, \aref{ass:domain1}, \aref{ass:covariance1} hold if~\aref{ass:vector2}, \aref{ass:domain2}, \aref{ass:covariance2} are satisfied, we may apply Lemma~\ref{lem:non_cofinment_f_F} in the proofs of this section.

\smallskip

\noindent \underline{(a) Recurrence for $\beta<\beta_c$}.
The definition of $\beta_c$ in~\eqref{eq:beta_c} and the assumption  $\beta<\beta_c$ imply that $\beta s_0/c_0<\sigma_1^2/\sigma_2^2$ (recall the the definition of $\sigma_1^2,\sigma_2^2$ and $s_0,c_0$
in Assumptions~\aref{ass:covariance2} and~\aref{ass:vector2}, respectively).
Pick $w\in(1-\sigma_1^2/\sigma_2^2,1-\beta s_0/c_0)\setminus\{0\}$
and note that
$0<1-(1-w)\sigma_2^2/\sigma_1^2$. Choose $\gamma\in(0,\min\{1,1-(1-w)\sigma_2^2/\sigma_1^2\})$
and observe the inequalities:
\begin{equation}
    \label{eq:Recur_parm_choices}
    \sigma_1^2(\gamma-1)+\sigma_2^2(1-w)<0\quad\text{and}\quad
    \gamma(s_0\beta/c_0-1+w)<0.
\end{equation}
Lemma~\ref{lem3.1} and the inequalities in~\eqref{eq:Recur_parm_choices} and~\eqref{eq:bound_Lyapunov_f} imply that there exists
$r_0\in(0,\infty)$ such that for all $z=(x,y)\in\cD$ with $x\geq 2^{-1/|w|}r_0-k_w$
we have
\begin{equation}
\label{eq:negative_drift_inequalities_f_w_gamma}
\Delta_\Sigma f_{w,\gamma}(z)<0\quad \text{and}\quad
\langle \nabla f_{w,\gamma}(z),\phi(z)\rangle < 0.
\end{equation}

Consider the process $\kappa = f_{w,1}(Z)$ and a continuous function $V:\RP\to(0,\infty)$, satisfying $V(r) = r^{\gamma}$ for all $r\in(2^{-1/|w|-1},\infty)$.
Since, by~\eqref{eq:bound_Lyapunov_f}, $f_{w,1}(z)\geq 2^{-1/|w|}$ for all $z\in\cD$,
we have $V(\kappa)=f_{w,\gamma}(Z)$. Moreover, by~\eqref{eq:bound_Lyapunov_f} and the fact that $0<\gamma<1$, there exist constants $D_1,D_2\in(0,\infty)$ such that the inequality $f_{w,\gamma}(z) \leq D_1 \|z\|_{d+1}+D_2$ 
holds
for all $z\in\cD$. Thus, by Theorem~\ref{thm:non_explosion_moments}, we get $0<\E[V(\kappa_t)]=\E_z[f_{w,\gamma}(Z_t)] \leq D_1 \E_z \|Z_t\|_{d+1}+D_2< \infty$ for any $t\in\RP$ and $z\in\cD$.

For any fixed $t_0\in\RP$ and any $r\in(r_0,\infty)$,
recall the definition in~\eqref{eq::exit} of
the exit time $S_{r,t_0}$
of the process $\kappa$ from the interval $(r_0,r)$ after time $t_0$.
The choice of $r_0$, the inequalities in~\eqref{eq:negative_drift_inequalities_f_w_gamma}, and 
It\^o's formula in~\eqref{eq::Ito}, applied to the process
$V(\kappa_{(\cdot+t_0) \wedge S_{r,t_0}})=f_{w,\gamma}(Z_{(\cdot+t_0) \wedge S_{r,t_0}})$
 imply 
$$
f_{w,\gamma}(Z_{(t+t_0) \wedge S_{r,t_0}})- f_{w,\gamma}(Z_{t_0})-(M_{(t+t_0) \wedge S_{r,t_0}}-M_{t_0})\leq 0\quad\text{a.s.}
$$
The local martingale $(M_t)_{t\in\RP}$ has integrable quadratic variation by~\eqref{eq::f_QV_bound}, making it a true martingale and implying $\E[V(\kappa_{(t+t_0)\wedge S_{r,t_0}}) -V(\kappa_{(s+t_0)\wedge S_{r,t_0}})|\mathcal{F}_{s+t_0}]\leq 0$ for all $0 \leq s\leq t$. Since
$V(\kappa_{t_0\wedge S_{r,t_0}})=V(\kappa_{t_0})$
is integrable, the process
$(V(\kappa_{(t+t_0)\wedge S_{r,t_0}}))_{t\in\RP}$ is an $(\cF_{t+t_0})$-supermartingale for all $r\in(r_0,\infty)$. By~Lemma~\ref{lem:non_cofinment_f_F},
we have $\limsup_{t \to \infty} \kappa_t = \infty$ a.s. Since $\lim_{r \to \infty}V(r) = \infty$, we may apply Lemma~\ref{lem2.2} to conclude that $\kappa= f_{w,1}(Z)$ is recurrent. By~\eqref{eq:bound_Lyapunov_f} and~\eqref{eq:basic_dereministic_bound}, the recurrence of $Z$ follows.

\smallskip

\noindent \underline{(b) Transience for $\beta>\beta_c$}.
The definition of $\beta_c$ in~\eqref{eq:beta_c} and the assumption  $\beta>\beta_c$ imply that $\sigma_1^2/\sigma_2^2<\beta s_0/c_0$. Pick $w \in (1-\beta s_0/c_0,1-\sigma_1^2/\sigma_2^2)\setminus\{0\}$ and note that $1-\sigma_2^2/\sigma_1^2(1-w)<0$. Choose $\gamma \in (1-\sigma_2^2/\sigma_1^2(1-w),0)$ and observe the inequalities:
\begin{equation}
\label{eq::Tran_parm_choice}
\gamma(\sigma_1^2(\gamma-1)+\sigma_2^2(1-w))<0\quad\text{and}\quad
    \gamma(s_0\beta/c_0-1+w)<0.
\end{equation}
Lemma~\ref{lem3.1}, along with the inequalities in~\eqref{eq::Tran_parm_choice} and~\eqref{eq:bound_Lyapunov_f} imply that there exists 
$r_0\in(0,\infty)$ such that for all $z=(x,y)\in\cD$ with $x\geq 2^{-1/|w|}r_0-k_w$
we have
$$
\Delta_\Sigma f_{w,\gamma}(z)<0\quad \text{and}\quad
\langle \nabla f_{w,\gamma}(z),\phi(z)\rangle < 0.
$$
Consider the process $\kappa = f_{w,1}(Z)$ and a continuous function $V:\RP\to(0,\infty)$, satisfying $V(r) = r^{\gamma}$ for all $r\in(2^{-1/|w|-1},\infty)$.
Since, by~\eqref{eq:bound_Lyapunov_f}, $f_{w,1}(z)\geq 2^{-1/|w|}$ for all $z\in\cD$,
we have $V(\kappa)=f_{w,\gamma}(Z)$.
For any 
stopping time $T \in \cT$, satisfying  $T < \infty$ a.s.,
we have $\E_zV(\kappa_T)<\infty$ for all $z\in\cD$ since the function $V$ is bounded.
Pick $r\in(r_0,\infty)$
and recall the definition in~\eqref{eq::exit} of
the exit time $S_{r,T}$
of the process $\kappa$ from the interval $(r_0,r)$ after time $T$.
The choice of $r_0$ and 
It\^o's formula in~\eqref{eq::Ito}, applied to the process
$V(\kappa_{(\cdot+T \wedge S_{r,T}})=f_{w,\gamma}(Z_{(\cdot+T) \wedge S_{r,T}})$,
imply
$$
f_{w,\gamma}(Z_{(t+T) \wedge S_{r,T}})- f_{w,\gamma}(Z_{T})-(M_{(t+T) \wedge S_{r,T}}-M_{T})\leq 0\quad\text{a.s.}
$$
Moreover, local martingale $(M_t)_{t\in\RP}$ has integrable quadratic variation by~\eqref{eq::f_QV_bound}, making it a true martingale, and for any $t\in\RP$, $\E[V(\kappa_t)] < \infty$ by the fact that $V$ is bounded. This implies $\E[V(\kappa_{(t+T)\wedge S_{r,T}})-V(\kappa_{(s+T)\wedge S_{r,T}})|\cF_{s+T}]\leq 0$ for all $0 \leq s \leq t$.  Thus, $(V(\kappa_{(t+T)\wedge S_{r,T}})_{t\in\RP}$ is an $(\cF_{t+T})$-supermartingale for all $r\in(r_0,\infty)$.  Since $\limsup_{t \to \infty} \kappa_t = \infty$ a.s. and $\lim_{r \to \infty}V(r) = 0$, Lemma \ref{lem2.3} yields transience of $\kappa=f_{w,1}(Z)$.
By~\eqref{eq:bound_Lyapunov_f}, the transience of $Z$ follows.

\smallskip

\noindent \underline{(c) The critical case $\beta=\beta_c$}. Assume $\beta=\beta_c$ and \aref{ass:domain2plus}, \aref{ass:vector2plus}, \aref{ass:covariance2plus}. Consider the process $\kappa := g_{\delta}(Z)$ with $g_\delta$ defined in~\eqref{eq:g_delta} and the parameter $\delta$ chosen to satisfy the assumption in Lemma~\ref{lem:g_delta}.
Then, by Lemma~\ref{lem:g_delta}, there exists  $x_0>0$ such that the inequalities in~\eqref{eq:sigma_lap_scalar_prod_g_delta} hold.
Define $r_0:=C_\delta + \log x_0$ and note that, 
by~\eqref{eq:bound_Lyapunov_g}, the inequality
$r_0\geq g_\delta(z)$ (where $z=(x,y)$)
implies 
$x\in[x_0,\infty)$.
Set $V(r) = r$ for all $r \in (1/2,\infty)$
and note $V(\kappa_t) =\kappa_t= g_{\delta}(Z_t)$ (recall $g_\delta > 1$ on $\cD$).
Pick $t_0\in\RP$ and $r\in(r_0,\infty)$
and 
recall the definition in~\eqref{eq::exit}
of the exit time $S_{r,t_0}$ 
of the process $\kappa$ from the interval $(r_0,r)$ after time $t_0$.
The choice of $r_0$, the inequalities in~\eqref{eq:sigma_lap_scalar_prod_g_delta} and 
It\^o's formula in~\eqref{eq::Ito}, applied to the process
$\kappa_{(\cdot+t_0) \wedge S_{r,t_0}}=g_{\delta}(Z_{(\cdot+t_0) \wedge S_{r,t_0}})$,
imply
$$
g_\delta(Z_{(t+t_0) \wedge S_{r,t_0}})- g_\delta(Z_{t_0})-(M_{(t+t_0) \wedge S_{r,t_0}}-M_{t_0})\leq 0\quad\text{a.s.}
$$
By continuity, the gradient $\|\nabla g_\delta(z)\|_{d+1}^2$ is bounded on compact sets and $\|\Sigma\|_{\text{op}}$ is bounded by Assumption~\aref{ass:covariance1}. Thus, by the representation in~\eqref{eq::Ito_QV}, we can bound the quadratic variation
\begin{align*}
[M]_{(t+t_0) \wedge S_{r,t_0}}-[M]_{t_0} \leq \int_{t_0}^{(t+t_0) \wedge S_{r,t_0}} \|\Sigma(Z_s)\|_{\text{op}}\|\nabla g_{\delta}(Z_s)\|_{d+1}^2ds \leq \tilde C_1t\quad\text{a.s.},
\end{align*}
where $\tilde C_1\in(0,\infty)$ is a positive constant. Thus the process $(M_{(t+t_0) \wedge S_{r,t_0}}-M_{t_0})_{t\in\RP}$ is a true martingale.
Moreover, since $g_\delta(z) \leq \tilde C_2\|z\|_{d+1} + \tilde C_3$ holds for all $z\in\cD$ for some positive constants $\tilde C_2,\tilde C_3$, Theorem~\ref{thm:non_explosion_moments} implies $\E_z[\kappa_t]=\E_z[g_{\delta}(Z_t)]\leq  \tilde C_2 \E_
z \|Z_t\|_{d+1} +\tilde C_3<\infty$ for any $t\in\RP$ and $z\in\cD$.
Thus, $\E[\kappa_{(t+t_0)\wedge S_{r,t_0}} -\kappa_{(s+t_0)\wedge S_{r,t_0}}|\mathcal{F}_{s+t_0}]\leq 0$ for all $0 \leq s\leq t$, and hence $(\kappa_{(t+t_0)\wedge S_{r,t_0}})_{t\in\RP}$ is an $(\cF_{t+t_0})$-supermartingale for all $r\in(r_0,\infty)$ and any $t_0\in\RP$. Moreover, by 
Lemma~\ref{lem:non_cofinment_f_F},
we have
 $\limsup_{t \to \infty}\kappa_t=\infty$ a.s.
Since $\lim_{r \to \infty}V(r) = \infty$, we may apply Lemma~\ref{lem2.2} to conclude that $\kappa= g_{\delta}(Z)$ is recurrent. By~\eqref{eq:bound_Lyapunov_g} and~\eqref{eq:basic_dereministic_bound}, the recurrence of $Z$ follows.
\end{proof}

\section{Return times and drift conditions}
\label{subsection:return_times_drift_conditions}

The tails of return times are controlled by Propositions~\ref{prop:return_time_upper_bound} and~\ref{prop:return_time_lower_bound},
established in this section. The two propositions are crucial
in the proof of Theorem~\ref{thm:return_times}(b);
see Remark~\ref{rem:proof_of_return_time_thm} below for more details.
Moreover, Propositions~\ref{prop:return_time_lower_bound} is key in obtaining the lower bounds on the tails of the  invariant distribution of $Z$ in the positive-recurrent regime:
see the proof of Theorem~\ref{thm:invariant_distributon}
in Section~\ref{subsection:main_proofs} below. The drift conditions in Lemma~\ref{Lem:drift_conditions}, proved in the present section, are used for establishing finite moments (and hence upper bounds on the tails) of the invariant distribution and the rate of convergence to stationarity of $Z$
in the positive-recurrent regime (see the proof of Proposition~\ref{prop:upper_bounds} in Section~\ref{subsection:main_proofs} below). The common theme of the proofs of the results in this section is that they are all based on the supermartingale property of certain processes. 

Recall the definitions of the return time $\varsigma_r$ (for $r\in(0,\infty)$) in~\eqref{eq::varsigma} and  of the critical exponent $m_c=(1-\beta/\beta_c)/2$ in~\eqref{eq::m0}.

\begin{prop}
\label{prop:return_time_upper_bound}
Suppose that~\aref{ass:domain2}, \aref{ass:covariance2}, \aref{ass:vector2} hold and $\beta < \beta_c$. Then, for every $p \in(0,m_c)$, there exists $x_0>0$ such that for all $x_1 \in[x_0,\infty)$ and $z=(x,y)\in\cD$  there exists a constant $C \in \RP$ (depending only on $x$ and $p$) for which
$
\E_z[\varsigma_{x_1}^p]\leq C.
$
\end{prop}

\begin{proof}
Pick $\gamma \in (0,(1-\beta/\beta_c)/2)$.  Then we can choose $w \in  (-\infty,1-\beta s_0/c_0)\setminus\{0\}$, such that $\sigma_1^2(2\gamma -1) + \sigma_2^2(1-w)<0$.
Thus, for any $\eps\in(0,-(\sigma_1^2(2\gamma -1) + \sigma_2^2(1-w)))$,
Lemma~\ref{lem3.1} implies the existence of $\ell_0 \in(0,\infty)$,  such that for all $z=(x,y)\in\cD$ with $x > 2^{-1/|w|}\ell_0 -k_w$ (the constant $k_w$
is given above display~\eqref{eq:Lyapunov_polynomial}), 
the function $f_{w,2\gamma}$ defined in~\eqref{eq:Lyapunov_polynomial} 
satisfies
\begin{equation}
\label{eq:negative_drift_tail_return_time}
\Delta_\Sigma f_{w,2\gamma}(z) +\eps f_{w,2\gamma-2}(z) < 0  \quad \text{and} \quad
\langle \nabla f_{w,2\gamma}(z),\phi(z)\rangle < 0.
\end{equation}

Having chosen the parameters $\gamma$, $w$, and $\eps$, consider the process $\kappa = f_{w,1}(Z)$.
The key step in the proof of the proposition consists of the application of~\cite[Thm~2.1]{menshikov1996passage} to deduce that,  for any $z\in\cD$ with $\kappa_0=f_{w,1}(z)> \ell_0$, 
the return time $\lambda_{\ell_0}$ of $\kappa$
below the level $\ell_0$, defined in~\eqref{eq::lambda}, has finite $\gamma$-moment, i.e.,
$\E_{z}[\lambda_{\ell_0}^\gamma] < \infty$. 
This will hold by~\cite[Thm~2.1]{menshikov1996passage}
if we establish that the process 
$(\xi_t)_{t\in\RP}$, given by 
$\xi_t:=\kappa_{t \wedge \lambda_{\ell_0}}^{2\gamma} + \epsilon \int_0^{t \wedge \lambda_{\ell_0}} \kappa_u^{2\gamma - 2}\ud u$,
is a supermartingale. 

With this in mind, take an arbitrary $r\in({\ell_0},\infty)$ and consider the stopped process 
$(\xi_{t \wedge \rho_r})_{t\in\RP}$,
where the stopping time $\rho_r$, defined in~\eqref{eq::rho}, is the first time the process $\kappa$ reaches level $r$.
Since $0\leq \xi_{t \wedge \rho_r}\leq \max\{r^{2\gamma},f_{w,2\gamma}(z)\}+\eps t \max\{r^{2\gamma-2},{\ell_0}^{2\gamma-2}\} $
for all $t\in\RP$, we have $\E_z[\xi_{t \wedge \rho_r}]<\infty$.
Moreover, by It\^o's formula in~\eqref{eq::Ito}, the inequalities
in~\eqref{eq:negative_drift_tail_return_time}
and the fact that $\kappa_t\geq {\ell_0}$ for all $t\in\RP$,
for any two times $0\leq s \leq t < \infty$ we have
$\xi_{t \wedge \rho_r}-\xi_{s \wedge \rho_r}-M_{{t \wedge \lambda_{\ell_0}\wedge \rho_r}}+M_{{s \wedge \lambda_{\ell_0}\wedge \rho_r}}\leq 0$ a.s.
Since the local martingale $M$ has, by~\eqref{eq::Ito_QV}
 and~\aref{ass:covariance2}, bounded quadratic variation for each 
$t\in\RP$ with probability one,
 the stopped process $(\xi_{t \wedge \rho_r})_{t\in\RP}$
 is a supermartingale for any $r\in(\ell_0,\infty)$.
 By Theorem~\ref{thm:non_explosion_moments}
 we have $\lim_{r\to\infty}\rho_r=\infty$.
 Thus $\xi_t=\liminf_{r\to\infty}\xi_{t\wedge\rho_r}$ for all $t\in\RP$.
 Since the process 
 $(\xi_{t})_{t\in\RP}$ is non-negative, the conditional Fatou lemma implies that it is a supermartingale.
 Thus, we may apply~\cite[Thm~2.1]{menshikov1996passage} to deduce that for any $p\in(0,\gamma)$ and $z\in\cD$ there exist $C_1,C_2\in(0,\infty)$, such that
 $\E_{z}[\lambda_{\ell_0}^p] \leq C_1f_{w,2\gamma}(z)+C_2\leq C_12^{2\gamma/|w|}(x+k_w)^{2\gamma}+C_2$, where the second inequality follows from~\eqref{eq:bound_Lyapunov_f}.

Recall that $f_{w,1}(Z)=\kappa$. Hence the second inequality in~\eqref{eq:bound_Lyapunov_f} implies that, for $x_0:= 2^{1/|w|}\ell_0-k_w$, we have $\varsigma_{x_0}\leq\lambda_{\ell_0}$ $\P_z$-a.s. for every $z\in\cD$. For every $p\in(0,m_c)$ and $x\in(0,\infty)$, define $C:= C_12^{2\gamma/|w|}(x+k_w)^{2\gamma}+C_2$, where $\gamma\in(p,m_c)$. Thus, we have $\E_z[\varsigma_{x_0}^p]\leq \E_z[\lambda_{\ell_0}^p] \leq C$ for any $z=(x,y)\in\cD$. Moreover, for any 
$x_1\in(x_0,\infty)$, we have 
 $\P_z(\varsigma_{x_0}\geq \varsigma_{x_1})=1$ for every $z\in\cD$, 
 implying $\E_z[\varsigma_{x_1}^p]\leq C$.
\end{proof}

The next proposition provides \textit{lower bounds} on the tails of return times and related path functionals in the recurrent case. This result is a key ingredient in the proof of Theorems~\ref{thm:return_times}(b), as well as in the proof of lower bounds in Theorem~\ref{thm:invariant_distributon}.
The proof of Proposition~\ref{prop:return_time_lower_bound}
is based on an application of Lemma~\ref{lem:return_times} of Section~\ref{subsec:lower_bounds_semimartingale_return_Times} above.
Recall $m_c=(1-\beta/\beta_c)/2$ defined in~\eqref{eq::m0}.

\begin{prop}
\label{prop:return_time_lower_bound}
Suppose that~\aref{ass:domain2}, \aref{ass:covariance2}, \aref{ass:vector2} hold with $\beta < \beta_c$. Then, for every $p\in(2m_c,\infty)$, there exist $x_0\in(0,\infty)$ and constants $c_1,c_2\in(1,\infty)$
such that,
for every 
non-decreasing measurable function $h:\RP\to\RP$,
$q\in(0,1)$,
$x_1\in(x_0,\infty)$
and $z=(x,y)\in\cD\cap (c_1x_1+c_2,\infty)\times\R^d$
%and $\eps\in(0,C^{-1}q(1-q)]$,
we have 
%$q\in(0,1)$ there exist constants $x_0\in(0,\infty)$ and %$x_1\in[x_0,\infty)$, $\alpha \in \RP$, and $z=(x,y)\in \cD$ %satisfying $x>c_1x_1 +c_2$,
$$
\P_z\left(\int_0^{\varsigma_{x_1}} h\left(c_1(X_s + c_2)\right)\ud s \geq \eps r^2h(r)\right) \geq q\min\{(c_1^{-p}(x-c_2)^p-x_1^p)(1-q)^pr^{-p},1\},
$$
for all $r\in(c_1x_1+c_2,\infty)$ and all sufficiently small $\eps>0$. In particular, $$\P_z(\varsigma_{x_1} \geq t) \geq q\min\{(c_1^{-p}(x-c_2)^p-x_1^p)\eps^{p/2}(1-q)^{p} t^{-p/2},1\},$$ for every $t\in(c_1x_1+c_2,\infty)$ and all  sufficiently small $\eps>0$.
\end{prop}

\begin{rem}
\label{rem:proof_of_return_time_thm}
Propositions~\ref{prop:return_time_upper_bound} and~\ref{prop:return_time_lower_bound} provide crucial estimates in the proof of Theorem~\ref{thm:return_times}(b) in Section~\ref{subsubsec:proof_of_thm_return_times} below.
The only assertion of 
Theorem~\ref{thm:return_times}(b) not contained in Propositions~\ref{prop:return_time_upper_bound} and~\ref{prop:return_time_lower_bound}
is that the bounds in 
Propositions~\ref{prop:return_time_upper_bound} and~\ref{prop:return_time_lower_bound}
actually holds for all $x_0\in(0,\infty)$ and $z\in\cD\cap (x_0,\infty)\times\R^d$
and not only for large enough $x_0$ and the starting points $z$ sufficiently far (in the $x$-direction) from $x_0$. This generalisation requires uniform ellipticity and will be established in Section~\ref{subsubsec:proof_of_thm_return_times}.
\end{rem}

\begin{proof}[Proof of Proposition~\ref{prop:return_time_lower_bound}]
 Pick $p\in (1-\beta/\beta_c,\infty)$ and note $p>0$. Then there exists $w \in(1-\beta s_0/c_0,\infty)\setminus\{0\}$, such that $p  > 1 - \sigma_2^2/\sigma_1^2(1-w)$. Lemma~\ref{lem3.1} implies that there exist $\ell_0> 0$ and a constant $C \in \RP$ such that, for all $z=(x,y)\in\cD$ with $x > 2^{-1/|w|}\ell_0 -k_{w}$ (the constant $k_w$
is defined above display~\eqref{eq:Lyapunov_polynomial}),
we have
\begin{align}
\label{eq:sub_mart_return_times_lower_bound}
\Delta_\Sigma f_{w,p}(z)  > 0 & \quad\text{and}\quad
\langle \nabla f_{w,p}(z),\phi(z)\rangle > 0,\\
\Delta_\Sigma f_{w,-2}(z)  \leq Cf_{w,-4}(z)  &  \quad \text{and}\quad
\langle \nabla f_{w_1,-2}(z),\phi(z)\rangle < 0.
\label{eq:sup_mart_return_times_lower_bound}
\end{align}

Define $\kappa := f_{w,1}(Z)$
and recall the return time $\lambda_{\ell_0}$ of $\kappa$
below the level $\ell_0$, defined in~\eqref{eq::lambda}.
By~\eqref{eq:bound_Lyapunov_f}, on the event $\{\kappa_t=f(Z_t)\geq \ell_0\}$, the first coordinate $X_t$  of $Z_t$ satisfies
$X_t > 2^{-1/|w|}\ell_0 -k_{w}$.
It\^o's formula in~\eqref{eq::Ito} 
applied to  $\kappa = f_{w,1}(Z)$ and the inequalities in~\eqref{eq:sub_mart_return_times_lower_bound}
imply that, for all $r>\ell\geq \ell_0$ and $0\leq s \leq t < \infty$, we have 
$
\kappa_{t\wedge \lambda_\ell\wedge\rho_r }^{p} - M_{t\wedge \lambda_\ell\wedge \rho_r}  \geq \kappa_{s \wedge \lambda_\ell\wedge\rho_r }^{p} - M_{s \wedge \lambda_\ell\wedge \rho_r}$.
 The process $M$ is a true martingale, since its quadratic variation is bounded by~\eqref{eq::Ito_QV} and~\aref{ass:covariance2}. Thus,
 since  $\E_z[\kappa_{t\wedge \lambda_\ell\wedge\rho_r}^p] \leq \max\{r^p, f_{w,p}(z)\}$, the process $(\kappa_{t\wedge \lambda_\ell\wedge\rho_r}^p)_{t\in\RP}$ is an $(\cF_{t})$-submartingale.

 Pick $r\in(\ell_0,\infty)$, $q\in(0,1)$ and set $r_q := r/(1-q)$. Recall that $f_{w,1}(Z_t) = \kappa$ and define the process $\xi=(\xi_t)_{t\in\RP}$ by $\xi_t :=  \kappa_{{(\rho_{r_q} + t})\wedge \lambda_{r,\rho_{r_q}}}^{-2} - C\int_{\rho_{r_q}}^{t \wedge \lambda_{r,\rho_{r_q}}}\kappa_u^{-4} \ud u$,
 where the constant $C>0$ is as in~\eqref{eq:sup_mart_return_times_lower_bound}.
 By It\^o's formula in~\eqref{eq::Ito} and the inequalities in~\eqref{eq:sup_mart_return_times_lower_bound},  for every  $0\leq s \leq t < \infty$, we have 
 \begin{align*}
 \xi_t - \xi_s - (M_{{(\rho_{r_q} + t})\wedge\lambda_{r,\rho_{r_q}}}- M_{{(\rho_{r_q} + s})\wedge \lambda_{r,\rho_{r_q}}}) \leq 0\quad \text{a.s.}
 \end{align*}
 Since the process $\xi$ is bounded, an analogous argument to the one in the previous paragraph implies that $\xi$ is a supermartingale.
 %Moreover, by~\eqref{eq::Ito_QV} and \aref{ass:covariance2} the %process $(M_{{(\rho_{r_q} + t})\wedge %\lambda_{r,\rho_{r_q}}})_{t\in\RP}$ has an integrable quadratic %variation, making it a true martingale, and $\E[\xi_0] \leq %(1-q)^2r^{-2}$, which implies integrability of the process.  

 We have now proved that  Assumptions (a) and (b) of Lemma~\ref{lem:return_times} are satisfied. 
 Moreover, by 
 Lemma~\ref{lem:non_cofinment_f_F},
  we have $\limsup_{t\to\infty}\kappa_t=\infty$ a.s.
Hence, by Lemma~\ref{lem:return_times}, for any 
% Since for any $r > \ell \geq \ell_0$ the process $(\kappa_{t\wedge %\lambda_{\ell}\wedge \rho_r})_{t\in\RP}$ is a submartingale and the %process $(\xi_t)_{t\in\RP}$ is a supermartingale, %Lemma~\ref{lem:return_times} implies that for every $q\in(0,1)$,
 $z\in\cD$ and $\ell>\ell_0$, satisfying $f_{w,1}(z) > \ell$, any non-decreasing measurable function $h:\RP\to\RP$ and all sufficiently small $\eps>0$, 
 we obtain
 \begin{equation}
 \label{eq:lower_bound_kappa}
\P_z\left(\int_0^{\lambda_\ell} h(f_{w,1}(Z_s)) \ud s\geq \eps r^2h(r)\right) \geq q\min\{(f_{w,1}(z)^p-\ell^p)(1-q)^{p}r^{-p},1\},
\quad\text{$r\in(\ell,\infty)$.}
 \end{equation}
 %for all $r\in(\ell,\infty)$.

%The inequalities %in~\eqref{eq:bound_Lyapunov_f} imply that for %$z=(x,y)\in\cD$, satisfying $f_{w,1}(z)=\ell$, %we have $2^{-1/|w|}\ell-k_w\leq x\leq %2^{1/|w|}\ell-k_w$. 
Define $x_0:=2^{-1/|w|}\ell_0-k_w$,
$c_1:=2^{1/|w|}$ 
and
$c_2:=k_w$. Then, for any $x_1\in(x_0,\infty)$
there exists $\ell>\ell_0$,
such that 
$x_1=2^{-1/|w|}\ell-k_w$.
The second inequality  in~\eqref{eq:bound_Lyapunov_f} implies
$ \varsigma_{x_1}\geq \lambda_\ell$
and 
$ h(f_{w,1}(Z_s))\leq h\left(2^{1/|w|}(X_s+k_w)\right)$
for all non-decreasing measurable functions $h:\RP\to\RP$ and $s\in[0,\lambda_\ell]$.
Thus the inequality 
$\int_0^{\lambda_\ell} h(f_{w,1}(Z_s)) \ud s\leq \int_0^{\varsigma_{x_0}} h\left(2^{1/|w|}(X_s+k_w)\right) \ud s$
holds, implying  by~\eqref{eq:lower_bound_kappa} the inequality in the proposition
for all $r\in(2^{1/|w|}x_0+k_w,\infty)$.
The special case follows by choosing $h\equiv1$.
\end{proof}

The next result establishes a drift condition (in the positive-recurrent case), used in the proofs of the upper bounds 
of Theorem~\ref{thm:invariant_distributon}
concerning the finite moments of the invariant distribution $\pi$ of $Z$ and the total variation distance between $\P_z(Z_t\in\cdot)$ and $\pi$.
%Drift conditions provide an upper bound for tail of %the invariant distribution and the rate of %convergence.
The proof of Lemma~\ref{Lem:drift_conditions} is analogous to the proof of Proposition~\ref{prop:return_time_upper_bound}.
For any $r\in\RP$, denote
\begin{equation}
\label{eq:D_r}
    \cD^{(r)} := \cD\cap [0,r]\times \R^d.
\end{equation}

\begin{lem}
\label{Lem:drift_conditions}
Suppose that~\aref{ass:covariance2}, \aref{ass:vector2} and \aref{ass:domain2} hold with $\beta < -\beta_c$. Then, for any  $\gamma \in(0,1-\beta/\beta_c)$, there exist parameters $w \in(-\infty,1-\beta s_0/c_0)\setminus\{0\}$ and $x_0,x_1,k\in\RP$, defining the function $F_{w,\gamma}$ in~\eqref{eq:F_w,gamma}, and
$x_2\in\RP$, $C_1,C_2\in \RP$, such that the process $\xi = (\xi_t)_{t\in\RP}$,
\begin{equation}
\label{eq::F_w_gamma_supermart}
\xi_t := F_{w,\gamma}(Z_t) + C_1\int_0^t  F_{w,\gamma-2}(Z_u)\ud u - C_2\int_0^t\mathbbm{1}_{\cD^{(x_2)}}(Z_u)\ud u,
\end{equation}
is an $(\cF_t)$-supermartingale.
%i.e.,
%\begin{equation}
%\label{eq::Ly2}
%    \E_z[F_{w,\gamma}(Z_s)] + C\E_z[\int_0^s F_{w,\gamma-2}(Z_u)du] \leq F_{w,\gamma}(x) + b\E_z[\int_0^s\mathbbm{1}_{D_{x_2}}(Z_u)du].
%\end{equation}
\end{lem}

Note that the process  $F_{w,\gamma}(Z)$ in Lemma~\ref{Lem:drift_conditions} gets a non-positive push 
(by Lemma~\ref{lem:F_w,gamma})
when $Z$ hits the boundary $\partial\cD$. The constant $C_1$ (resp.~$C_2$)  needs to be 
sufficiently small (resp.\ large) for the process $\xi$ to have non-positive drift in the entire interior of $\cD$.

\begin{proof}[Proof of Lemma~\ref{Lem:drift_conditions}]
Pick %$\eps \in (0,1-\beta/\beta_c)$,
$\gamma\in(0,1-\beta/\beta_c)$ 
and note that 
$1+\sigma_1^2/\sigma_2^2(\gamma-1)<1-\beta s_0/c_0$, since
$\beta_c=c_0\sigma_1^2/(s_0\sigma_2^2)$ by definition~\eqref{eq:beta_c}.
%(1-\beta s_0/c_0-\eps\sigma_1^2/(2\sigma_2^2),1-\bet%a s_0/c_0)\setminus\{0\}$.
Pick 
$w\in(1+\sigma_1^2/\sigma_2^2(\gamma-1),1-\beta s_0/c_0)\setminus\{0\}$
and note that 
$\gamma(\beta s_0/c_0-1+w)< 0$. Lemma~\ref{lem:F_w,gamma} implies that there exist %for some 
$x_0,x_1,k\in(0,\infty)$ such that the function $F_{w,\gamma}$ defined in~\eqref{eq:F_w,gamma} satisfies  $\langle \nabla F_{w,\gamma}(z),\phi(z)\rangle < 0$ for all $z\in\partial\cD$.

By~\eqref{eq:F_w,gamma} we have  $F_{w,\gamma}(z)=f_{w,\gamma}(z)$ on $z\in\cD\cap(x_1,\infty)\times\R^d$
and~\eqref{eq:Sigma_Laplacian_f} yields
$$
\Delta_\Sigma F_{w,\gamma}(z) = \gamma F_{w,1}(z)^{\gamma-2}(\sigma_1^2(\gamma-1)+\sigma_2^2(1-w)  +o_\cD(1))\quad\text{as $x\to\infty$,}
%\leq \gamma F_{w,1}(z)^{\gamma-2}(z)(f(z)-\eps/2)
$$
where $o_\cD(1)$ is defined after Assumption~\aref{ass:vector2}. 
%for any $z=(x,y)\in\cD$ with $x > x_1$, where $f: %\cD \to \R$ satisfies $\sup_{y:(x,y)\in\cD} |f(x,y)| %\to 0$ as $x \to \infty$. 
Since $\sigma_1^2(\gamma-1)+\sigma_2^2(1-w)<0$, there exists $x_2\in(x_1,\infty)$, 
such that for 
$C_1:=-(\sigma_1^2(\gamma-1)+\sigma_2^2(1-w))\gamma/4$
we have 
%such that $\sup_{y:(x,y)\in\cD}|f(x,y)| < \eps/4$ %for all $x\in[x_2,\infty)$, the following holds
$$
\frac{1}{2}\Delta_\Sigma F_{w,\gamma}(z) + C_1F_{w,\gamma-2}(z) \leq 0 \quad \text{on } z \in \cD\cap(x_2,\infty)\times\R^d.
$$
 Thus, since the functions $\Delta_\Sigma F_{w,\gamma}$ and $\Sigma$ are bounded on the compact set $\cD_{x_2}$,
 there exists $C_2\in\RP$
 such that 
\begin{equation}
\label{eq:F_w_gamma_drift_bound}
\frac{1}{2}\Delta_\Sigma F_{w,\gamma}(z) + C_1F_{w,\gamma-2}(z) \leq C_2\mathbbm{1}_{\cD^{(x_2)}}(z) \quad \text{ for all } z\in\cD.
\end{equation}

 Recall the definition of~$\xi$ in~\eqref{eq::F_w_gamma_supermart}
 and set $\kappa:=F_{w,1}(Z)$. 
 Note that by definition~\eqref{eq:F_w,gamma}, there exist $\delta_0>0$ such that  $\inf_{z\in\cD}F_{w,1}(z)>\delta_0$.
 For any $r>\max\{1,\delta_0\}$, the stopped process $(\xi_{t\wedge \rho_r})_{t\in\RP}$, where the stopping time $\rho_r$, defined in~\eqref{eq::rho}
 as the first time the process $\kappa$ crosses level $r$, satisfies 
 $-C_2 t \leq \E_z[\xi_{t \wedge \rho_r}] \leq  \max\{F_{w,\gamma}(z),r^\gamma\} +C_1 t\max\{ r^{\gamma-2},\delta_0^{\gamma-2}\}$ for all $t\in\RP$. Thus $\E_z | \xi_{t\wedge \rho_r} | <\infty$ for all $t\in\RP$ and $z\in\cD$. 
 Moreover, the inequality~\eqref{eq:F_w_gamma_drift_bound} and  It\^o's formula~\eqref{eq::Ito} applied to $F_{w,\gamma}(Z)$ imply that, for any $0 \leq s \leq t<\infty$, we have $\xi_{t\wedge \rho_r}-\xi_{s\wedge \rho_r} - (M_{t\wedge \rho_r} - M_{s\wedge \rho_r}) \leq 0$ a.s. Since, by~\eqref{eq::Ito_QV} and~\aref{ass:covariance2},
 $[M]_{t\wedge \rho_r}\leq C_0 t$ a.s.~for all $t\in\RP$ and some constant $C_0>0$,  $(\xi_{t\wedge \rho_r})_{t\in\RP}$ is a supermartingale for any $r\in\RP$. By Theorem~\ref{thm:non_explosion_moments} we have $\lim_{r\to\infty}\rho_r = \infty$. Thus, $\xi_s = \liminf_{r\to\infty}\xi_{s\wedge \rho_r}$ for all $s\in\RP$. Since $\xi_s\geq -C_2t$ for all $s\in[0,t]$, the conditional Fatou lemma implies that for any $0\leq s \leq t <\infty$, we have $$\E_z[\xi_t \vert\cF_s] = \E_z[\liminf_{r\to\infty}\xi_{t\wedge \rho_r}\vert \cF_s] \leq \liminf_{r\to\infty}\E_z[\xi_{t\wedge \rho_r}\vert \cF_s] \leq \xi_s,$$
 in addition we deduce the integrability of $\xi_t$ by choosing $s=0$ and noting that $\E_z[\xi_0]= F_{w,\gamma}(z)$, hence $\xi$ is an $(\cF_{t})$-supermartingale.
\end{proof}

\section{Feller continuity and irreducibility of the reflected process and applications}
%with applications to petite sets and Harris recurrence}
\label{subsection:Markovian_stability}
The existence of the invariant distribution of $Z$ requires positive recurrence (see definition preceding Theorem~\ref{thm:invariant_distributon} above). The study of its moments requires certain technical results established in the present section. Section~\ref{subsubsec:Feller_continuity} is dedicated to the proofs of Feller continuity and irreducibility of the process $Z$. In Section~\ref{subsubsec:continuous_markov_theory} we apply these two properties to prove that the reflected process $Z$
is Harris recurrent with an irreducible skeleton chain and that the set $\cD\cap [0,r]\times \R^d$ is petite for any $r \in(0,\infty)$.
Moreover, in Section~\ref{subsubsec:continuous_markov_theory} we will also show that every petite set for $Z$ is bounded. 

\subsection{Feller continuity and irreducibility of the reflected process}
\label{subsubsec:Feller_continuity}
In this section we prove
that the reflected process $Z$ is Feller-continuous (see Theorem~\ref{thm:Feller_continuity_of_Z} below) and that 
the $(d+1)$-dimensional Lebesgue measure 
$\lebm_{d+1}$
on 
the Borel $\sigma$-algebra $\cB(\cD)$ on $\cD$
is absolutely continuous with respect to 
its marginals of the reflected process at positive times (Proposition~\ref{prop:marginal_equivalent_to_Lebesgue} below).
We start with the latter.
%generated by the induced topology of the ambient %Euclidean space  $\R^{d+1}$.
%Using these two continuity properties, we will %show that $Z$ satisfies the properties mentioned %in the previous paragraph.

\begin{prop}
\label{prop:marginal_equivalent_to_Lebesgue}
Let~\aref{ass:domain2}, \aref{ass:covariance2}, \aref{ass:vector2} hold. Then, for any $z\in\cD$ and $t\in(0,\infty)$ and any $A\in \cB(\cD)$, such that 
$\lebm_{d+1}(A)>0$, we have $\P_z(Z_t \in A)>0$. 
\end{prop}

The proofs of Proposition~\ref{prop:marginal_equivalent_to_Lebesgue} and Theorem~\ref{thm:Feller_continuity_of_Z} below require the following lemma.
For small $h>0$, define a ``thin'' neighborhood of $\partial\cD$ in $\cD$ by $\cD_{h} := \{z\in\cD: \exists z'\in \partial\cD \text{ such that } \|z-z'\|_{d+1}<h\}$.

\begin{lem}
\label{lem:G}
Let~\aref{ass:domain1}, \aref{ass:covariance1}, \aref{ass:vector1} hold.
Then there exists a functions $G:\cD\to\RP$, twice-differentiable on a neighbourhood of $\cD\subset\R^{d+1}$ 
and strictly positive on the open set $\cD\setminus \partial\cD$. Moreover, the function $g(z) := G(z)^2$ satisfies $\nabla g(z) = 0$  for all $z\in\partial \cD$ and for any $r>0$ there exist $h_r>0$, $\delta > 0$,  such that the following hold
\begin{itemize}
\item[(a)] 
$\|\Sigma^{1/2}(z)\nabla G(z)\|_{d+1}^2> \delta$  for all $z\in\cD_{h_r}\cap[0,r]\times \R^d$;
\item[(b)]
$\langle \phi(z),\nabla G(z)\rangle >  \delta$ for all $z\in\partial \cD \cap[0,r]\times \R^d$.
\end{itemize}
\end{lem}

\begin{rem}
Recall that Assumptions~\aref{ass:domain2}, \aref{ass:covariance2}, \aref{ass:vector2} imply Assumptions~\aref{ass:domain1}, \aref{ass:covariance1}, \aref{ass:vector1}.
\end{rem}

\begin{proof}[Proof of Lemma~\ref{lem:G}]
Extend the function 
$G:\cD\to\RP$, defined by $G(z):=b(x)^2-\|y\|_d^2$
for any $z=(x,y)\in\cD$, to a 
$C^2$-function on a neighbourhood of $\cD$ in $\R^{d+1}$.  Hence 
$\|\nabla G(z)\|_{d+1}^2=4(b(x)^2b'(x)^2+\|y\|_d^2)$
for all $z=(x,y)\in \cD$ and $\nabla g(z) = 2G(z)\nabla G(z) = 0$ for all $z\in \partial \cD$, since $G(z) = 0$ on~$\partial\cD$. 
Recall that~\aref{ass:domain1} yields $\liminf_{x\to0}b(x)b'(x)>0$. Thus, for any $r>0$, there exists a sufficiently small $h_r>0$, such that $0<\inf_{z\in\cD_{h_r}\cap[0,r]\times \R^d}\|\nabla G(z)\|_{d+1}^2=:\delta_{h_r}$
(we restrict to $x\in[0,r]$ because of functions $b$ with $\beta \leq 0$, e.g. example in Lemma~\ref{lem:oscilating_domain}).
Moreover, by~\aref{ass:covariance1}, there exists $\delta_\Sigma>0$
such that 
$\|\Sigma^{1/2}(z)\nabla G(z)\|_{d+1}^2=\langle \Sigma(z)\nabla G(z),\nabla G(z)\rangle\geq \delta_\Sigma\|\nabla G(z)\|_{d+1}^2\geq \delta_\Sigma\cdot \delta_{h_r}>0$
for all $z\in\cD_{h_r}\cap[0,r]\times \R^d$, implying~(a).

Note that gradient $\nabla G(z)$ for any $z\in\partial\cD$ equals 
$n(z)\|\nabla G(z)\|_{d+1}$, where $n(z)$ is the inwards-pointing unit normal vector to $\partial \cD$ at $z$. Hence, by~\aref{ass:vector1}, there exists a $\delta_\phi>0$ such that $\langle \phi(z),\nabla G(z)\rangle \geq \delta_{h_r}^{1/2} \langle \phi(z), n(z)\rangle > \delta_\phi \cdot\delta_{h_r}^{1/2}>0$ for all $z\in \partial\cD \cap [0,r]\times \R^d$.
\end{proof}

\begin{proof}[Proof of Proposition~\ref{prop:marginal_equivalent_to_Lebesgue}]
We start by proving that $Z$ spends no time at the boundary. 

\noindent \textbf{Claim 0.} The equality $\int_0^\infty \1{Z_t\in\partial \cD}\ud t=0$ 
holds $\P_z$-a.s. 
for any starting point $z\in\cD$.

\noindent \underline{\textit{Proof of Claim 0.}} Let $G$ be a function whose existence is guaranteed by Lemma~\ref{lem:G}. Define the non-negative continuous semimartingale $\xi=(\xi_t)_{t\in\RP}$, 
$\xi_t:=G(Z_t)$, denote its local time field by $(L_t^u(\xi))_{t,u\in\RP}$
(see~\cite[Ch.~VI]{revuz2013continuous} for definition and properties)
and note that $\xi_t=0$
if and only if $Z_t\in\partial\cD$, since $G>0$ on $\cD\setminus\partial\cD$ by  Lemma~\ref{lem:G}.
Thus the local-time process $L$ in SDE~\eqref{eq::SDE} satisfies
$L^0(\xi)=L$.  
It\^o's formula in~\eqref{eq::Ito}, applied to $G(Z)$,
yields that the quadratic variation $[\xi]$ equals  $[M]$ given in~\eqref{eq::Ito_QV}. 
Pick any $t,r\in(0,\infty)$.
The occupation times formula in~\cite[Cor~VI.1.6]{revuz2013continuous} applied to the indicator  $u\mapsto \ind_{\{0\}}(u)$ of zero,
the representation of the quadratic variation $[M]$ in~\eqref{eq::Ito_QV},
the property in Lemma~\ref{lem:G}(a)
and the fact that $\xi_s=0$
is equivalent to 
 $Z_s\in\partial\cD$ for all $s\in\RP$
yield
$$
0=\int_{\RP}\ind_{\{0\}}(u)L_{t\wedge \varrho_r}^u(\xi)\ud u = \int_0^{t\wedge \varrho_r}\ind_{\{0\}}(\xi_s)\ud [M]_s\geq \delta \int_0^{t\wedge \varrho_r} \1{Z_s\in\partial\cD}\ud s\geq0,
$$
implying $0=\int_0^{t\wedge \varrho_r} \1{Z_s\in\partial\cD}\ud s$. 
Since $\varrho_r$, given in~\eqref{eq::varrho}, satisfies $\lim_{r\to\infty}\varrho_r=\infty$ a.s. (by Theorem~\ref{thm:non_explosion_moments})
and $t>0$ is arbitrary, our claim follows.

Fubini's theorem and Claim 0 yield 
$\int_0^\infty \P_z(Z_t\in\partial\cD)\ud t=\E_z\int_0^\infty \1{Z_t\in\partial \cD}\ud t=0$
for any $z\in\cD$
(since $Z$ is continuous, it is progressively measurable, implying the  various integrals are well defined and measurable).
In particular, denoting $\ocD := \cD \setminus \partial \cD$, for any $z\in\cD$ and $t>0$, it holds that $\int_0^t \P_z(Z_s\in \ocD)\ud s = t$.
%and by $\lebm$ the $(d+1)$-dimensional Lebesgue measure on $\cB(\cD)$. 
%The following claim holds.

\noindent \textbf{Claim 1.} For every $z\in\cD$, $t>0$ and $A\in\cB(\ocD)$ the following holds: if $\lebm_{d+1}(A) > 0$ then $\int_0^t \P_z(Z_s\in A)\ud s > 0$.

In order to prove Claim 1, we need Claims 2 and 3 below. For $z\in\cD$ and $h>0$, define the open ball in $\cD$ by $B(z,h) := \{ z'\in\cD: \|z-z'\|_{d+1} < h\}$.

\noindent \textbf{Claim 2.} Pick any $z\in\cD$, $s\in\RP$ and any ball $B(z',h)\subset \ocD$ 
and
$A\in\cB(\ocD)$.
The inequalities 
$\P_{z}(Z_s\in B(z',h))>0$ and 
$\lebm_{d+1}(B(z',h)\cap A)>0$
imply
$\P_{z}(Z_v\in B(z',h)\cap A) >0$
for all $v\in(s,s+h^2)$.

\noindent \underline{\textit{Proof of Claim 2.}} Since $B(z',h)\subset \ocD$,
the stopping time $\tau_{\partial B(z',h)}:=\inf\{t\in\RP: Z_t\notin B(z',h)\}$
is strictly positive $\P_{z''}$-a.s. for all 
$z''\in B(z',h)$.
Moreover, the process $Z$ on the stochastic interval $[0,\tau_{\partial B(z',h)})$, started at any $z''\in B(z',h)$, coincides with a uniformly elliptic diffusion on $\R^{d+1}$, stopped upon exiting the ball $B(z',h)$.
Thus,~\cite[Thm~II.1.3]{stroock1988diffusion} is applicable and, together with the strong Markov property of $Z$~\cite[Thm A.1]{menshikov2022reflecting}, yields the claim.

\noindent \textbf{Claim 3.}
For any $z\in\cD$, $t>0$ and $A_0\in\cB(\ocD)$, such that $\lebm_{d+1}(A_0)>0$,
there exist  $z_0\in \ocD$, $h_0>0$
and $s\in(0,t)$ satisfying
$B(z_0,h_0)\subset \ocD$,  $\lebm_{d+1}(A_0 \cap B(z_0,h_0)) > 0$ and $\P_z(Z_s \in B(z_0,h_0)) > 0$.

\noindent \underline{\textit{Proof of Claim 3.}} Since $\int_0^t \P_z(Z_s \in \ocD) \ud s = t$, there exist $s < t$, $z'\in\ocD$ and $h>0$, such that $B(z',h)\subset \ocD$ and  $\P_z(Z_s\in B(z',h)) > 0$.
Moreover, the assumption $\lebm_{d+1}(A_0) > 0$ implies that there exists a ball $B(z_0,h_0) \subset \ocD$ such that $\lebm_{d+1}(B(z_0,h_0) \cap A_0) > 0$.

It remains to prove that for some $s\in(0,t)$ we have $\P_z(Z_s \in B(z_0,h_0)) > 0$.
If $B(z_0,h_0)\cap B(z',h)\neq \emptyset$, then, since
$\lebm_{d+1}(B(z_0,h_0) \cap B(z',h))>0$,
by Claim~2 applied with 
$A:=B(z_0,h_0)$
and 
$\P_z(Z_s\in B(z',h))>0$,
there exist $v\in(s,t)$
with
$\P_z(Z_v\in B(z_0,h_0))>0$.
If $B(z_0,h_0)\cap B(z',h)=\emptyset$,
then there exists a  sequence of $n\in\N$ balls
$B(z_i,h_i)\subset \ocD$, where  $i\in\{1,\ldots,n\}$,
 such that $z_n=z_0$, $h_n=h_0$ and 
$z_1=z'$, $h_1=h$  and $B(z_i,h_i)\cap B(z_{i+1},h_{i+1})\neq \emptyset $ for all $i\in\{1,\ldots,n-1\}$.
Since 
$\P_z(Z_{s}\in B(z_1,h_1))>0$,
by~Claim 2 (applied with $A:=B(z_2,h_2)$),
there exists time $v_1\in(s,t)$, such that 
$\P_z(Z_{v_1}\in B(z_2,h_2))>0$.
The Markov property at $v_1$ and~Claim 2
imply the existence of $v_2\in(v_1,t)$
such that 
$\P_z(Z_{v_2}\in B(z_3,h_3))>0$.
Construct inductively the increasing sequence $v_1,v_2,\ldots,v_{n-1}\in(0,t)$, set $s:=v_{n-1}$
and note $\P_z(Z_{v_{n-1}}\in B(z_n,h_n))>0$, implying  Claim 3.

\noindent \underline{\textit{Proof of Claim 1.}} Assume that Claim~1 does not hold. More precisely, there exist $t>0$, $z\in\cD$ and $A_0 \in \cB(\ocD)$, such that $\lebm_{d+1}(A_0)>0$ and $\int_0^t \P_z(Z_v\in A_0)\ud v = 0$.
%The process $Z$ is a $\P_z$-a.s. continuous, strong
%solution of SDE~\eqref{eq::SDE} and thus progressively measurable. By Fubini's theorem, 
%$\int_0^t \P_z(Z_s \in \cdot)\ud s$ is a measure on $\cB(\ocD)$.
%In order to obtain a contradiction, we require the following claim, proved below.
By Claim~3 there exist a ball $B(z_0,h_0)$ in $\ocD$ and $s\in(0,t)$ such that 
$\lebm_{d+1}(A_0 \cap B(z_0,h_0)) > 0$ and $\P_z(Z_s \in B(z_0,h_0)) > 0$.
Claim 2 (applied with $z':=z_0$, $h:=h_0$ and $A:=A_0$) yields the contradiction:
$0=\int_0^t\P_z(Z_v\in A_0)\ud v \geq \int_s^{t \wedge (s+h_0^2)}\P_z(Z_v\in B(z_0,h_0)\cap A_0) \ud v>0$.

%Hence, for every $z\in\cD$, $t > 0$ and %$A\in\cB(\ocD)$ with $\lebm_{d+1}(A) > 0$ it holds %that $\int_0^t \P_z(Z_s\in A)\ud s > 0$. 

To conclude the proof of the proposition, we strengthen Claim 1. 
Suppose there exist $z\in\cD$, $t>0$ and $A\in\cB(\ocD)$ with $\lebm_{d+1}(A)>0$, such that $\P_z(Z_t\in A)=0$. Since there exists a ball $B(z',h)\in\cB(\ocD)$, such that $\lebm_{d+1}(A \cap B(z',h)) > 0$, 
by Claim~1 applied to $A \cap B(z',h)$ we have
 $\int_0^{h^2} \P_z(Z_v \in A \cap B(z',h))\ud v > 0$
for all $z\in\cD$.
We may assume $h^2\in(0,t)$.
Since  $Z$ is Markov, we have
$\int_{t-h^2}^t \P_z(Z_s \in A \cap B(z',h))\ud s > 0$
and hence $\P_z(Z_s \in A \cap B(z',h)) > 0$
for some $s\in(t-h^2,t)$. Thus $t\in(s, s+h^2)$.
By~Claim 2
we get 
%there exists $t-h^2 < s < t$, such that $\P_z(Z_s %\in A \cap B(z',h)) > 0$.  By %applying~\cite[Thm.~II.1.3]{stroock1988diffusion}, we %deduce that 
$\P_z(Z_t\in A) \geq \P_z(Z_t \in A \cap B(z',h))>0$, completing the proof of the proposition.
\end{proof}

\begin{rem}
    \label{rem:no_mass_on_boundary}
    By Claim~0 in the proof of Proposition~\ref{prop:marginal_equivalent_to_Lebesgue}  above, for any $z\in\cD$, the equality $\P_z(Z_t\in\partial \cD)=0$
    holds for Lebesgue almost every $t\in\RP$. 
    Note also that the proof of Claim~0
    uses only the occupation times formula for continuous semimartingales and basic properties of the solution of SDE~\eqref{eq::SDE}. 
\end{rem}

Domain $\cD$, defined in~\eqref{eq::domain}, with increasing boundary (e.g. $\beta>0$, see~\eqref{eq::beta} for definition) satisfies the conditions of~\cite{stroock1971diffusion}, which establishes Feller continuity for reflecting processes $Z$. However, as explained in~\cite[Rem.~2.3(f)]{menshikov2022reflecting}, 
the assumptions of~\cite{stroock1971diffusion} are not satisfied if the boundary function $b$ decreases to zero (e.g. $\beta<0$). In the case $\beta=0$, the domain $\cD$, may but need not, satisfy the assumptions of~\cite{stroock1971diffusion}, see example in Lemma~\ref{lem:oscilating_domain} below.
Since the case $\beta<0$ is when positive recurrence occurs, we develop a new approach to Feller continuity of $Z$, relying on the localisation of the process. This is  more involved than the standard approach in the literature (see e.g.~\cite{dai1990steady})
due to the difficulty of obtaining a global bound on the growth of the local time in the case 
$\beta<0$, which requires localisation.

\begin{thm}
\label{thm:Feller_continuity_of_Z}
Let~\aref{ass:domain2}, \aref{ass:covariance2}, \aref{ass:vector2} hold. For a
continuous bounded function $f:\cD\to\RP$,
$t\in\RP$,
and a convergent sequence $(z_n)_{n\in\N}$ in $\cD$ with limit $\lim_{n\to\infty}z_n= z_\infty\in\cD$, we have 
$$
\E_{z_n}[f(Z_t)] \to \E_{z_\infty}[f(Z_t)]\quad\text{as $n\to\infty$.}
$$
\end{thm}

The proof of Theorem~\ref{thm:Feller_continuity_of_Z}
requires Lemmas~\ref{lem:expected_growth_loc_time_norm}
and~\ref{lem:Feller_continuity_from_stopped_to_general}. Lemma~\ref{lem:expected_growth_loc_time_norm} provides growth estimates required for the proof of tightness via the Aldus's criterion~\cite[VI.~Thm 4.5]{Jacod2003}.

\begin{lem}
\label{lem:expected_growth_loc_time_norm}
Let~\aref{ass:domain2}, \aref{ass:covariance2}, \aref{ass:vector2} hold.
Fix any $r>0$, $T>0$ and $\theta\in(0,T]$. Then there exist positive constants $C_{i}$, for $i\in\{1,2,3,4\}$,
such that for any $z\in\cD$ and $(\cF_{t})$-stopping times $S_1,S_2$  ($(\cF_{t})_{t\in\RP}$ is the Brownian filtration  in~\eqref{eq::SDE}), satisfying $S_1\leq S_2\leq S_1+\theta \leq T$,
for any $
\eps>0$ the following hold:
\begin{itemize}
\item[(a)]$\P_z(L_{S_2\wedge \varrho_r}-L_{S_1\wedge \varrho_r}\geq \eps) \leq  (C_1 \theta + C_2\theta^{1/2})/\eps$; 
\item[(b)]$\P_z(\|Z_{S_2\wedge \varrho_r} - Z_{S_1\wedge \varrho_r}\|_{d+1}^2\geq \eps) \leq (C_3\theta+ C_4\theta^{1/2})/\eps$.
\end{itemize}   
Moreover, $\E_{z}[L_{T \wedge \varrho_r}]\leq C_1T+C_2T^{1/2}<\infty$ holds.
\end{lem}

\begin{proof}
Let $G$ be the function whose existence is guaranteed by Lemma~\ref{lem:G} and let $g(z) := G(z)^2$,  $z\in\cD$. Since the first and second derivatives of functions $G$ and $g$ are continuous on $\cD$, and $\Sigma$ is bounded by~\aref{ass:covariance2}, it follows that $\|\Sigma^{1/2}\nabla G\|_{d+1}^2$, $\|\Sigma^{1/2}\nabla g\|_{d+1}^2$, $\lvert  \Delta_\Sigma G\rvert$ and  $\lvert \Delta_\Sigma g\rvert$  are bounded on $\cD^{(r)}= \cD\cap [0,r]\times\R^d$ for any $r>0$.

Pick $r>0$,  $T>0$ and $\theta\in (0,T]$. Recall from Lemma~\ref{lem:G} that $\nabla g(z) = 0$ and $\langle \phi(z),\nabla G(z)\rangle >  \delta_r$ for all $z\in\partial \cD \cap[0,r]\times \R^d$ and some positive constant $\delta_r$.
Let the bounded stopping times $S_1,S_2$ be as in the statement of the lemma.
It\^o's formula~\eqref{eq::Ito} applied to the processes $g(Z)$ and $G(Z)$ on the stochastic interval $[S_1\wedge \varrho_r,S_2\wedge \varrho_r]$
(recall the definition of $\varrho_r$ in~\eqref{eq::varrho})
yields 
$$
g(Z_{S_2 \wedge \varrho_r}) = g(Z_{S_1 \wedge \varrho_r}) + M_{S_2\wedge \varrho_r}^g-M_{S_1\wedge \varrho_r}^g + \frac{1}{2}\int_{S_1\wedge \varrho_r}^{S_2\wedge \varrho_r} \Delta_\Sigma g(Z_u)\ud u,  \quad \P_z\text{-a.s.,}
$$
and 
$$
\delta_r (L_{S_2 \wedge \varrho_r}-L_{S_1\wedge \varrho_r})\leq  G(Z_{S_2\wedge \varrho_r}) - G(Z_{S_1\wedge \varrho_r}) - M_{S_2 \wedge \varrho_r}^G+ M_{S_1 \wedge \varrho_r}^G - \frac{1}{2}\int_{S_1\wedge \varrho_r}^{S_2\wedge \varrho_r} \Delta_\Sigma G(Z_u)\ud u,~\P_z\text{-a.s.}
$$
 for any $z\in\cD$.
 %satisfying $S_1 \leq S_2 \leq S_1+\theta \leq T$,
 %where 
 Note that $(M^g_{t\wedge \varrho_r})_{t\in\RP}$ and $(M^G_{t\wedge \varrho_r})_{t\in\RP}$ are true martingales by~\eqref{eq::Ito_QV}. The
optional sampling theorem 
at the bounded stopping time $S_2\wedge \varrho_r$
yields $\E_z[M_{S_2\wedge \varrho_r}^g-M_{S_1\wedge \varrho_r}^g|\cF_{S_1\wedge \varrho_r}]=0$, $\P_z$-a.s. and $\E_z[M_{S_2\wedge \varrho_r}^G-M_{S_1\wedge\varrho_r}^G|\cF_{S_1\wedge \varrho_r}]=0$, $\P_z$-a.s.
There exists a constant $C_1>0$ such that $|\Delta_\Sigma g(z)|\leq 2C_1$ for all $z\in\cD^{(r)}$, which implies that
\begin{equation}
\label{eq:bound_on_g_at_S2}
    \E_z[g(Z_{S_2\wedge \varrho_r}) \vert \cF_{S_1\wedge \varrho_r}] \leq  g(Z_{S_1\wedge \varrho_r}) + C_1\theta,\quad\text{$\P_z$-a.s.}
\end{equation}
holds for any $z\in\cD$. Moreover, there exists a constant $C_2$ such that $|\Delta_\Sigma G(z)|\leq 2C_2$ for $z\in\cD^{(r)}$.
Using the fact that $g(z) = G(z)^2$, and applying the optional sampling theorem at the bounded stopping time $S_2\wedge \varrho_r$ implies 
\begin{align}
\label{eq:local_time_growth}
\delta_r \E_z[L_{S_2 \wedge \varrho_r} - L_{S_1 \wedge \varrho_r}] &\leq \E_z\left[\E_z[G(Z_{S_2 \wedge \varrho_r})\vert \cF_{S_1 \wedge \varrho_r}] - G(Z_{S_1 \wedge \varrho_r})+ C_2 (S_2-S_1)\right] \\
\nonumber
&\leq \E_z\left[\E_z[g(Z_{S_2 \wedge \varrho_r})\vert \cF_{S_1 \wedge \varrho_r}]^{1/2} - G(Z_{S_1 \wedge \varrho_r})+ C_2 \theta \right] \\
\nonumber
&\leq \E_z\left[(C_1 \theta + g(Z_{S_1 \wedge \varrho_r}))^{1/2} - G(Z_{S_1 \wedge \varrho_r}) + C_2 \theta \right] \\
\nonumber
&\leq (C_1 \theta)^{1/2} + C_2 \theta,
\end{align}
 where we used the Cauchy-Schwarz inequality, the inequality in~\eqref{eq:bound_on_g_at_S2} and the triangle inequality in the 
 second, third and fourth inequalities in the display above, respectively. Application of Markov inequality to~\eqref{eq:local_time_growth} implies  $(a)$. Moreover, since  $L_0 = 0$, setting $S_1 = 0, S_2 = T, \theta = T$ in~\eqref{eq:local_time_growth}, we get $\E_{z}[L_{T \wedge \varrho_r}]\leq (C_1T)^{1/2}+C_2T<\infty$, as claimed.

To prove $(b)$, define the function $V_{z_0}:\cD \to \RP$ by $V_{z_0}(z) := \|z-z_0\|_{d+1}^2$ for a parameter $z_0\in\cD$. Pick $r>0$, $T>0$ and $\theta\in(0,T]$. 
Since first and second derivatives of $V_{z_0}(z)$ are continuous on $\cD$ in both variable $z$ and parameter $z_0$, it follows by \aref{ass:vector2} and \aref{ass:covariance2} that there exist constants $\tilde C_1,\tilde C_2,\tilde C_3$ such that  $\sup_{z_0\in\cD^{(r)}}\langle\nabla V_{z_0}(z),\phi(z)\rangle \leq \tilde C_1$, $\sup_{z_0\in\cD^{(r)}}|\Delta_\Sigma V_{z_0}(z)|\leq 2\tilde C_2$, and $\sup_{z_0\in\cD^{(r)}}\|\Sigma(z)^{1/2}\nabla V_{z_0}(z)\|_{d+1}^2\leq \tilde C_3$ hold for every $z\in\cD^{(r)}$.
It\^o's formula~\eqref{eq::Ito} applied to $V_{z_0}(Z)$, with $Z$ started at $z_0$, gives
\begin{align*}
V_{z_0}(Z_{S\wedge \varrho_r}) =  M_{S\wedge \varrho_r}^{V_{z_0}} + \int_{0}^{S\wedge \varrho_r} \langle \nabla V_{z_0}(Z_u),\phi(Z_u)\rangle\ud L_u + \frac{1}{2}\int_{0}^{S\wedge \varrho_r} \Delta_\Sigma V_{z_0}(Z_u)\ud u,\quad\text{$\P_{z_0}$-a.s.,}
\end{align*}
for any stopping time $S$ and any $z_0\in\cD$.
Moreover, by~\eqref{eq::Ito_QV} the process $(M^{V_{z_0}}_{t\wedge \varrho_r})_{t\in\RP}$ is a true martingale.  By~\eqref{eq:local_time_growth}, for $S$  satisfying $S\leq \theta$,  we have $\sup_{z_0\in\cD^{(r)}}\E_{z_0}[L_{S\wedge \varrho_r}]\leq(C_1 \theta)^{1/2} + C_2 \theta$. Hence, the optional sampling theorem at the bounded stopping time $S\wedge \varrho_r$ yields
$$
 \E_{z_0}[\|Z_{S\wedge \varrho_r}-Z_0\|_{d+1}^2] = \E_{z_0}[V_{z_0}(Z_{S\wedge \varrho_r})] \leq \tilde C_1 \E_{z_0}[L_{S\wedge \varrho_r}] + \tilde C_2 \theta\leq \tilde C_1((C_1\theta)^{1/2}+C_2\theta)+\tilde C_2 \theta,$$
for any $S\leq \theta$ and $z_0\in\cD^{(r)}$, where $C_1,C_2$ come from part (a) of this lemma.  
It follows from the strong Markov property that for any stopping times $S_1,S_2$ satisfying assumptions of the lemma, and any $z_0\in\cD$ we have
$$
\E_{z_0}[\|Z_{S_2\wedge \varrho_{r}}-Z_{S_1\wedge \varrho_{r}}\|_{d+1}^2] =\E_{z_0}[\E_{z_0}[\|Z_{S_2\wedge \varrho_r}-Z_{S_1\wedge \varrho_r}\|_{d+1}^2\vert \cF_{S_1\wedge \varrho_{r}}]] \leq \tilde C_1((C_1\theta)^{1/2}+C_2\theta)+\tilde C_2 \theta.
$$
Application of Markov inequality gives part $(b)$ of the lemma.
\end{proof}

\begin{lem}   \label{lem:Feller_continuity_from_stopped_to_general}
Let~\aref{ass:domain2}, \aref{ass:covariance2}, \aref{ass:vector2} hold. 
Suppose that for any $t>0$, any continuous bounded function $f:\cD\to\R$, and a convergent sequence $z_n\to z_\infty\in \cD$ (as $n\to\infty$), there exists a sequence $(r_k)_{k\in\N}$ satisfying $\lim_{k\to\infty}r_k = \infty$, such that for every $k\in\N$ the following holds:
\begin{equation}
\label{eq:stopped_feller}
\E_{z_n}[f(Z_{t\wedge \varrho_{r_k}})]\to\E_{z_\infty}[f(Z_{t\wedge \varrho_{r_k}})]\quad  %\&\quad \P_{z_n}(\varrho_{r_k}< t) \to \P_{z_\infty} (\varrho_{r_k}< t), 
\quad \text{as $n\to\infty$}.
\end{equation}
Then the process $Z$ is Feller continuous, that is, $\E_{z_n}[f(Z_t)] \to \E_{z_\infty}[f(Z_t)]$ as $n\to\infty$.
\end{lem}

\begin{proof}
Fix $t>0$. For any $r>0$, define the events $\cA_r := \{\varrho_r\leq t\}$ (and their complements  $\cA_r^c$). 

\noindent \textbf{Claim.} For any convergent sequence  $z_n \to z_\infty\in\cD$ and  $\eps>0$, there exists $r_0>0$ such that for all $r\in[r_0,\infty)$ we have $\sup_{n\in\N} \P_{z_n}(\cA_r)\leq \eps$ and $\P_{z_\infty}(\cA_r)\leq \eps$.

\noindent \textit{\underline{Proof of claim.}} Pick $w\in (-\infty,1-\beta s_0/c_0)$. Lemma~\ref{lem:F_w,gamma} guarantees the existence of constants $k,x_0,x_1\in\RP$ in the definition of the function $F_{w,1}$ such that $\langle \nabla F_{w,1}(z),\phi(z)\rangle \le0$ for all $z\in\partial\cD$. By~\eqref{eq:F_w,gamma}, we have $ F_{w,1}(z)=f_{w,1}(z)$ for all $z=(x,y)\in\cD \cap [x_1,\infty)\times \R^d$. Since $F_{w,1}$ is smooth on a neighbourhood of $\cD$, by~\eqref{eq:Sigma_Laplacian_f} and~\eqref{eq:grad_f_gamma_w_bound} in  Lemma~\ref{lem3.1},
the functions $z\mapsto|\Delta_\Sigma F_{w,1}(z)|$ and $z\mapsto\|\Sigma(z)^{1/2}\nabla F_{w,1}(z)\|_{d+1}^{2}$ are bounded on $\cD$. Hence,
there exist constants 
$\tilde C_1,\tilde C_2\in(0,\infty)$ satisfying $|\Delta_\Sigma F_{w,1}(z)|\leq 2\tilde C_1$ and $\|\Sigma(z)^{1/2}\nabla F_{w,1}(z)\|_{d+1}^{2}\leq \tilde C_2$ for all $z\in\cD$. 
Define the process $\kappa := 2^{1/\lvert w\rvert}F_{w,1}(Z)$,  and recall the definition of $\rho_r$ in~\eqref{eq::rho}. Since $F_{w,1}(z) = f_{w,1}(z)$ for $z\in\cD\cap [x_1,\infty)\times\R^d$,~\eqref{eq:bound_Lyapunov_f} implies that for $r\in[x_1,\infty)$ we have $\rho_r\leq \varrho_r$.
Moreover, application of It\^o's formula~\eqref{eq::Ito} implies that the process
$
\kappa_{t\wedge \rho_r} - C(t\wedge \rho_r)
$
is a supermartingale for $C=2^{1/|w|}\tilde C_1$, and any $r>0$, where the local martingale appearing in It\^o's formula is a true martingale by~\eqref{eq::Ito_QV}. Thus, applying Proposition~\ref{prop:maximal} (with $\xi=\kappa$ and a constant function $f\equiv C$) we infer that for any $r\geq x_1$ and any $z\in\cD$,
$$
\P_z(\cA_r)\leq \P_{z}(\rho_r \leq t) \leq r^{-1}(2^{1/|w|}F_{w,1}(z) 
+ C t).
$$
Since $(2^{1/|w|}\sup_{n\in\N}F_{w,1}(z_n)+Ct)$ is finite by the continuity of $F_{w,1}$, there exists $r_0$ as claimed.

To prove the lemma, pick an arbitrary continuous bounded function $f:\cD\to\R$, $\eps>0$, and a convergent sequence $z_n\to z_{\infty}$.  Let  the sequence $(r_k)_{k\in\N}$ tend to infinity and satisfy~\eqref{eq:stopped_feller} for every $k\in\N$.
 Since $f$ is bounded, the Claim above implies that there exists $k\in\N$ such that $\sup_{z\in\cD}\lvert f(z)\rvert(\P_{z_n}(\cA_{r_k})+\P_{z_\infty}(\cA_{r_k}))<\eps/3$ holds for every $n\in\N$. 
Thus, as $f(Z_t)=f(Z_{t\wedge \varrho_{r_k}})\mathbbm{1}(\cA_{r_k}^c)+f(Z_t)\mathbbm{1}(\cA_{r_k})$, for every $n\in\N$ we get
\begin{align}
\label{eq:Feller_local_to_global_bound}
\rvert \E_{z_n}[f(Z_t)]-\E_{z_\infty}[f(Z_t)]\rvert&\leq\rvert \E_{z_n}[f(Z_{t\wedge \varrho_{r_k}})\mathbbm{1}(\cA_{r_k}^c)]-\E_{z_\infty}[f(Z_{t\wedge \varrho_{r_k}})\mathbbm{1}(\cA_{r_k}^c)]\rvert +
\eps/3.
\end{align}
Next, by~\eqref{eq:stopped_feller}, for all large $n\in\N$ we have $\lvert \E_{z_n}[f(Z_{t\wedge \varrho_{r_k}})] - \E_{z_\infty}[f(Z_{t\wedge \varrho_{r_k}})]\rvert <\eps/3$. 
 Since
 $f(Z_{t\wedge \varrho_{r_k}})\mathbbm{1}(\cA_{r_k}^c)=f(Z_{t\wedge \varrho_{r_k}})-f(Z_{t\wedge \varrho_{r_k}})\mathbbm{1}(\cA_{r_k})$,
 the triangle inequality yields 
\begin{align*}
\lvert\E_{z_n}[f(Z_{t\wedge \varrho_{r_k}})\mathbbm{1}(\cA_{r_k}^c)&]-\E_{z_\infty}[f(Z_{t\wedge \varrho_{r_k}})\mathbbm{1}(\cA_{r_k}^c)]\rvert \\
\leq \lvert \E_{z_n}[f(Z_{t\wedge \varrho_{r_k}})] &- \E_{z_\infty}[f(Z_{t\wedge \varrho_{r_k}})] \rvert + \sup_{z\in\cD}\lvert f(z)\rvert(\P_{z_n}(\cA_{r_k})+\P_{z_\infty}(\cA_{r_k})) \leq 2\eps/3,
\end{align*}
which, together with~\eqref{eq:Feller_local_to_global_bound},
implies $\rvert \E_{z_n}[f(Z_t)]-\E_{z_\infty}[f(Z_t)]\rvert<\eps$
for all large $n\in\N$.
%Thus, we can estimate
%\begin{align*}
%\rvert \E_{z_n}[f(Z_t)]-\E_{z_\infty}[f(Z_t)]\rvert&\leq\rvert %\E_{z_n}[f(Z_{t\wedge \varrho_{r_k}})\mathbbm{1}(\cA_{r_k}^c)]-%\E_{z_\infty}[f(Z_{t\wedge \varrho_{r_k}})\mathbbm{1}%(\cA_{r_k}^c)]\rvert +\eps\\
%&\leq 2\eps + \eps. 
%\end{align*}
%Since this holds for any $\eps>0$, lemma follows.
\end{proof}

\begin{proof}[Proof of Theorem~\ref{thm:Feller_continuity_of_Z}] 
By Lemma~\ref{lem:Feller_continuity_from_stopped_to_general}, it suffices to prove that for any $f:\cD\to\RP$,
$T\in\RP$ and a convergent sequence $z_n\to z_\infty\in\cD$ (as $n\to\infty$), there exists a sequence $(r_k)_{k\in\N}$ tending to infinity, such that for any $k\in\N$ we have
$$
\E_{z_n}[f(Z_{T\wedge \varrho_{r_k}})] \to \E_{z_\infty}[f(Z_{T\wedge \varrho_{r_k}})]
%\quad \&\quad \P_{z_n}(\varrho_{r_k}<T)\to\P_{z_\infty}(\varrho_{r_k}<T),
\quad\text{as $n\to\infty$}.
$$
The first step is to prove that, for a fixed $r>0$, the laws of
$(Z_{\cdot \wedge \varrho_r},L_{\cdot \wedge \varrho_r},W_{\cdot})$ under $\P_{z_n}$ are tight. 
%the convergence of the laws of the stopped processes $(Z_{\cdot \wedge \varrho_r},L_{\cdot \wedge \varrho_r},W_{\cdot})$, and then 
The second step consists of proving that 
every subsequence converges to the law of 
$(Z_{\cdot \wedge \varrho_r},L_{\cdot \wedge \varrho_r},W_{\cdot})$ 
under $\P_{z_\infty}$.
%follow with the extension of the state space in order to prove that %the limiting process solves the SDE~\eqref{eq::SDE}.

%In this proof we will have to extend our probability space. 
Let $E^{(1)} :=E_1\times E_2\times E_3$, where $E_1:=\cD$, $E_2:=  \RP$ and $E_3:=\R^{d+1}$, and denote by $D_{E^{(1)}}([0,\infty))$  the space of c\`adl\`ag  functions (i.e.~right-continuous functions with left limits), mapping the interval $[0,\infty)$ into the metric space $E^{(1)}$. Denote by $C_{E^{(1)}}([0,\infty)) \subset D_{E^{(1)}}([0,\infty))$ the subspace of continuous functions. Endow $D_{E^{(1)}}([0,\infty))$ with the Skorohod topology and its Borel $\sigma$-algebra $\cM^1$, see e.g.~\cite[Ch.~16]{Billingsley1999} for details. Any function $\omega \in D_{E^{(1)}}([0,\infty))$
can be expressed as $\omega = (a^{(1)},a^{(2)},a^{(3)})$, for some ``coordinate'' c\`adl\`ag functions $a^{(i)}: [0,\infty)\to E_i$ for $i=1,2,3$.
For any $t\in[0,\infty)$, define maps $A^{(i)}_t:D_{E^{(1)}}([0,\infty))\to E_i$
by
$A^{(i)}_t(\omega):=a^{(i)}(t)$, $i\in\{1,2,3\}$. %(see~\cite[Thm~12.5]{Billingsley1999} for the proof of %$\cM$-measurability of $A^{(i)}_t$).
%For every $t\in[0,T]$,
%define a 
Define a $\sigma$-algebra 
$\cM_t^1 := \sigma\{A^{(i)}_s:0\leq s\leq t, i\in\{1,2,3\}\}\subset\cM^1$ generated by  continuous maps $A^{(i)}_s$.
By~\cite[Thm~16.6]{Billingsley1999} we have $\cM_\infty^1=\cM^1$.
%This is essentially because 
%the preimages of measurable sets in $E^{(1)}$ under the projections %$(A_s^{(1)},A_s^{(2)},A_s^{(3)}):D_{E^{(1)}}([0,\infty))\to E^{(1)}$, %for $s\in[0,\infty)$,
%form a $\pi$-system that generates the Borel $\sigma$-algebra of the %Skorokhod topology on $D_{E^{(1)}}([0,\infty))$,
%making $(D_{E^{(1)}}([0,\infty)),(\cM_{t}^1)_{t\in[0,\infty)},\cM^1)$
%into a filtered measurable space.  

For any $z\in\cD$, 
by~\cite[Thm A.1]{menshikov2022reflecting}, there exists a filtered probability space $(\Omega,(\cF_t)_{t\in\RP},\cF,\P_z)$,
%with probability measure $\P_z$
supporting the  processes $(Z,L,W)$, taking values in $E^{(1)}$
such that, under $\P_z$,
the SDE in~\eqref{eq::SDE} holds and 
$W$ is a standard $(\cF_t)$-Brownian motion.
We may assume that the filtration $(\cF_t)_{t\in\RP}$ is complete (i.e.~$\cF_0$
contains all $\P_z$-null sets of $\cF$) and right continuous (i.e.~$\cF_t = \cap_{t<s}\cF_s$).
%By possibly removing a $\P_z$-null set from %$\Omega$, 
%we may assume that $\textit{all}$ paths of %$(Z,L,W)$ are continuous. 
Pick arbitrary $r_0\in\RP$ and $T>0$, and recall the stopping time 
$\varrho_{r_0}$, defined in~\eqref{eq::varrho}, is the first time the coordinate $X$
of $Z=(X,Y)$
reaches level $r_0$.
The stopped process   
$(Z_{t \wedge \varrho_{r_0}\wedge T},L_{t \wedge \varrho_{r_0}\wedge T},W_{t\wedge T})_{t\in\RP}$
produces a measurable map 
$(Z_{\cdot \wedge \varrho_{r_0}\wedge T},L_{\cdot \wedge \varrho_{r_0}\wedge T},W_{\cdot\wedge T}):(\Omega,\cF)\to(D_{E^{(1)}}([0,\infty)),\cM^1)$ 
and induces the probability measure $Q_z^1(\cdot):=\P_z((Z,L,W)\in \cdot)$ on $\cM^1$.

Denote by $\cP(D_{E^{(1)}}([0,\infty)),\cM^1)$ the space of probability measures on the measurable space $(D_{E^{(1)}}([0,\infty)),\cM^1)$ and endow it with the topology of weak convergence~(see \cite[VI.~3.]{Jacod2003} for details). We say that the sequence of measures $(Q_{n})_{n\in\N}\in \cP(D_{E^{(1)}}[0,\infty),\cM^1)$ is \textit{C-tight}, if every subsequence of $(Q_{n})_{n\in\N}$ has a convergent subsequence and the limiting probability measure $Q_{*}$ charges only the set $C_{E^{(1)}}([0,\infty))$, i.e.~$Q_{*}(C_{E^{(1)}}[0,\infty)) = 1$.

Recall that the measures $Q_z^1$, for $z\in\cD$, are defined by the laws of the stopped processes $(Z_{\cdot\wedge \varrho_r\wedge T},L_{\cdot\wedge \varrho_r\wedge T},W_{\cdot\wedge T})$ under the measure $\P_z$.
Lemma~\ref{lem:expected_growth_loc_time_norm}, and the fact that $W$ is a Brownian motion, imply that for any $\theta>0$ there exist positive constants $C_1,C_2$ such that for any stopping times $S_1,S_2\in\cF_T$ satisfying $S_1\leq S_2\leq S_1+\theta\leq T$,
for any $n\in\N$ we have 
$$\P_{z_n}(\| (Z_{S_2\wedge \varrho_{r_0}},L_{S_2 \wedge \varrho_{r_0}},W_{S_2})-(Z_{S_1\wedge \varrho_{r_0}},L_{S_1 \wedge \varrho_{r_0}},W_{S_1})\|_{2d+3}^2 \geq \eps) \leq (C_1 \theta + C_2\theta^{1/2})/\eps.$$ 
There exists a positive constant $C>0$, such that  $\sup_{t\in\R_+}
\|Z_{t\wedge \varrho_{r_0}}\|_{d+1}\leq  C$ $\P_{z_n}$-a.s. for all $n\in\N$.  Since Lemma~\ref{lem:expected_growth_loc_time_norm} bounds the expected growth of the local time $L$, for any $\eps>0$, there exists $K>0$, such that for all $n\in\N$ we have $\P_{z_n}(\sup_{0\leq s\leq T} \|(Z_{s \wedge \varrho_{r_0}},L_{s\wedge \varrho_{r_0}},W_{s})\|_{2d+3} \geq K)\leq \eps$. Thus we may apply Aldous's tightness criterion~\cite[VI.~Thm 4.5]{Jacod2003} to deduce that the sequence of measures  $(Q_{z_n}^1)_{n\in\N}$ is tight. Moreover, since $Q_{z_n}^1(C[0,\infty)) = 1$ for every $n\in\N$,~\cite[VI.~Prop. 3.26]{Jacod2003} implies that the sequence $(Q_{z_n}^1)_{n\in\N}$ is $C$-tight.
It follows that there exists a subsequence 
%$(z_{n_k})_{n\in\N}$ of $(z_n)_{n\in\N}$, such that the measures 
$(Q_{z_{n_k}}^1)_{k\in\N}$ that converges weakly to a probability measure  $Q_{*}^1$, satisfying $Q_{*}^1(C[0,\infty)) = 1$. For notational convenience we assume that the sequence $(Q_{z_n}^1)_{n\in\N}$ itself converges to $Q_{*}^1$. %Equivalently, this means that the laws of the processes $(Z_{\varrho_{r_0}\wedge T\wedge \cdot},L_{\varrho_{r_0}\wedge T\wedge \cdot},W_{T\wedge \cdot})$ weakly converge to $(A^{(1)},A^{(2)},A^{(3)})$ under the measure $Q_{*}^1$. 

Our aim is to prove that $(A^{(1)},A^{(2)},A^{(3)})$, under the measure $Q_{*}^1$, solves SDE~\eqref{eq::SDE}. 
%this will require an extension of the state space. The extended state %space will accommodate the processes
Consider the process $\cZ:=\bar\cZ_{\cdot \wedge T}$, where 
$$
\bar\cZ :=\left(Z_{\cdot\wedge\varrho_{r_0}},L_{\cdot\wedge\varrho_{r_0}},W_\cdot,\int_0^{\cdot} \Sigma^{1/2}(Z_u)\ud W_u,\int_0^{\cdot\wedge\varrho_{r_0}} \phi(Z_u)\ud L_u,\int_0^{\cdot\wedge\varrho_{r_0}}(b(X_u)^2 - \|Y_u\|_d^2)\ud L_u\right).
$$ 
The state space of $\cZ$ is $E^{(2)} := E^{(1)}\times E_4\times E_5\times E_6$, where $E_4 = \R^{d+1},E_5=\R^{d+1}$ and $E_6=\R$. Let  $(D_{E^{(2)}}([0,\infty)),(\cM_t^2)_{t\in[0,\infty)}, \cM^2)$ be the filtered measurable space with coordinate projections $A^{(i)}:D_{E^{(2)}}([0,\infty))\to D_{E_i}([0,\infty))$ for $i\in\{1,\dots,6\}$ (note that $A^{(i)}$, for $i\in\{1,2,3\}$, agrees with the definition in the beginning of the second paragraph of this proof).  Moreover, a measurable map $\cZ:(\Omega,\cF)\to(D_{E^{(2)}}([0,\infty)),\cM^2)$
 induces the probability measure $ Q_z^2(\cdot):=\P_z(\cZ\in \cdot)$ on $\cM^2$. Note that $Q_z^2((A^{(1)},A^{(2)},A^{(3)})^{-1}(\cdot)) = Q_z^{1}(\cdot)$ for all $z\in\cD$. Hence 
 the first three coordinates converge weakly under  $(Q_{z_n}^2)_{n\in\N}$.

%Recall that $A^{(1)}: D_{E^{(2)}}
%[0,\infty))\to D_{\cD}([0,\infty))$ and 
Denote $A^{(1)} = (A^{(1,X)},A^{(1,Y)})$, where $A^{(1,X)}:D_{E^{(2)}}([0,\infty)) \to D_{\RP}([0,\infty))$ and $A^{(1,Y)}:D_{E^{(2)}}([0,\infty)) \to D_{\R^d}([0,\infty))$.
 %Convergence of the measures $Q_{z_n}^2$ will follow from convergence of $(A^{(1)},A^{(2)},A^{(3)})$ under  $Q_{z_n}^1$.
 Note that for each $z\in\cD$, $L$ and $W$ under the measures $\P_z$ are adapted to the complete, right-continuous filtration $(\cF_{t})_{t\in\RP}$ and have continuous sample paths, which are thus in $D_{E_2}([0,\infty))$ and $D_{E_3}([0,\infty))$, respectively.\footnote{We work with c\`adl\`ag paths because we apply~\cite[Thm~2.2]{Kurtz1991} to conclude the stability of the stochastic integrals.} Moreover, the processes $W,L$ are semimartingales, which satisfy $\E_{z_n}[L_{t\wedge \varrho_{r_0}\wedge T}]<\infty$ (by Lemma~\ref{lem:expected_growth_loc_time_norm})  and  $\E_{z_n}[[W]_{t\wedge T}] \leq T$ for any $t\in\RP$ and $n\in\N$. Thus, Assumption~\cite[C2.2(i)]{Kurtz1991} is satisfied with the deterministic time $\tau^\alpha_n := \alpha+1$, where $\alpha\in(0,\infty)$, $n\in\N$ are arbitrary parameters and $\tau^\alpha_n$ is a sequence of stopping times in~\cite[C2.2(i)]{Kurtz1991}. Since   the functions $z\mapsto\Sigma(z)$, $z\mapsto \phi(z)$, $z\mapsto b(x)^2-\|y\|_{d+1}^2$ are continuous for $z=(x,y)\in\cD$ and $\P_z(\cZ\in\cdot) = Q_z^2(\cdot)$,~\cite[Thm~2.2]{Kurtz1991} implies that there exists a probability measure $Q_{*}^2$ on $(D_{E^{(2)}}([0,\infty)),\cM^2)$, such that
 %measures $Q_{z_n}^2$ weakly converge to some %$Q_{*}^2$. Moreover, it holds that
$$
\widetilde A :=\left(A^{(1)}_\cdot,A_\cdot^{(2)},A_\cdot^{(3)},\int_0^{\cdot} \Sigma^{1/2}(A^{(1)}_u)\ud A_u^{(3)},\int_0^{\cdot} \phi(A^{(1)}_u)\ud A^{(2)}_u,\int_0^{\cdot} (b(A_u^{(1,X)})^2 - \|A^{(1,Y)}_u\|_d^2)\ud A^{(2)}_u\right),
$$
under the measure $Q_{z_n}^2$, converges weakly
(as $n\to\infty$) to
$\widetilde A$
under $Q_{*}^2$. In particular,\cite[Thm~2.2]{Kurtz1991} ensures that $A^{(2)},A^{(3)}$ are semimartingales under $Q_{*}^2$.
Note that all limiting processes have continuous paths $Q_{*}^2$-a.s., which follows from the fact that $A^{(1)},A^{(2)}$ and $A^{(3)}$ have continuous paths $Q_{*}^2$-a.s.
We now extend (using~\cite[Lem~4.3]{Georgiou2019Invariance}) the weak convergence to stopping times $\varrho_r$ and the corresponding stopped processes.

For $a^{(1,X)}\in D_{\RP}([0,\infty))$ and $r>0$, denote $\tau_{r}(a^{(1,X)}) := \inf \{ t>0: a^{(1,X)}_t \geq r \text{ or } a^{(1,X)}_{t-} \geq r\}$, where $a^{(1,X)}_{t-}$ denotes the left limit of $a^{(1,X)}$ at time $t$. Since, for any $z\in\cD$, the process $X$ has $\P_z$-a.s.~continuous paths, we have $\tau_r(X) = \varrho_r$, $\P_z$-a.s.
We now  make the final extension of the state space, which will capture the convergence of the stopping times. Denote $E^{(3)} := E^{(2)}\times \RP$, and let $(D_{E^{(3)}}([0,\infty)),(\cM_t^3)_{t\in[0,\infty)}, \cM^3)$ be the filtered measurable space with coordinate projections $A^{(i)}$, for $i\in\{1,\dots,7\}$.
For any $r\in (0,\infty)$, a measurable map $(\cZ,\varrho_r):(\Omega,\cF)\to(D_{E^{(3)}}([0,\infty)),\cM^3)$
 induces the probability measure $ Q_z^{3,r}(\cdot):=\P_z((\cZ,\varrho_r)\in \cdot)$ on $\cM^3$. %Note that $\P_z(\varrho_r\in\cdot) = Q_{z}^{3,r}(\tau_{r}(A^{(1,X)})^{-1}(\cdot))$. 
 Furthermore, for any $r\in(0,\infty)$, it holds that $Q_z^{3,r}(\widetilde A^{-1}(\cdot)) = Q_z^{2}(\cdot)$.

As before, $\widetilde A$ under $Q^{3,r}_{z_n}$ converge to $\tilde A$ under some probability measure $Q^{3,r}_{*}$.
Moreover, $A^{(1,X)}$ has $Q_{*}^{3,r}$-a.s. continuous paths.
Thus,~\cite[Rem 4.8 and Lem 4.3]{Georgiou2019Invariance} implies that for all but at most countably many $r\in(0,r_0)$, the map $D_{\RP}([0,\infty))\to\RP$, given by $a^{(1,X)}\mapsto \tau_r(a^{(1,X)})$, is continuous at $A^{(1,X)}$, $Q_{*}^{3,r}$-a.s. (recall that $r_0>0$ is an arbitrary number, fixed at the beginning of the proof). Pick $r_1\in(r_0/2,r_0)$ such that $a^{(1,X)}\mapsto \tau_{r_1}(a^{(1,X)})$ is continuous at $a^{(1,X)}$, $Q_{*}^{r_1}$-a.s. The  Continuous mapping theorem~\cite[Thm 2.7]{Billingsley1999} implies that $(\tilde A,\tau_{r_1}(A^{(1,X)}))$, under $Q_{z_n}^{3,r_1}$, converges weakly  to $(\tilde A,\tau_{r_1}(A^{(1,X)}))$ under $Q_{*}^{3,r_1}$. %Since $\P_z(\varrho_r\in\cdot) = Q_{z}^{3,r}(\tau_{r}(A^{(1,X)})^{-1}(\cdot))$, this implies that $Q_{z_n}^{3,r_1}$ weakly converges to some $Q_{*}^{3,r_1}$.
The convergence of the stopped processes follows by applying~\cite[Theorem 2.2]{Kurtz1991} to the stochastic integral $\int_0^t \mathbbm{1}\{s\leq\varrho_{r_1}\} \ud\cZ_s$  (condition~\cite[C2.2(i)]{Kurtz1991} is satisfied with $\tau^\alpha_n := \alpha+1$ as above). Since $\P_z((\cZ,\varrho_{r_1})\in\cdot) =Q_z^{3,r_1}(\cdot)$, it follows that
$(\tilde A_{\cdot \wedge \tau_{r_1}(A^{(1,X)})})$, under $Q_{z_n}^{3,r_1}$, converges weakly as $n\to\infty$ to $(\tilde A_{\cdot \wedge \tau_{r_1}(A^{(1,X)})})$ under the probability measure $Q_{*}^{3,r_1}$.

To conclude the proof, we have to show that the process $(A^{(1)},A^{(2)},A^{(3)})$ under $Q_{*}^{3,r_1}$ solves SDE~\eqref{eq::SDE}.
Denote by $\mathbf{0}\in D_{E_1}([0,\infty))$ 
the function mapping every $t\in[0,\infty)$ into the origin of $\R^{d+1}$.
For every $n\in\N$, we have
$$A^{(1)}_{\cdot \wedge \tau_{r_1}(A^{(1,X)})} - z_n -\int_0^{\cdot \wedge \tau_{r_1}(A^{(1,X)})} \phi(A^{(1)}_u)\ud A^{(2)}_u - \int_0^{\cdot \wedge \tau_{r_1}(A^{(1,X)})} \Sigma^{1/2}(A^{(1)}_u)\ud A^{(3)}_u\equiv \mathbf{0}, \quad \text{$Q_{z_n}^{3,r_1}$-a.s.}$$
Since the set $\{\mathbf{0}\}$ is closed in 
$ D_{E_1}([0,\infty))$,~\cite[Thm~2.1(iii)]{Billingsley1999} implies that $$A_{\cdot \wedge \tau_{r_1}(A^{(1,X)})}^{(1)} -z_{\infty} - \int_0^{\cdot \wedge \tau_{r_1}(A^{(1,X)})} \Sigma^{1/2}(A^{(1)}_u)\ud A_u^{(3)} - \int_0^{\cdot \wedge\tau_{r_1}(A^{(1,X)})} \phi(A^{(1)}_u)\ud A^{(2)}_u\equiv\mathbf{0}, \quad \text{$Q_{*}^{3,r_1}$-a.s.}$$

Note that since $A^{(3)}$ is a Brownian motion under 
$Q_{z_n}^{3,r_1}$ for every $n\in\N$, it is also a Brownian motion under the weak limit $Q_{*}^{3,r_1}$.

It remains to prove that, under $Q^{3,r_1}_{*}$, $A^{(2)}$ is a local time
of $A^{(1)}$ at the boundary $\partial\cD$. For any $z\in\cD$, 
the local time has to satisfy $L_{t\wedge \varrho_r} = \int_0^{t\wedge \varrho_r} \1{Z_s\in\partial\cD}\ud L_s$, $\P_{z}$-a.s.
This requirement is equivalent to 
\begin{equation} 
\label{eq:loc_time_definition}
\int_0^{t\wedge \varrho_r} (b(X_s)^2 - \|Y_s\|_d^2) \ud L_s = 0, \quad\text{$\P_z$-a.s.}
\end{equation}
Indeed, since $b(x)^2-\|y\|_d^2=0$ for all $(x,y)\in\partial \cD$, we have
$\int_0^{t\wedge \varrho_r} (b(X_s)^2 - \|Y_s\|_d^2) \ud L_s =
\int_0^{t\wedge \varrho_r} (b(X_s)^2 - \|Y_s\|_d^2) \1{Z_s\in\partial\cD}\ud L_s=
0$, $\P_z$-a.s., by the definition of local time.
Conversely, since $b(x)^2-\|y\|_d^2>0$ for all $(x,y)\in\cD\setminus \partial \cD$, the sets $\cD_k:=\{(x,y)\in\cD:b(x)^2-\|y\|_d^2\in[1/(k+1),1/k)\}$
are pairwise disjoint and satisfy $\cup_{k\in\N}\cD_k=\cD\setminus\partial \cD$.
Moreover, we have
$0\leq\int_0^{t\wedge \varrho_r} \1{Z_s\in\cD_k}\ud L_s\leq (k+1)\int_0^{t\wedge \varrho_r} (b(X_s)^2 - \|Y_s\|_d^2) \1{Z_s\in\cD_k}\ud L_s=0$, implying~\eqref{eq:loc_time_definition} via
$L_{t\wedge \varrho_r} = \int_0^{t\wedge \varrho_r} \1{Z_s\in\partial\cD}\ud L_s +\sum_{k\in\N}\int_0^{t\wedge \varrho_r} \1{Z_s\in\cD_k}\ud L_s
=\int_0^{t\wedge \varrho_r} \1{Z_s\in\partial\cD}\ud L_s$.

Note that $\int_0^{\cdot\wedge \tau_{r_1}(A^{(1,X)})} (b(A^{(1,X)})^2 - \|A^{(1,Y)}\|_d^2)\ud A^{(2)}_u \equiv \mathbf{0}$, $Q_{z_n}^{3,r_1}$-a.s., for every $n\in\N$, where $\mathbf{0}$  now denotes the zero function in $D_{E_6}([0,\infty))$ (recall that $E_6=\R$).
Thus, by~\cite[Thm. 2.1(iii)]{Billingsley1999},
we get $\int_0^{\cdot\wedge \tau_{r_1}(A^{(1,X)})} (b(A^{(1,X)})^2 - \|A^{(1,Y)}\|_d^2)\ud A^{(2)}_u \equiv 0$, $Q_*^{3,r_1}$-a.s. Hence, the condition in~\eqref{eq:loc_time_definition} implies that 
$A^{(2)}$ is indeed the local time of $A^{(1)}$
at the boundary $\partial\cD$.

 We have thus proved that $(A^{(1)}_{\cdot \wedge \tau_{r_1}(A^{(1,X)})},A^{(2)}_{\cdot \wedge \tau_{r_1}(A^{(1,X)})},A^{(3)}_{\cdot \wedge \tau_{r_1}(A^{(1,X)})})$, under $Q_{*}^{3,r_1}$ (and hence under $Q_{*}^1$), solves SDE~\eqref{eq::SDE} on the stochastic interval $[0,\varrho_{r_1}]$. The pathwise uniqueness~\cite[Thm A.1.]{menshikov2022reflecting} of solutions of SDE~\eqref{eq::SDE} implies 
 that every sub-sequential limit $Q_{*}^1$ of the sequence of $(Q_{z_n}^1)_{n\in\N}$ 
 equals $Q_{z_\infty}^1$, implying the
 Feller continuity for the process stopped at $\varrho_{r_1}$, i.e.~the limit in~\eqref{eq:stopped_feller} holds for $r_1$.
 Since  $r_0$ was chosen arbitrarily and $r_1\in(r_0/2,r_0)$, we can inductively construct a sequence $(r_k)_{k\in\N}$ with $\lim_{k\to\infty}r_k = \infty$, such that~\eqref{eq:stopped_feller} holds for every $k\in\N$. An application of Lemma~\ref{lem:Feller_continuity_from_stopped_to_general} concludes the proof of Theorem~\ref{thm:Feller_continuity_of_Z}.
\end{proof}

\subsection{Application to return times}
\label{subsubsec:proof_of_thm_return_times}
An easy consequence of Feller continuity and irreducibility is the non-confinement of the reflected process $Z$ in any compact set. We will use this property to extend the asymptotic results about return times in Section~\ref{subsection:return_times_drift_conditions} to the entire domain $\cD$.

Recall the definition of the first passage time $\varrho_r$ (over $r$) in~\eqref{eq::varrho}, the return time~$\varsigma_r$ (below $r$) in~\eqref{eq::varsigma} and the neighbourhood $\cD_{h} = \{z\in\cD: \exists z'\in \partial\cD \text{ such that } \|z-z'\|_{d+1}<h\}$ of the boundary $\partial\cD$ for $h\in(0,\infty)$.
 We start with the following proposition.
 
\begin{prop}
\label{prop:crossing_time_probabilities}
For any $0<r_0<r_1<\infty$, the following statements hold.
\begin{itemize}
    \item[(a)] For all sufficiently small $h>0$ and $z\in (\cD\setminus \cD_h )\cap (r_0,r_1)\times\R^d$, we have $\P_z(\varrho_{r_1}<\varsigma_{r_0})>0$.

\item[(b)] For any $\delta\in(0,r_1-r_0)$ and sufficiently small $h>0$, there exists $\eps>0$, such that $\P_z(\varsigma_{r_0}<\varrho_{r_1})>\eps$  for any $z\in (\cD\setminus\cD_h) \cap [r_0,r_1-\delta]\times\R^d$.

\item[(c)] All moments of the first exit time of the interval $(r_0,r_1)$ are finite: $\E_z[(\varsigma_{r_0}\wedge \varrho_{r_1})^k]<\infty$ for all $k\in\N$ and $z\in\cD$.
\end{itemize}
\end{prop}

\begin{proof}
Since proofs of both (a) and (b) follow from the same PDE argument, we will prove them together. Let $i\in\{0,1\}$.
Fix  $0<h'<h<\infty$.
Let $\cD_{h}$ be as above and set  $\cD_{r_0,r_1,h} := (\cD \setminus \cD_h) \cap (r_0,r_1)\times\R^d$. Note that $\cD_{r_0,r_1,h}\subset \cD_{r_0,r_1,h'}$. For any $h''\in(h',h)$ there exists a closed domain $\hat\cD_{r_0,r_1,h'}$ with $C^2$ boundary, satisfying
$$
\cD_{r_0,r_1,h''} \subset \hat\cD_{r_0,r_1,h'} \subset \cD_{r_0,r_1,h'}, \quad \cD_{r_0,r_1,h''} \cap \{r_i\}\times\R^d  = \hat\cD_{r_0,r_1,h'} \cap \{r_i\}\times\R^d \quad \text{for $i\in\{0,1\}$}.
$$
Closed domain $\hat\cD_{r_0,r_1,h'}$ can be obtained from $\cD_{r_0,r_1,h'}$ by smoothing its corners appropriately. 

Choose continuous functions $f_{i}: \partial \hat\cD_{r_0,r_1,h'} \rightarrow \R$, such that $f_{i} \equiv 1$ on $\cD_{r_0,r_1,h} \cap \{r_i\}\times\R^d = \partial \cD_{r_0,r_1,h} \cap \{r_i\}\times\R^d $ and $f_i \equiv 0$ on $\partial \hat\cD_{r_0,r_1,h'} \setminus (\hat\cD_{r_0,r_1,h'} \cap \{r_i\}\times\R^d )$ for $i\in\{0,1\}$.
Dirichlet problems on $\hat\cD_{r_0,r_1,h'}$ with boundary conditions $f_{i}: \partial \hat\cD_{r_0,r_1,h'} \rightarrow \R$ are given by 
\begin{align}
    \label{eq::PDE1}
    \frac{1}{2}\Delta_{\Sigma}u_{i} &= 0 \text{ on } \hat\cD_{r_0,r_1,h'}\setminus \partial \hat\cD_{r_0,r_1,h'}; \\
    \label{eq::PDE2}
    u_{i} &= f_{i}, \text{ on } \partial \hat\cD_{r_0,r_1,h'}.
\end{align}
  Then, by~\cite[pp.\ 364--366]{karatzas2012brownian}, the functions
$$
u_i(z) := \E_z[f_i(Z_{\tau})], \quad \text{where } \tau:= \inf\{t\in\RP:Z_t\in\partial\hat\cD_{r_0,r_1,h'}\},
$$
solve the respective Dirichlet problem in \eqref{eq::PDE1}--\eqref{eq::PDE2} on $\hat\cD_{r_0,r_1,h'}$ for $i\in\{0,1\}$. Moreover, $f_0(Z_\tau) \leq \mathbbm{1}\{\varsigma_{r_0} < \varrho_{r_1}\}$, $\P_z$-a.s., and $f_1(Z_\tau) \leq \mathbbm{1}\{\varrho_{r_1} < \varsigma_{r_0}\}$, $\P_z$-a.s., on $\hat\cD_{r_0,r_1,h'}$,  implying $u_0(z) \leq \P_z(\varsigma_{r_0} < \varrho_{r_1})$ and $u_1(z) \leq \P_z(\varrho_{r_1} < \varsigma_{r_0})$. Since $f_i$ are continuous, $\hat\cD_{r_0,r_1,h'}$ has a $C^2$ boundary and thus satisfies the inside the sphere property (see~\cite[p.~59]{friedman2008partial} for definition), and the coefficients in \eqref{eq::PDE1}--\eqref{eq::PDE2} are continuous and uniformly elliptic by assumption \aref{ass:covariance2}, the maximum principle \cite[Thm.~21, p.~55]{friedman2008partial}, applied to $-u_i$, yields $u_i(z) > 0$ for all $z\in\hat\cD_{r_0,r_1,h'}$. This directly implies part (a). Moreover, since $\cD_{r_0,r_1,h}\cap [r_0,r_1-\delta]\times\R^d$ is compact and $u_0$ is continuous with $u_0(z)>0$ for all $z\in\cD_{r_0,r_1,h}\cap [r_0,r_1-\delta]\times\R^d$, part (b) follows.

Recall  the notation $Z_t=(X_t,Y_t)\in\cD$, where $X_t\in\RP$, for all $t\in\RP$. To prove part (c), consider a function $p(z) := \P_z(X_1>r_1)$. Proposition~\ref{prop:marginal_equivalent_to_Lebesgue} implies that $p(z)>0$ for all $z\in\cD$. Moreover, by Theorem~\ref{thm:Feller_continuity_of_Z} and~\cite[Thm~2.1]{Billingsley1999}, the function $p$ is lower semi-continuous. Since the set $\cD\cap [0,r_1]\times \R^d$ is compact, there exists $\eps>0$ such that $p(z)>\eps$ for all $z\in \cD\cap [0,r_1]\times \R^d$. Thus, $\P_z(X_1\leq r_1)\leq 1-\eps$ for all $z\in\cD\cap [0,r_1]\times \R^d$.
This, together with the Markov property,  implies that for every $n\in\N$ and $z\in\cD$ we have $\P_z(\varsigma_{r_0}\wedge \varrho_{r_1}>n)\leq \P_z(\cap_{k=1}^n \{X_k\leq r_1\} )\leq (1-\eps)^n$. Thus, $\E_z[(\varsigma_{r_0}\wedge \varrho_{r_1})^k]<\infty$ holds for all $k\in\N$ and every $z\in \cD$.
\end{proof}

\begin{proof}[Proof of Theorem~\ref{thm:return_times}]
Transience and the lower bounds in the recurrent case both require the following claim.

\noindent\textbf{Claim 1.} For any $0<r_0<r_1<\infty$ and $z\in\cD \cap (r_0,r_1)\times \R^d$ we have $\P_z(\varrho_{r_1}<\varsigma_{r_0})>0$.

\noindent \underline{Proof of Claim 1}. By Proposition~\ref{prop:crossing_time_probabilities}(a), for $h>0$ sufficiently small and $z\in(\cD\setminus \cD_h)\cap(r_0,r_1)$, we have $\P_z(\varrho_{r_1}<\varsigma_{r_0})>0$. Pick $z \in \partial \cD\cap (r_0,r_1)\times\R^d$ and define the stopping time $v_h := \inf\{t>0:Z_t\in\cD\setminus\cD_h\}$. The continuity of paths implies $\E_z[\varsigma_{r_0}]>0$. Moreover, Assumptions~\aref{ass:domain2}, \aref{ass:vector2} and~\aref{ass:covariance2} and~\cite[Lem.~4.5]{menshikov2022reflecting} imply that
for some $\delta_\Sigma>0$ and all
sufficiently small $h > 0$, we have $\E_z[v_h \wedge \varrho_{r_1}] \leq 2h^2/\delta_\Sigma$. Hence, if $h\in(0, (\delta_\Sigma \E_z[\varsigma_{r_0}]/2)^{1/2})$ then $\E_z[v_h \wedge \varrho_{r_1}] < \E_z[\varsigma_{r_0}]$. In particular, this implies  $\P_z(v_h \wedge \varrho_{r_1} < \varsigma_{r_0}) > 0$. If $\P_z(\{v_h \wedge \varrho_{r_1} < \varsigma_{r_0}\} \cap \{\varrho_{r_1} <v_h\} ) > 0$ the proof is complete. Otherwise,  $\{v_h \wedge \varrho_{r_1} < \varsigma_{r_0}\}  = 
\{v_h \wedge \varrho_{r_1} < \varsigma_{r_0}\}\cap \{\varrho_{r_1} >v_h\}=
\{v_h < \varsigma_{r_0}\}$, $\P_z$-a.s. Thus $\P_z(v_h < \varsigma_{r_0})=\P_z(v_h \wedge \varrho_{r_1} < \varsigma_{r_0}) > 0$.
By the strong Markov property at $v_h$ and Proposition~\ref{prop:crossing_time_probabilities}(a), Claim 1 follows. 

\noindent\underline{Proof of~(a)}. In this case we have $\beta > \beta_c$. By Theorem~\ref{thm:rec_tran} the process $Z$ is transient.  
The Lyapunov function $f_{w,\gamma}(x,y)$ (with $\gamma<0$) tends to zero  by~\eqref{eq:bound_Lyapunov_f} as $x\to\infty$. Analogous to the proof of Theorem~\ref{thm:rec_tran}(b), an application of Lemma~\ref{lem2.3} (and in particular
the bound in~\eqref{eq:upper_bound_lower_exit} with $\kappa=f_{w,1}(Z)$, $V(u)=u^\gamma$),  implies that for every $x_0\in\RP$, there exist $c(x_0)$, such that for all $z\in\cD \cap [x_0+c(x_0),\infty)\times\R^d$, we have $\P_z(\varsigma_{x_0} = \infty) > 0$. Moreover, Claim~1 implies that 
for all $z\in\cD \cap (x_0,\infty)\times\R^d$ we have 
$\P_z(\varrho_{x_0+c(x_0)}<\varsigma_{x_0})>0$. 
The strong Markov property at the stopping time 
$\varrho_{x_0+c(x_0)}$ concludes the proof of Theorem~\ref{thm:return_times}(a).

\noindent\underline{Proof of the lower bound in (b)}.  In this case we have $\beta < \beta_c$.  Recall the definition of $m_c$ in~\eqref{eq::m0}. Moreover, Proposition~\ref{prop:return_time_lower_bound} implies that, for any $p\in(m_c,\infty)$, there exist constants $x_0\in(0,\infty)$ and $c_1,c_2\in(1,\infty)$, such that for all $x_1\in[x_0,\infty)$ and $z\in\cD \cap (c_1(x_1+1)+c_2,\infty)\times\R^d$ we have $\P_z(\varsigma_{x_1} > t) \geq Ct^{-p}$, for some constant $C$ and all $t \geq 1$, say. Pick $x_1\in[x_0,\infty)$. Note that by Claim~1, for any $x_2\in(0,x_1]$ and any $z\in\cD \cap (x_2,c_1(x_1+1)+c_2)\times\R^d$ we have $\P_z(\varrho_{c_1(x_1+1)+c_2}<\varsigma_{x_2})>0$. Moreover, $\varsigma_{x_2}\geq\varsigma_{x_1}$, $\P_z$-a.s. Thus, the strong Markov property applied at the stopping time $\varrho_{c_1(x_1+1)+c_2}$ implies the lower bound in Theorem~\ref{thm:return_times}(b).

\noindent\textbf{Claim 2.} For any $0<r_0<r_1<\infty$ and any $\delta\in(0,\min\{(r_1-r_0)/2,1\})$, there exists an $\eps>0$ such that $\P_z(\varsigma_{r_0}<\varrho_{r_1})>\eps$ for any $z\in\cD\cap[r_0,r_1-2\delta]\times\R^d$.

\noindent \underline{Proof of Claim 2.}
For any $z=(x,y)\in\cD\cap[r_0,r_1-2\delta]\times\R^d$, It\^o's formula~\eqref{eq::Ito} implies that $X_{t\wedge \varrho_{r_1-\delta}} = x+M_{t\wedge \varrho_{r_1-\delta}} + \int_0^{t\wedge \varrho_{r_1-\delta}}\langle e_x,\phi(Z_s)\rangle \ud L_s$.  Moreover, by~\eqref{eq::Ito_QV} and Assumption~\aref{ass:covariance1}, $M$ is a local martingale. By~\aref{ass:vector1} and Lemma~\ref{lem:expected_growth_loc_time_norm} there exist constants $C_1,C_2\in(1,\infty)$, such that
\begin{equation}
\label{eq:bound_expectation:x}
\E_z[X_{t\wedge \varrho_{r_1-\delta}}]\leq x + C_1t+C_2t^{1/2}\leq x+ (C_1+C_2)t^{1/2}\leq r_1-3\delta/2 \quad \text{ for all $t\in[0,t_0]$},
\end{equation}
where $t_0 := \delta^2/(4 (C_1+C_2)^2)$,
since $x\in[r_0,r_1-2\delta]$. Thus, $\eps_1:=\inf\{\P_z(\varrho_{r_1-\delta}>t_0):z\in\cD\cap[r_0,r_1-2\delta]\times\R^d\}>0$ and hence $\E_z[\varrho_{r_1-\delta}]>\eps_1t_0>0$. Otherwise, there would exist $\eps>0$ and $z\in\cD\cap[r_0,r_1-2\delta]\times\R^d$ such that $\E_z[X_{t\wedge \varrho_{r_1-\delta}}] \geq (r_1-\delta)(1-\eps)>r_1-3\delta/2$, contradicting the inequality in~\eqref{eq:bound_expectation:x}.

By~\aref{ass:domain2}, \aref{ass:vector2},~\aref{ass:covariance2} and~\cite[Lem.~4.5]{menshikov2022reflecting}, 
for some $\delta_\Sigma>0$ and all
sufficiently small $h > 0$, we have $\E_z[v_h \wedge \varrho_{r_1-\delta}] \leq 2h^2/\delta_\Sigma$ for all $z\in\cD$. Thus, for all $h\in(0,(\delta_\Sigma\eps_1 t_0)^{1/2}/2)$ and $z\in\cD\cap[0,r_1-2\delta]\times\R^d$, the strict inequality $\E_z[v_h \wedge \varrho_{r_1-\delta}]\leq t_0\eps_1/2<t_0\eps_1\leq \E_z[\varrho_{r_1-\delta}]$ implies $\P_z(v_h<\varrho_{r_1-\delta})>\eps_2$, for some $\eps_2>0$. Moreover, by Proposition~\ref{prop:crossing_time_probabilities}(b) there exists $h_0\in(0,\infty)$ such that the following holds: for any  $h\in(0,h_0)$ there exists $\eps_3>0$, such that for all $z\in\cD\setminus \cD_h \cap [r_1,r_0-\delta)$, we have $\P_z(\varsigma_{r_0}<\varrho_{r_1})>\eps_3$.  Pick $h =\min\{h_0/2,(\delta_\Sigma \eps_1)^{1/2}/2\}$. The strong Markov property at time $v_h$ yields the claim: 
\begin{align*}
\P_z(\varsigma_{r_0}<\varrho_{r_1}) &\geq \P_z(\varsigma_{r_0}<\varrho_{r_1},v_h<\varrho_{r_1-\delta})\geq \E_z[\1{v_h<\varrho_{r_1-\delta}}\cdot \P_{Z_{v_h}}(\varsigma_{r_0}<\varrho_{r_1})] \\
&\geq \E_z[\eps_3\1{v_h<\varrho_{r_1-\delta}}]\geq \eps_2\eps_3 \quad \text{for all $z\in\cD\cap [0, r_1-2\delta]\times\R^d$}.
\end{align*}

\noindent\underline{Proof of the upper bounds in (b).} 
Pick $p\in(0,m_c)$. Proposition~\ref{prop:return_time_upper_bound} implies the following: there exists $x_0\in(0,\infty)$ such that for every $x_1\in[x_0,\infty)$ and $z\in\cD$  we have $\E_z[\varsigma_{x_1}^p]<\infty$. We now extend this result to all $x_1\in(0,\infty)$. Thus for any $q\in(p,m_c)$, there exist $x_0,\tilde C\in(0,\infty)$, such that $\E_z[\varsigma_{x_0}^{q}]\leq \tilde C$ for all $z\in \cD \cap \{x_0+1\} \times \R^d$.  Pick $x_1\in(0,x_0)$ and $z\in\cD$. In order to prove  $\E_z[\varsigma_{x_1}^p]<\infty$ for all $x_1\in(0,x_0)$, we introduce the sequence of stopping times $T_0 :=0$, 
$$
S_k :=\inf\{t>T_{k-1}:X_t\leq x_0\} \quad \text{and} \quad T_k:=\{t>S_k: X_t\notin (x_1,x_0+1)\}, \quad \text{for $k\in\N$.}
$$
%Propositions~\ref{prop:return_time_upper_bound} and %\ref{prop:crossing_time_probabilities}(b)--(c), along with the %strong Markov property of $Z$, imply the existence of
There exist constants $\eps, C_1,C_2\in(0,\infty)$ such that, for every $k\in\N$, we have:  $\P_z(Z_{T_k} \leq x_1)>\eps$ 
(resp. $\E_z[(S_k-T_{k-1})^q]\leq C_1$, $\E_z[(T_k-S_k)^q]\leq C_2$) by
Prop.~\ref{prop:crossing_time_probabilities}(b)
(resp. Prop.~\ref{prop:return_time_upper_bound}, Prop.~\ref{prop:crossing_time_probabilities}(c)) and  the strong Markov property at $S_k$ (resp.  $T_{k-1}$, $S_k$).

Since $T_k = \sum_{j=1}^k (S_j-T_{j-1} + T_{j}-S_j)$, we have $\P_z(T_k<\infty)=1$, for all $k\in\N$. Moreover, for $k\in\N$, we have $\{Z_{T_k}\leq x_1\} =\{T_k = \varsigma_{x_1}\} $, $\P_z$-a.s.
(recall that on the event $\{Z_{T_k}\leq x_1\}$, we have $T_n=T_k=\varsigma_{x_1}$ for all $n\geq k$).
Hence, by the strong Markov property at $S_j$, for $j\in\{1,\ldots,k\}$, it follows that $\P_z(\varsigma_{x_1}>T_k)=
\P_z(\cap_{j=1}^k Z_{T_j}>x_1)\leq 
(1-\eps)^k$. We thus conclude $\P_z(\varsigma_{x'}<\infty)=1$ for all $x'\in(0,x_0)$. 

Define the indicator $\mathbbm{1}(A_k) :=\1{Z_{T_k}\leq x_1}\prod_{j=1}^{k-1}\1{Z_{T_j}>x_1}$. Note that the events $A_k$, $k\in\N$, are pairwise disjoint, $\P_z(A_k)\leq (1-\eps)^{k-1}$ and, since $\P_z(\varsigma_{x_1}<\infty)=1$, the following equality holds
$$
\varsigma_{x_1} = \sum_{k=1}^{\infty}\mathbbm{1}(A_k)T_k \quad \text{$\P_z$-a.s.}
$$ 
Thus, by H\"older's inequality with exponents $p/q$ and $1-p/q$, we obtain
\begin{align*}
\E_z[\varsigma_{x_1}^p] &= \E_z\left[\sum_{k=1}^{\infty}\mathbbm{1}(A_k)T_k^p\right] = \sum_{k=1}^\infty \E_z\left[\mathbbm{1}(A_k)\left(\sum_{j=1}^k(T_j-S_j) + \sum_{j=1}^k(S_j-T_{j-1})\right)^p\right] \\
&\leq \sum_{k=1}^\infty (2k)^p\E_z\left[\mathbbm{1}(A_k) \max_{j\in\{1,\dots,k\}}\{(T_j-S_j)^p,(S_j-T_{j-1})^p\}\right] 
\\
&\leq \sum_{k=1}^\infty (2k)^p\P_z(A_k)^{1-p/q}\E\left[ \max_{j\in\{1,\dots,k\}}\{(T_j-S_j)^q,(S_j-T_{j-1})^q\}\right]^{p/q} 
\\
&\leq  \sum_{k=1}^\infty (2k)^p (1-\eps)^{(k-1)(1-p/q)}\E_z\left[\sum_{j=1}^k\left((T_j-S_j)^q + (S_j-T_{j-1})^q\right)\right]^{p/q}  \\
&\leq  \sum_{k=1}^\infty  (2k)^{p+1}(1-\eps)^{(k-1)(1-p/q)}(C_1+C_2)^{p/q} <\infty,
\end{align*}
where the first equality in the display above holds since $A_k$ are pairwise disjoint. 
We conclude that for every $x_1\in(0,\infty)$, all $p\in(0,m_c)$ and every $z\in\cD$ we have $\E_z[\varsigma_{x_1}^p]<\infty$. Since Markov's inequality implies $\P_z(\varsigma_{x_1}>t)\leq \E_z[\varsigma_{x_1}^p]/t^p<\infty$ for all $t\in(0,\infty)$, we have 
$\P_z(\varsigma_{x_1}>t)\leq Ct^{-p}$ for all $t\in[1,\infty)$
and $z\in\cD$ with $C:= \E_z[\varsigma_{x_1}^p]$.
\end{proof}

\subsection{Application to petite sets and Harris recurrence}
\label{subsubsec:continuous_markov_theory}
%Recall the definitions of the classical notions in the theory of continuous Markov chains in the context of the process $Z$.
The family of probability measures 
$\P_z(Z_t \in \cdot)$ on $\cB(\cD)$, indexed by $t\in\RP$
and $z\in\cD$ 
constitutes a Markov transition kernel by~\cite[Thm~A.1]{menshikov2022reflecting} and Theorem~\ref{thm:Feller_continuity_of_Z}.
%$z \in \cD$, $A \in \cB(\cD)$, and %$t\in \RP$ define the Markov %transitional kernel by, $\cP_t(z,A) := $$\P(Z_t \in A \vert Z_0 = z) =\P_z(Z_t $$\in A)$.
A non-empty set $B \in \cB(\cD)$ is called \textit{petite} if there exists  
a probability measure $a$ on $(\RP,\cB(\RP))$, which does
not charge zero (i.e.~ $a(\{0\})=0$), and 
a non-trivial measure $\varphi_a$  on $(\cD,\cB(\cD))$  satisfying
\begin{equation}
\label{eq::petite}
K_a(z, \cdot) \geq \varphi_a(\cdot)\qquad\text{for all $z \in B$,}
\end{equation}
where the Markov transition function
$K_a: \cD \times \cB(\cD) \rightarrow \RP$ is given by
\begin{equation}
\label{eq:Markov_transition_kernel}
K_a(z,\cdot) := \int_{\RP}\P_z(Z_t\in\cdot)a(\ud t).
\end{equation}
The measurability of  $z\mapsto K_a(z,A)$ for any
$A\in\cB(\cD)$ follows from~\cite[Ch.~III.~Prop.~1.6]{revuz2013continuous}.

Let $\varphi$ be a $\sigma$-finite measure on the Borel $\sigma$-algebra $\cB(\cD)$. 
%The process $Z$ is \textit{$\varphi$-irreducible}  if $\varphi(A) > 0$ implies $\E_z[\int_0^{\infty}\mathbbm{1}_A(Z_s)\ud s] > 0$ for all $A\in\cB(\cD)$ and $z \in \cD$.
The  process $Z$ is \textit{Harris recurrent} if  $\varphi(A) > 0$ implies  $\int_0^\infty \mathbbm{1}_A(Z_s) \ud s = \infty$, $\P_z$-a.s., for all $A\in\cB(\cD)$ and $z \in \cD$. 
%(Harris recurrence trivially implies $\varphi$-irreducibility).
%The process $Z$ admits a \textit{$\varphi$-irreducible skeleton} chain if for every $A\in\cB(\cD)$, such that  $\varphi(A) > 0$,  and $z \in \cD$, there exist $m \in (0,\infty)$ such that $\P_{z}(Z_{km}\in A) > 0$ for every $k \in \N$. 

\begin{prop}
\label{prop:harris_irreducible_petite}
Suppose that $\aref{ass:covariance2},$ $\aref{ass:vector2}$ and $\aref{ass:domain2}$ hold. Then all compact subsets of $\cD$ are petite. Moreover, if $\beta<\beta_c$, the process $Z$ is Harris recurrent.
\end{prop}

\begin{proof}
% The process $Z$ is Feller continuous by %Theorem~\ref{thm:Feller_continuity_of_Z} and %irreducible with  respect to the Lebesgue measure by %Proposition~\ref{prop:marginal_equivalent_to_Lebesgue}. %Hence~\cite[Thm~7.1 and 5.1]{tweedie94} imply that %every %compact subset of $\cD$ is a  petite set (the %definition of Feller continuity in~\cite{tweedie94} %coincides with our definition by~\cite[Thm~2.1]%{Billingsley1999}).
%Alternative proof:

Fix $z_0\in\cD\setminus\partial\cD=\ocD$
and $h_0>0$, such that 
$B(z_0,2h_0)=\{z\in\cD:\|z-z_0\|_{d+1}<2h_0\}\subset \ocD$.
We now prove that the ball $B(z_0,h_0)$ is petite.
Recall that 
the stopping time $\tau_{\partial B(z_0,2h_0)}:=\inf\{t\in\RP: Z_t\notin B(z_0,2h_0)\}$
is strictly positive $\P_z$-a.s. for all 
$z\in B(z_0,h_0)$.
Moreover, the process $Z$ on the stochastic interval $[0,\tau_{\partial B(z_0,2h_0)})$, started at any $z\in B(z_0,h_0)$, coincides with a uniformly elliptic diffusion on $\R^{d+1}$, stopped upon exiting the ball $B(z_0,2h_0)$.
Thus, by~\cite[Thm~II.1.3]{stroock1988diffusion}, we have $\inf_{z\in B(z_0,h_0)} \P_z(Z_{h_0^2}\in \cdot)\geq \varphi(\cdot)$, where 
$\varphi$ is the Lebesgue measure multiplied by a positive scalar and supported in $B(z_0,h_0)$. 
%for some non-trivial measure $\varphi_{a_1}$ on $(\cD,\cB(\cD))$. 
Hence, for any $z\in B(z_0,h_0)$,  condition~\eqref{eq::petite} holds with $a_1(\ud t) = \delta_{h_0^2}(\ud t)$, where $\delta_{h_0^2}$ is the Dirac delta concentrated at $h_0^2>0$, and taking the non-trivial measure $\varphi_{a_1}:=\varphi$ on $(\cD,\cB(\cD))$. 

We now prove that  an arbitrary compact set $D$ in $\cD$ is also petite. Since $B(z_0,h_0)$ is an open set, Theorem~\ref{thm:Feller_continuity_of_Z} and~\cite[Thm~2.1]{Billingsley1999} imply that the function $z\mapsto\P_z(Z_1\in B(z_0,h_0))$ is lower semi-continuous on $\cD$. Moreover, by Proposition~\ref{prop:marginal_equivalent_to_Lebesgue}, we have $\P_z(Z_1\in B(z_0,h_0))>0$ for all $z\in D$. Hence, by compactness of $D$ and the lower semi-continuity of the function $z\mapsto\P_z(Z_1\in B(z_0,h_0))$, we get $\inf_{z\in D}\P_z(Z_1\in B(z_0,h_0))>0$. 
%Since $B(z_0,h_0)$ is a petite set with probability %measure $\delta_{h_0^2}$ and non-zero measure %$\varphi_{a_1}$ we obtain, using also 
For any $z\in D$, the Markov property of $Z$ implies $$\P_z(Z_{1+h_0^2}\in \cdot)\geq \P_z(Z_1\in B(z_0,h_0))  \inf_{z'\in B(z_0,h_0)} \P_{z'}(Z_{h_0^2}\in \cdot)\geq 
\inf_{z''\in D}\P_{z''}(Z_1\in B(z_0,h_0))\varphi_{a_1}(\cdot).$$ 
Thus the set $D$ satisfies condition~\eqref{eq::petite}   with the probability measure $a_2(\ud t) = \delta_{1+h_0^2}(\ud t)$ and non-trivial measure $\varphi_{a_2} := \inf_{z\in D}\P_z(Z_1\in B(z_0,h_0))\varphi_{a_1}$ on $(\cD,\cB(\cD))$, making $D$ petite.

 In particular, the set $\cD \cap [r_0,\infty)\times \R^d$ is a petite set for every $r_0\in(0,\infty)$. Moreover, if $\beta<\beta_c$, Proposition~\ref{prop:return_time_upper_bound} implies that, for all sufficiently large  $r_0\in(0,\infty)$, we have $\P_z(\varsigma_{r_0} < \infty) = 1$ for all $z \in \cD$. We conclude that the process $Z$ is Harris recurrent by \cite[Thm.~1.1]{tweedie1993generalized}. 
\end{proof}

%In the proof of %Proposition~\ref{harris_irreducible_petite} we saw that % petite sets are linked to the local behaviour of the %process $Z$, while Harris recurrence 
%requires information about its global behaviour (via %Proposition~\ref{prop:return_time_upper_bound}).

%Proposition~\ref{prop:bounded_petite} asserts that all %petite sets are bounded. 
The following proposition is crucial for establishing the lower bounds on the tail of the invariant distribution  of the reflected process $Z$ in Theorem~\ref{thm:invariant_distributon} (the proof of Theorem~\ref{thm:invariant_distributon} requires an estimate of the return times to an arbitrary petite set).

\begin{prop}
\label{prop:bounded_petite}
Suppose that $\aref{ass:covariance2},$ $\aref{ass:vector2}$ and $\aref{ass:domain2}$ hold with $\beta < \beta_c$. Then every petite set is bounded.
\end{prop}

\begin{proof}
 Let $B$ be an arbitrary petite set and let the probability measure $a$ on $\R_+$ and a non-zero measure $\varphi_a$ on $(\cD,\cB(\cD))$ be such that~\eqref{eq::petite} holds,
 with $K_a$ as in~\eqref{eq:Markov_transition_kernel}. 
 Denote $\cD^{(r)}=\cD \cap [0,r]\times\R^d$ for any $r>0$. 
  Since $\varphi_a$ is a non-trivial measure on $(\cD,\cB(\cD))$, there exists $r_0\in(0,\infty)$, such that $c:=\varphi_a(\cD^{(r_0)})> 0$.
 
%By definition of the return time $\varsigma_{r_0}$ %in~\eqref{eq::varsigma}, on the event $\{t_0 \leq %\varsigma_{r_0}\}$, we have $\P_z(Z_t \notin %\cD^{(r_0)})=1$ for all $t\in[0,t_0)$, $\P_z$-a.s. for %all $z\in\cD$. Thus, it is sufficient to 

 Proposition \ref{prop:return_time_lower_bound} implies that for every $q\in(0,1)$ and $p\in(1-\beta/\beta_c,\infty)$, there exist $r_1\in(0,\infty)$, $c_1,c_2\in(1,\infty)$ and $\eps\in(0,\infty)$, such that for every $r\in(r_1,\infty)$ and $z=(x,y)\in\cD$ we have
\begin{equation}
\label{eq:lower_bound_return_time}
\P_z(\varsigma_r \geq t_0) \geq q\min\{(c_1^{-p}(x-c_2)^p-r^p)(1-q)^p\eps^{-1/2} t_0^{-p/2},1\}
\end{equation}
for all $t_0\in(c_1 r_1+c_2,\infty)$
(recall that the return time $\varsigma_{r_0}$ is defined in~\eqref{eq::varsigma}).
Since $a$ is a probability measure, there exists $t_0\in(c_1 r_1+c_2,\infty)$ satisfying
$a([t_0,\infty)) < c/2$.

We now show that there exists $x_0\in(0,\infty)$, such that 
$\P_z(\varsigma_{r_0} <t_0) < c/2$
for all $z\in\cD\setminus\cD^{(x_0)}$. 
Indeed, fix
$r \geq \max\{r_0,r_1\}$ (note  $\varsigma_{r}\leq\varsigma_{r_0}$), $q \in(1-c/2,1)$ and $p\in(1-\beta/\beta_c,\infty)$. 
Pick $x_0>0$, such that $(c_1^{-p}(x-c_2)^p-r^p)(1-q)^p\eps^{-1/2} t_0^{-p/2} \geq 1$ for all $x\in(x_0,\infty)$. 
Note that this choice of $x_0$ implies 
$x_0>r$ and, in particular, $x_0\in(r_0,\infty)$.
For any $z=(x,y)\in\cD\setminus\cD^{(x_0)}$, the inequality in~\eqref{eq:lower_bound_return_time}
implies $\P_{z}(\varsigma_{r_0} < t_0) \leq\P_{z}(\varsigma_{r} < t_0) < 1-q <c/2$. 

By~\eqref{eq:Markov_transition_kernel}, we have 
$K_a(z,\cD^{(r_0)})\leq\int_0^{t_0}\P_z(Z_t\in\cD^{(r_0)})a(\ud t)+a([t_0,\infty))$.  Since, for all $z\in\cD\setminus\cD^{(x_0)}$, we have $\P_z(Z_t\in\cD^{(r_0)})\leq \P_{z}(\varsigma_{r_0} < t_0)<c/2$ for all $t\in[0,t_0]$, the inequality 
$\varphi_a(\cD^{(r_0)}) =c> K_a(z,\cD^{(r_0)})$ holds for all
$z\in\cD\setminus\cD^{(x_0)}$. 
Since the petite set $B$ satisfies~\eqref{eq::petite},
we must have $B\subset \cD^{(x_0)}$, making $B$ bounded.
\end{proof}

\section{Stability: the proof of Theorem~\ref{thm:invariant_distributon}} 
\label{subsection:main_proofs}

Existence and uniqueness of the invariant distribution of $Z$, the upper bounds on the tails of
the invariant distribution and upper bounds on the rate of convergence of $Z$ to stationarity will be established using the drift condition (i.e.~supermartingale property) given in Lemma~\ref{Lem:drift_conditions} (see Section~\ref{subsection:return_times_drift_conditions}) and the fact that every compact set in $\cD$ is petite for the reflected process $Z$ (see Proposition~\ref{prop:harris_irreducible_petite} in Section~\ref{subsubsec:continuous_markov_theory} above). The lower bounds on the tails of the invariant distribution and the rate of convergence to stationarity will follow from the fact that every petite set of $Z$ is bounded (see Proposition~\ref{prop:bounded_petite} above) and the control we have established on the return time and length of excursions away from bounded sets (see Proposition~\ref{prop:return_time_lower_bound}).
Theorem~\ref{thm:invariant_distributon} follows easily from Propositions~\ref{prop:upper_bounds} and~\ref{prop:lower_bound_invariant} proved in this section. 

 \subsection{Existence, uniqueness, and upper bounds}
The upper bounds on the tails of the invariant distribution are obtained by establishing finiteness of certain moments and applying the Markov inequality. In Section~\ref{sec:Lower_bounds} below we show that these bounds cannot be improved. 
 
\begin{prop}
\label{prop:upper_bounds}
Suppose that $\aref{ass:covariance2},$ $\aref{ass:vector2}$ and $\aref{ass:domain2}$ hold with $\beta < -\beta_c$
and recall 
$M_c = -(1+\beta/\beta_c)/2>0$.
Then there exists the unique invariant distribution $\pi$ on $(\cD,\cB(\cD))$  for the process $Z$. Pick $\epsilon>0$. There exists a constant $C_\pi\in(0,\infty)$, such that
$$
\pi(\{z\in\cD:\|z\|_{d+1}\geq r\})\leq C_\pi r^{-2M_c+\epsilon}\qquad\text{for all $r\in[1,\infty)$.}
$$
Furthermore, for every  $z\in\cD$ there exists a constant $C_{\mathrm{TV}}\in(0,\infty)$, such that 
 \begin{align*}
 \label{eq:upper_bound_rate_of_convergnece}
\| \P_z(Z_t\in\cdot)-\pi\|_{\mathrm{TV}} & \leq  C_{\mathrm{TV}}t^{-M_c + \eps}
\quad\text{for all $t\in[1,\infty)$.}
\end{align*}
%The constant $C_2$ dependents on the starting point $z \in \cD$.
\end{prop}

\begin{proof}
By~Proposition~\ref{prop:marginal_equivalent_to_Lebesgue}, the process $Z$ admits an irreducible skeleton chain. Moreover, by Proposition~\ref{prop:harris_irreducible_petite} the sets $\cD^{(r)}=\cD\cap[0,r]\times\R^d$, defined in~\eqref{eq:D_r}, are petite for every $r>0$. Pick arbitrary $\eps\in(0,1-\beta/\beta_c-2)$ and note that $\gamma:=1-\beta/\beta_c - \eps>2$. By Lemma~\ref{Lem:drift_conditions}, the process $\xi$, defined in~\eqref{eq::F_w_gamma_supermart}, is a supermartingale. Note that by definition of $F_{w,\gamma}$ in~\eqref{eq:F_w,gamma}, we have $F_{w,\gamma}(z)^{(\gamma-2)/\gamma} = F_{w,\gamma-2}(z)$ for $z\in\cD\cap[x_1,\infty)\times\R^d$, where $x_1$ is the constant appearing in the definition of the function $m$ in~\eqref{eq:F_w,gamma}. We may thus apply~\cite[Prop.~3.1]{douc2009subgeometric} (with $V=F_{w,\gamma}$, $\phi(u) = C_1 u^{(\gamma-2)/\gamma}$, $b=C_2$ and the petite set $\cD^{(x_2)}$ from Lemma~\ref{Lem:drift_conditions}), to deduce the existence and uniqueness of the invariant distribution $\pi$ of $Z$ and $\int_{\cD}F_{w,\gamma-2}(z)\pi(\ud z)<\infty$. From the definition of $F_{w,\gamma}(z)$ in~\eqref{eq:F_w,gamma} and the lower bound in~\eqref{eq:bound_Lyapunov_f} it follows that $\tilde C_\pi:=\int_{z=(x,y)\in\cD}x^{2M_c-\eps}\pi(\ud z)<\infty$. Moreover, Markov's inequality implies $$\pi(\cD\cap [r,\infty)\times \R^d)\leq \int_\cD (x/r)^{2M_c-\epsilon} \pi(\ud z)\leq \tilde C_\pi r^{-2M_c+\epsilon}\qquad\text{for all $r\in[1,\infty)$.}
$$

Recall that 
for any $z = (x,y)\in\cD$ we have $x \leq \|z\|_{d+1}\leq (x^2+b(x)^2)^{1/2}$ and the boundary function $b$ has sublinear growth (cf. Remark~\ref{rem:Ass2}), implying  $(x^2+b(x)^2)^{1/2}/x \to 1$ as $x \to \infty$. Thus, the upper bound $\pi(\cD\cap [r,\infty)\times\R^d)\leq \tilde C_\pi r^{-2M_c+\epsilon}$ implies the existence of the constant $C_\pi\in(0,\infty)$ such that the bound on the tail  $\pi(\{z \in \cD:\|z\|_{d+1} \geq r\})\leq C_\pi r^{-2M_c+\epsilon}$ holds for all $r\in[1,\infty)$ as claimed in the proposition.

Recall that $\lebm_{d+1}$ is a Lebesgue measure on $\R^{d+1}$. The process $Z$ admits an \textit{$\lebm_{d+1}$-irreducible skeleton} chain, since for every $A\in\cB(\cD)$, such that  $\lebm_{d+1}(A) > 0$, and $z \in \cD$, by Proposition~\ref{prop:marginal_equivalent_to_Lebesgue}, we have $\P_{z}(Z_{k}\in A) > 0$ for every $k \in \N$. In particular, Assumption~(i) in~\cite[Thm~3.2]{douc2009subgeometric} is satisfied for $Z$.
For $\gamma = 1-\beta/\beta_c-\eps$,  Lemma~\ref{Lem:drift_conditions} ensures that Assumption~(ii) of~\cite[Thm~3.2]{douc2009subgeometric} is satisfied. By~\cite[Thm~3.2, Eq.~(3.5)]{douc2009subgeometric}, with the pair of functions $\Psi = (\text{Id},\textbf{1})$, where $\text{Id},\textbf{1}:\RP\to\RP$ denote the identity and the constant functions, respectively, we obtain
$$
r_*(t)\| \P_z(Z_t\in\cdot)-\pi\|_{\mathrm{TV}} \leq  F_{w,\gamma}(z) \quad \text{ for all $t\geq 0$,}
$$
where $r_*(t) = \varphi \circ H_\varphi^{-1}(t)$. Here, the function $\varphi$ is positive, satisfying $\varphi(u) = \tilde C_1u^{(\gamma-2)/\gamma}$ for $u\geq 1$, and $H_\varphi^{-1}$ is the inverse of the increasing function $H_\varphi$, satisfying $H_\varphi(u) = \int_1^u \varphi(s)^{-1}\ud s$ for $u\geq 1$. This implies $r_*(t) = \tilde C_2 (t+1)^{\gamma/2-1}$ for a positive constant $\tilde C_2\in(0,\infty)$ and $t\in(0,\infty)$. 
\end{proof}

\subsection{Lower bounds} 
\label{sec:Lower_bounds}
The lower bounds on the tails of the invariant distribution $\pi$ of the reflected process $Z$ 
are closely related to the tail behaviour of 
certain additive functionals of the paths of $Z$ until the return time to a petite set. For a measurable subset $D \subset \cD$ and $\delta > 0$, define the return time of the process $Z$  to the set $D$ after the time $\delta>0$ by $\tau_\delta(D) := \inf \{ t > \delta: Z_t \in D\}$ (with convention $\inf\emptyset=\infty$).

\begin{prop}
\label{prop:bounded_set_excursions}
Suppose that~\aref{ass:domain2}, \aref{ass:covariance2}, \aref{ass:vector2} hold with $\beta < \beta_c$ and pick $p \in(1-\beta/\beta_c,\infty)$, a bounded measurable set $D\subset\cD$ and $z\in\cD$. Then there exist constants 
$\delta \in (0,\infty)$ and 
$C,r_0,c_1,c_2,\epsilon\in(0,\infty)$, such that for every 
non-decreasing continuous function $H:\RP\to\RP$
with $r_H:=\inf\{r'\geq:H(r')>0\}<\infty$,
we have
$$
\P_z\left(\int_0^{\tau_\delta(D)} H\left(c_1(X_s +c_2)\right)\ud s \geq r\right) \geq C/ G(r/\epsilon)^p\quad\text{for all $r\in(r_0,\infty)$,}
$$
where $G:\RP\to\RP$ is the inverse of the strictly increasing function defined on $[r_H,\infty)$ by the formula $v\mapsto v^2 H(v)$.
%In particular $\P(\tau_\delta(D) \geq r) \geq \tilde C r^{-p/2}$, for some constant $\tilde C\in(0,\infty)$.
\end{prop}

\begin{rem}
    (a) The assumption $\beta < \beta_c$ in Proposition~\ref{prop:bounded_set_excursions}
    covers both the null-recurrent and positive-recurrent cases.
    However, the main application of Proposition~\ref{prop:bounded_set_excursions} in the proof of 
    Lemma~\ref{lem:lower_bound_invariant}
     below requires only the positive-recurrent case.
     Since Lemma~\ref{lem:lower_bound_invariant} is crucial in the proof of 
     Proposition~\ref{prop:lower_bound_invariant}, 
the bound in Proposition~\ref{prop:bounded_set_excursions} is key for 
     the lower bounds in Theorem~\ref{thm:invariant_distributon}.
     \\
    (b) The statement of Proposition~\ref{prop:bounded_set_excursions}
    in fact holds for every $\delta>0$. 
    In the proof of Proposition~\ref{prop:bounded_set_excursions} below, we apply the non-confinement property of $Z$, given in Lemma~\ref{lem:non_cofinment_f_F}, to conclude that $\delta>0$  exists. 
    However, since  by Proposition~\ref{prop:marginal_equivalent_to_Lebesgue} $Z$ is irreducible, the event  $\{X_\delta > c_1(d_0+1)+c_2\}$ has positive probability for every $\delta>0$. Since the existence of $\delta>0$ is sufficient for our analysis of the lower bound on the tail of the invariant distribution and non-confinement is weaker (and easier to prove) than irreducibility, we use the formulation of Proposition~\ref{prop:bounded_set_excursions} above.\\
    (c) The assumed continuity of the function $H$ in Proposition~\ref{prop:bounded_set_excursions} is not necessary: measurability would be sufficient but it would complicate the formulation of the proposition.
\end{rem}

\begin{proof}[Proof of Proposition~\ref{prop:bounded_set_excursions}]
Pick  a non-decreasing continuous function $H:\RP\to\RP$ 
with $r_H=\inf\{r':H(r')>0\}<\infty$.
Proposition~\ref{prop:return_time_lower_bound} implies that for every $p \in (1-\beta/\beta_c,\infty)$ and $q\in(0,1)$, there exist constants $x_0\in(0,\infty)$, $c_1,c_2\in(1,\infty)$ and $\eps(q)\in(0,\infty)$, such that for every $x_1\in(x_0,\infty)$ and $z=(x,y)\in\cD\cap (c_1x_1+c_2,\infty)\times\R^d$ and function $\bar H(r):=H(c_1(r+c_2))$, $r\in\RP$, we have
\begin{equation}
\label{eq:lower_bound_ecurtions_final}
\P_z\left(\int_0^{\varsigma_{x_1}} \bar H\left(X_s\right)\ud s \geq \eps v^2H(v)\right) \geq q\min\{(c_1^{-p}(x-c_2)^p-x_1^p)(1-q)^{p}v^{-p},1\},
\end{equation}
for all $v\in(c_1x_1+c_2,\infty)$.
Recall here that $\varsigma_{x_1}$, defined in~\eqref{eq::varsigma} above,  is the return time of the first coordinate $X$ (of $Z$) below the level $x_1$ and $X_0=x$ $
\P_z$-a.s.
Note that the following elementary inequality holds since the function is monotonically increasing as $p>0$:
$$c_1^{-p}(x-c_2)^{p}-x_1^{p}\geq (x_1+1)^p-x_1^p\quad\text{for all $x\in(c_1(x_1+1)+c_2,\infty)$.}
$$
%\begin{itemize}
%    \item $x^\alpha\geq c_1^\alpha(x + c_2)^\alpha/(c_1+c_2)^\alpha$ for %all $x\in (x_1+c_1,\infty)$;
%    \item $c_1^{-p}(x-c_2)^{p}-x_1^{p}\geq (x_1+1)^p-x_1^p$ for all %$x\in(c_1(x_1+1)+c_2,\infty)$.
%\end{itemize}
Since $v\mapsto v^2H(v)$ is strictly increasing on $[r_H,\infty)$ with range equal to $\RP$, for any $r\in(c_1x_1+c_2,\infty)$ we 
can define $v:=G(r/\epsilon)$, implying
$r=\epsilon v^2 H(v)$. 
 Thus, for every $z=(x,y)\in\cD \cap (c_1(x_1+1)+c_2,\infty)\times\R^d$
 and any  $C_{x_1}\in(0,q((x_1+1)^p-x_1^p)(1-q)^p)$
 we have
\begin{align}
\label{eq:lower_bound_additive_functional_Z}
\P_z\left(\int_0^{\varsigma_{x_1}} \bar H\left(X_s\right)\ud s \geq r\right)
%\P_z\left(\int_0^{\varsigma_{x_1}} X_s^{\alpha}\ud s \geq r\right)
&\geq q\min\{((x_1+1)^p-x_1^p)(1-q)^p/G(r/\epsilon)^p,1\}\\
%(c_1^{-p}(x-c_2)^p-x_1^p)%(c_1+c_2)^{\alpha/(2+\alpha)}\eps^{-1/(2+\alpha)}r^{-p/(2+\alpha)},1\}\\
\nonumber &\geq C_{x_1} /G(r/\epsilon)^p \quad\text{for all $r\in(c_1x_1+c_2,\infty)$,}
\end{align}
where the first inequality in~\eqref{eq:lower_bound_additive_functional_Z} follows from~\eqref{eq:lower_bound_ecurtions_final} above.

Fix $z\in\cD$ and a bounded set $D\subset \cD$.
Let $m_D:=\sup\{x:(x,y)\in D\}\in\RP$ satisfy $D\subset [0,m_D]\times\R^d\cap\cD$ (note that $m_D<\infty$ by assumption on $D$) 
and denote $d_0 := \max\{m_D,x_1\}$. Since, by Lemma~\ref{lem:non_cofinment_f_F}, $X$ is not confined
to any compact set, there exists $\delta > 0$, such that $\P_z(X_\delta > c_1(d_0+1)+c_2)  > 0$. 
Note that, on the event 
$\{X_\delta > c_1(d_0+1)+c_2\}$, $\P_z$-a.s.
we have 
$Z_\delta\in\cD \cap (c_1(x_1+1)+c_2,\infty)\times\R^d$. The
Markov property of $Z$ and~\eqref{eq:lower_bound_additive_functional_Z} thus imply that for every $p\in(1-\beta/\beta_c,\infty)$ and 
$r_0:=c_1 d_0+c_2$,
there exists a constant $C_{d_0}\in(0,q((d_0+1)^p-d_0^p))$, such that
\begin{align*}
\P_z\left(\int_0^{\tau_\delta(D)} \bar H\left(X_s\right) \ud s \geq r\right) &
\geq \P_z\left( \int_0^{\tau_\delta(D)}   \bar H\left(X_s\right)\ud s \geq r, X_\delta > r_0+c_1\right)\\
&
\geq \E_z\left[\1{X_\delta > r_0+c_1}\cdot\P_{Z_\delta}\left( \int_0^{\tau_\delta(D)}  \bar H\left(X_s\right)\ud s \geq r\right)\right]\\
&
\geq \E_z\left[\1{X_\delta > r_0+c_1}\cdot\P_{Z_\delta}\left( \int_0^{\varsigma_{d_0}}  \bar  H\left(X_s\right)\ud s \geq r\right)\right]\\
&
\geq 
\P_z(X_\delta > r_0+c_1) C/G(r/\epsilon)^p \quad\text{for all $r\in(r_0,\infty)$.} \qedhere
\end{align*}
\end{proof}

We can now establish the lower bounds on the tail of the the invariant distribution of $Z$.

\begin{prop}
\label{prop:lower_bound_invariant}
Suppose that~\aref{ass:domain2}, \aref{ass:covariance2}, \aref{ass:vector2} hold with $\beta < -\beta_c$ and recall $M_c = -(1+\beta/\beta_c)/2$. Let $\pi$ be the invariant distribution of the process $Z$ and pick $\eps>0$. There exists a  constant
$c_\pi\in(0,\infty)$
 such that
$$
c_\pi r^{-2M_c-\epsilon} \leq \pi(\{z\in\cD:\|z\|_{d+1}\geq r\})\qquad\text{for all $r\in[1,\infty)$.}
$$
Furthermore, for any $z\in\cD$, there exists a constant $c_{\mathrm{TV}}\in(0,\infty)$, such that 
$$
c_{\mathrm{TV}}t^{-M_c-\epsilon} \leq \| \P_z(Z_t\in\cdot)-\pi\|_{\mathrm{TV}} \qquad\text{for all $t\in[1,\infty)$.}
$$
\end{prop}

The key step in the proof of Proposition~\ref{prop:lower_bound_invariant}
is the following lemma.

\begin{lem}
\label{lem:lower_bound_invariant}
Suppose that~\aref{ass:domain2}, \aref{ass:covariance2}, \aref{ass:vector2} hold with $\beta < -\beta_c$.
For any $\epsilon>0$ there exists a constant $c_\pi\in(0,\infty)$,
such that 
$c_\pi r^{1+\beta/\beta_c-\eps}\leq\pi(\cD\cap [r,\infty)\times\R^d)$ for all $r\in[1,\infty)$.
\end{lem}

\begin{proof}[Proof of Lemma~\ref{lem:lower_bound_invariant}]
The reflected process $Z$ is \textit{positive Harris recurrent}, i.e., $Z$ is Harris recurrent (by Proposition~\ref{prop:harris_irreducible_petite}) and admits an invariant distribution (by Proposition~\ref{prop:upper_bounds}). Thus, by~\cite[Theorem 1.2(b)]{tweedie1993generalized}, a measurable function $f: \cD \rightarrow [1,\infty)$ satisfies $\int_{z \in \cD}f(z)\pi(\ud z) < \infty$ if and only if $\sup_{z \in D}\E_z[\int_0^{\tau_\delta(D)}f(Z_s)\ud s]< \infty$ for some closed petite set $D\subset\cD$ and all $\delta > 0$, where $\tau_\delta(D)= \inf \{ t > \delta: Z_t \in D\}$ is the return time to the set $D$ after time $\delta$ (defined before the statement of Proposition~\ref{prop:bounded_set_excursions} above). Since by Proposition~\ref{prop:bounded_petite} all petite sets are bounded, Proposition~\ref{prop:bounded_set_excursions}
implies that for any closed petite set $D$
of $Z$ there exists $\delta>0$ such that for any
non-decreasing continuous function $\tilde H:\RP\to[1,\infty)$
we have 
$\E_z[\int_0^{\tau_\delta(D)}\tilde H(c_1(X_s+c_2))\ud s]\geq C\int_{r_0}^\infty G(r/\epsilon)^{-p}\ud r$
for all $z\in\cD$ and $p\in(1-\beta/\beta_c,\infty)$,
where $G:\RP\to\RP$
is the inverse of the function $v\mapsto v^2\tilde H(v)$
and the positive constants $\epsilon, r_0, C, c_1,c_2$ are as in Proposition~\ref{prop:bounded_set_excursions}.
By the criterion in~\cite[Theorem 1.2(b)]{tweedie1993generalized} stated above we thus obtain that 
\begin{equation}
\label{eq:integtal_of_inverse_G_infinite}
\int_{r_0}^\infty G(r/\epsilon)^{-p}\ud r = \infty\quad \text{ implies }\quad
\int_{z=(x,y)\in\cD} \tilde H(c_1(x+c_2))\pi(\ud z)=\infty.
\end{equation}

The proof of Lemma~\ref{lem:lower_bound_invariant} proceeds by contradiction. Note that the statement in the lemma is equivalent to the following: for every $\epsilon>0$
there exists $r_0\in(0,\infty)$ such that $$r^{1+\beta/\beta_c-\eps}\leq\pi(\cD\cap [r,\infty)\times\R^d)\qquad\text{for all $r\in[r_0,\infty)$.}$$
Assume to the contrary that there exists $\epsilon>0$, such that \textit{for every} $r_0>0$ there exists $r_1\in[r_0,\infty)$ satisfying
$r_1^{1+\beta/\beta_c-\eps}>\pi(\cD\cap [r_1,\infty)\times\R^d)$.
We may pick $r_0>1$ and $r_1>2r_0$. Using a recursive construction, we obtain a sequence $(r_n)_{n\in\N}$, such that 
$r_{n+1}>2r_n$ and $r_n^{1+\beta/\beta_c-\eps}>\pi(\cD\cap [r_n,\infty)\times\R^d)$ for all $n\in\N$.
Using this sequence, we now construct a function $H$ 
satisfying $\int_\cD H(x)\pi(\ud z)<\infty$ 
%with finite expectation with respect to the invariant %distribution $\pi$, 
but which violates the implication in~\eqref{eq:integtal_of_inverse_G_infinite}.
%the preceding paragraph.

Set $\alpha := -(1+\beta/\beta_c)+\eps$ and define the function 
$\mu:\RP\to\RP$ by 
$\mu(x):=1$ for $x\in[0,r_1)$ and
$\mu(x) := r_n^{-\alpha}$ for $x\in [r_n,r_{n+1})$, $n\in\N$.  
Since the function $x\mapsto \pi(\cD\cap [x,\infty)\times\R^d)$ is non-increasing, we have  
$\pi(\cD\cap [x,\infty)\times\R^d)\leq \mu(x)$ for all $x\in\RP$. 
Let $H:\RP\to[1,\infty)$ be a differentiable function such that $H(x) = 1$ 
for $x\in[0,r_1)$ and, for $x\in[r_n,r_{n+1})$, we have
$$
H'(x) = \begin{cases}
r_n^{\alpha-\eps/2},\quad x\in [r_n,r_{n}+1);\\
1/(r_{n+1}-r_n),\quad x\in [r_n+1,r_{n+1}).
\end{cases}
$$
Since $\mu$ is non-increasing by definition, for $x\in[r_n,r_{n+1})$, we have $\mu(x) \leq r_n^{-\alpha}$ and 
$$
H'(x)\mu(x) = \begin{cases}
r_n^{-\eps/2},\quad x\in[r_n,r_n+1); \\
r_n^{-\alpha}/(r_{n+1}-r_n),\quad x\in[r_n+1,r_{n+1}).
\end{cases}
$$
The inequality $\pi(\cD\cap[x,\infty)\times\R^d)\leq \mu(x)$ and Tonelli's theorem (with $H(0)=1$) imply
\begin{align}
\label{eq:finite_H_pi}
\nonumber\int_{z=(x,y)\in\cD} H(x)\pi(\ud z) &\leq 1+\int_{r_1}^\infty H'(x)\pi(\cD \cap [x,\infty)\times\R^d)\ud x 
 \leq 1+ \int_{r_1}^\infty H'(x)\mu(x)\ud x\\
&\nonumber=1+\sum_{n=1}^\infty \left(\int_{r_n}^{1+r_n}H'(x)\mu(x)\ud x + \int_{1+r_n}^{r_{n+1}}H'(x)\mu(x)\ud x \right)\\
&\leq1+\sum_{n=1}^{\infty}r_n^{-\eps/2} + \sum_{n=1}^{\infty} r_n^{-\alpha}<\infty,
\end{align}
where the final inequality follows from $2^{n-1}r_1\leq r_n$ for every $n\in\N$ and $\alpha>\eps>0$.

Let $\eps>0$ and $c_1, c_2\in(1,\infty)$ whose existence is guaranteed by Proposition~\ref{prop:bounded_set_excursions}. Let
$\tilde H:\RP\to[1,\infty)$ be a non-decreasing differentiable function  satisfying $\tilde H(x)=H(x/c_1- c_2)$ for all $x\in(c_1 c_2,\infty)$.
 %For all $x\in(c_1 (r_n+1+c_2),\infty)$, from the definition %of $H(x)$, we deduce that $$\tilde H(x)=H(x/c_1-c_2)>
 %%\tilde H(x)-\tilde H(c_1(r_n+c_2))=
 %H(x/c_1-c_2)-H(r_n)>r_n^{\alpha-\eps/2}$$ 
 %and thus $x^2\tilde H(x)>%c_1^2 (r_n+1+c_2)^2r_n^{\alpha-%\eps/2}>
 %x^2r_n^{\alpha-\eps/2}$. % (recall $c_1,c_2>1$). 
 %Since $\tilde H$ is strictly positive on $\RP$, there exists %an inverse $G:\RP\to \RP$ of the function $x\mapsto %x^2\tilde H(x)$. Moreover, for every, $n\in\N$, the %inequality $G(x)< r_n^{1/(2+\alpha-\eps/2)}/2$ holds for all %$x\in [c_1(r_n+1+c_2),2c_1(r_n+1+c_2))$. 
Pick $p\in(1-\beta/\beta_c,1-\beta/\beta_c+\eps/2)$ and note
that $p>2$ and
$2-p+\alpha-\eps/2>0$.
Let $G:\RP\to\RP$ be the inverse of the function $u\mapsto u^2 \tilde H(u)$. 
Introduce the substitution 
 $r=\epsilon u^2 \tilde H(u)$
 into the first integral in~\eqref{eq:integtal_of_inverse_G_infinite} to obtain
 \begin{align}
 \nonumber
 \int_{1}^{\infty}G(r/\eps)^{-p}\ud r &= \eps\int_{G(1/\epsilon)}^{\infty} u^{-p}(2u\tilde H(u) + u^2\tilde H'(u))\ud u \\
 \nonumber
&\geq \eps\sum_{n=n_0}^{\infty}\int_{c_1(r_n+c_2)}^{c_1(2r_n+c_2)} u^{1-p}\tilde H(u)\ud u=
\eps\sum_{n=n_0}^{\infty}H(r_n)\int_{c_1(r_n+c_2)}^{c_1(2r_n+c_2)} u^{1-p}\ud u
 \\
 \nonumber
 &=\eps/(p-2)\sum_{n=n_0}^{\infty} r_n^{\alpha-\eps/2} c_1^{2-p}((r_n-c_2)^{2-p}-(2r_n-c_2)^{2-p}) \\
 \nonumber
 &\geq \eps/(p-2)\sum_{n=n_0}^{\infty} r_n^{\alpha-\eps/2}(r_n-c_2)^{2-p} (1-((2r_n-c_2)/(r_n-c_2))^{2-p})\\
 \label{eq:lower_bound_inverse_G}
 &\geq \eps(1-2^{2-p})/(p-2)\sum_{n=n_0}^{\infty} (r_n-c_2)^{2-p+\alpha-\eps/2} =\infty,
 \end{align}
where $n_0\in\N$ is sufficiently large so that  
$c_1(r_n+c_2)>G(1/\epsilon)$ and $r_n>c_2$ hold for all $n\geq n_0$. The first inequality in the previous display uses the fact that $\tilde H$ is non-decreasing and positive, while the second follows from $c_1>1$ and $2-p<0$.

Note that $\tilde H(c_1(x+c_2))=H(x)$ for every $x\in\RP$.
Thus, the implication  in~\eqref{eq:integtal_of_inverse_G_infinite} 
and the estimate in~\eqref{eq:lower_bound_inverse_G}
yield   
 $$\infty  = \int_{z=(x,y)\in\cD}\tilde H(c_1(x+c_2))\pi(\ud z) = \int_{z=(x,y)\in\cD}H(x)\pi(\ud z).$$ 
 This contradicts~\eqref{eq:finite_H_pi} and concludes the proof of the lemma.
\end{proof}

\begin{proof}[Proof of Proposition~\ref{prop:lower_bound_invariant}] 
Note that 
$\pi(\{z\in\cD:\|z\|_{d+1}\geq r\}) \geq \pi(\cD \cap [r,\infty)\times\R^d)$
for every $r\in\RP$.
Thus, the lower bound on the invariant distribution $\pi$ follows form Lemma~\ref{lem:lower_bound_invariant}.

Pick arbitrary $\eps'\in(0,1)$. Lemma~\ref{Lem:drift_conditions} implies that for  and $\gamma := 1-\beta/\beta_c - \eps'$, there exists $w\in(-\infty,1-\beta s_0/c_0)$, some constant $x_0,x_1,d\in\RP$, defining the function $F_{w,\gamma}$, and a constant 
$C_3\in(0,\infty)$
such that $\E_{z}[F_{w,\gamma}(Z_t)] \leq F_{w,\gamma}(z)+C_3t$ for all $t\in\RP$ and $z\in\cD$. Since $F_{w,\gamma}(z) = f_{w,\gamma}(z)$ on $z=(x,y)\in\cD \cap[x_1,\infty)\times\R^d$, the upper bound in~\eqref{eq:bound_Lyapunov_f} 
implies that 
\begin{align*}
\pi(\{z\in\cD:F_{w,\gamma}(z) \geq r\}) &\geq \pi(\{z=(x,y)\in\cD:x \geq 2^{\gamma/|w|}r^{1/\gamma}-k_w\})\quad \text{ for all $r\in(x_1,\infty)$.}
\end{align*}
This inequality and the lower bound on the invariant distribution 
in Lemma~\ref{lem:lower_bound_invariant}
imply that there exists a constant $C_4\in(0,1)$ such that 
\begin{align*}
\pi(\{z\in\cD:F_{w,\gamma}(z) \geq r\}) & \geq  C_4r^{(1 + \beta/\beta_c - \epsilon')/\gamma}= C_4r^{(1 + \beta/\beta_c - \epsilon')/(1-\beta/\beta_c-\epsilon')}
\end{align*}
for all $r\in(x_1,\infty)$. By further reducing $C_4>0$
if necessary, we may assume that the inequality in the last display holds for all $r\in[1,\infty)$.

 Define the functions $f: [1,\infty) \to (0,1]$, $f(a)  := C_4 a^{(1 + \beta/\beta_c - \epsilon')/(1-\beta/\beta_c-\epsilon')}$
 and $F(a):=af(a)= C_4 a^{(2 -2\epsilon')/(1-\beta/\beta_c-\epsilon')}$ and note that $F$ is strictly increasing with  $\lim_{a \to \infty} F(a) = \infty$. By Lemma~\ref{lem:lower_bound_convergence_rate}, applied  with  functions $G(z) = F_{w,\gamma}(z)$ (recall from the previous display that $\pi(\left\{z\in\cD:G(z)\geq r\right\})\geq f(r)$ for all $r\in[1,\infty)$) and  $g(z,t) := F_{w,\gamma}(z)+C_3 t\geq \E_z[G(Z_t)]$, yields  constants $C_2,C_5\in(0,\infty)$, such that for all $t\in[1,\infty)$ we get 
\begin{equation}
\label{eq:final_estimate}
C_2t^{(1+\beta/\beta_c-\eps')/(2-2\eps')}\leq \frac{C_5}{2}(2g(z,t))^{(1+\beta/\beta_c-\epsilon')/(2-2\epsilon')} \leq \| \P_z(Z_{t_n}\in\cdot)-\pi\|_{\mathrm{TV}}.
\end{equation}

Pick arbitrary $\eps\in(0,1)$ and let $\epsilon'\in(0,1)$ 
be such that 
$0>(1+\beta/\beta_c-\eps')/(2-2\eps')>(1+\beta/\beta_c)/2-\eps$. 
Then the bound in the proposition follows from~\eqref{eq:final_estimate}.
\end{proof}

\subsection{Concluding remarks}
\label{subsec:Concluding_rem}

In the case of 
the asymptotically oblique reflection in the domain $\cD$, the local time $L_t$ either explodes in finite time or is proportional (as $t\to\infty$) to the integral of the boundary function $b$ (which in this case tends to infinity)~\cite{menshikov2022reflecting}.
In the case of the asymptotically normal reflection considered in this paper, the long-term behaviour of the local time $L$ is much harder to determine.  
As our assumptions in any compact region of $\cD$, given by~\aref{ass:domain1}, \aref{ass:covariance1}, \aref{ass:vector1},
are non-quantitative
(and, in fact, equal to the general existence and uniqueness assumptions in~\cite{lions1984stochastic}), the limiting behaviour of $L_t$
as $t\to \infty$ appears to be most tractable in the transient case, where the process spends all of its time (from some point on) in the region where the asymptotic assumptions in~\aref{ass:domain2}, \aref{ass:covariance2}, \aref{ass:vector2} can be applied. The recurrent case appears to be much harder in this generality. 

Heuristic~\eqref{eq:Pinsky_heuristic} in Section~\ref{subsec:heuristic} above suggest that  the expected local time $\E_z[L_t]$ grows as $\int_0^t 1/(1+b(X_s))\ud s$
when $t\to\infty$, implying that $\E_z[L_t]\to\infty$ as $t\to\infty$ in all the cases. 
Theorem~\ref{thm:non_explosion_moments} suggests that $X$
is diffusive. Thus,
in the transient case (i.e. $0<\beta_c<\beta<1$ by Theorem~\ref{thm:rec_tran}), where $X_t$ is expected approximately to equal to $t^{1/2}$ for large $t$, the expectation $\E_z[L_t]$  is approximately of the order $t^{1-\beta/2} \approx \int_0^t 1/(1+s^{\beta/2})\ud s$ as $t\to\infty$ (recall that $\beta$ in~\eqref{eq::beta} is the growth rate of $b$).

It is feasible that the methods developed in this paper could be applied to find deterministic (law-of-iterated-logarithm type) bounds  for $X_t$ 
of order $t^{1/2}$, which would reveal that the asymptotic behaviour 
 of $\E_z[L_t]$ as $t\to\infty$ is of order 
$t^{1-\beta/2}$. The lack of quantitative assumptions on any compact set (discussed in the first paragraph of this section), where the process spends most of its time in the recurrent case, makes it difficult to quantify the growth of $L_t$.
This is left as an open problem. 

In our proofs, we circumvent the problem of having to quantify the long-term behaviour of local time $L$ by localising the process and/or controlling the sign of the local time term via a suitable choice of the state space transformation. However, unlike in the asymptotically oblique case~\cite{menshikov2022reflecting} (where 
%analogous principles were applied to the asymptotically oblique reflection and 
the long-term behaviour of local time can be deduced from the results and the SDE itself), in the asymptotically normal case the results obtained in this paper do not provide directly any information about the growth of local time.

\appendix

\section{A lower bound on the convergence to stationarity of a Markov process}
\label{subsec:lower_bounds_convergence_to_stationarity}
Fix $m\in\N$ and 
let $\kappa=(\kappa_t)_{t\in\RP}$ be a Markov process on an unbounded domain $\cD_\kappa$ in $\R^m$ % with continuous trajectories.  
with an invariant distribution $\pi_\kappa$ satisfying $\pi_\kappa(\cdot) = \int_{\cD_\kappa}\P_u(\kappa_t\in\cdot) \pi_\kappa(\ud u)$ for every $t>0$.
Via a suitable Lyapunov function, the following lemma converts a lower bound estimate on the tail of invariant distribution $\pi_\kappa$ into a lower bound on the rate of convergence in total variation of the law of $\kappa_t$  to the invariant distribution $\pi_\kappa$. The elementary proof of Lemma~\ref{lem:lower_bound_convergence_rate} below is adapted from~\cite[Thm~5.1 and Cor~5.2]{hairer2010convergence}. This  lemma is key in the proof of the lower bound on the rate of convergence in total variation stated in Theorem~\ref{thm:invariant_distributon}(b).

\begin{lem}
\label{lem:lower_bound_convergence_rate}
Let 
 $\cD_\kappa$ be an an unbounded domain in $\R^m$ and
$\kappa = (\kappa_t)_{t\in\RP}$ a $\cD_\kappa$-valued Markov process with invariant distribution $\pi_\kappa$. Assume the function $G:\cD_\kappa\to [1,\infty)$ satisfies: 
\begin{itemize}
\item[(a)] there exists  $f:[1,\infty)\to(0,1]$, such that the function $F:y\mapsto yf(y)$ is increasing, $\lim_{y\uparrow\infty}F(y)=\infty$ and  $\pi_\kappa(\{v\in\cD_\kappa:G(v)\geq y\})\geq f(y)$ for all $y\in[1,\infty)$;
\item[(b)] there exists $g:\cD_\kappa\times\RP\to[1,\infty)$, such that for every $u\in\cD_\kappa$ 
the function $t\mapsto g(u,t)$ is continuous and increasing to infinity and 
%increasing and continuous in the second argument, such that for every $(u,t)\in\cD_\kappa\times\RP$ we have
$\E_u[G(\kappa_t)] \leq g(u,t)$ for all $t\in\RP$.
\end{itemize}
Then, for any starting point $u\in\cD_\kappa$
we have 
$$
\|\pi_\kappa(\cdot)-\P_u(\kappa_{t}\in\cdot)\|_{\mathrm{TV}} \geq\frac{1}{2}f\left(F^{-1}(2g(u,t))\right)
\qquad\text{for all $t\in\RP$.}
$$
\end{lem}

\begin{rem}
A good choice for the function $G$ in Lemma~\ref{lem:lower_bound_convergence_rate}
has the following properties:  the expectation   
$\E_u [G(\kappa_t)]$
is bounded as a function of the starting point $u$ and time $t$
\textit{and} 
the function 
$y\mapsto \pi_\kappa(G^{-1}([y,\infty)))$
satisfies
$\lim_{y\to\infty}\pi_\kappa(G^{-1}([y,\infty)))=0$
and
$\lim_{y\to\infty}y\pi_\kappa(G^{-1}([y,\infty)))=\infty$.
The proof of Lemma~\ref{lem:lower_bound_convergence_rate}
shows that if the assumption in~(a) holds for $y$
sufficiently large, then the conclusion of the lemma is valid for all $t$ sufficiently large. 
\end{rem}

\begin{proof}[Proof of Lemma~\ref{lem:lower_bound_convergence_rate}]
The definition of the total variation distance (together with assumption~(a)) and the Markov inequality (together with assumption~(b)) imply  that for every
$u\in\cD_\kappa$ and 
$t\in\RP$ the following inequalities hold for all $y\in[1,\infty)$:
$$
\|\pi_\kappa(\cdot)-\P_u(\kappa_t\in\cdot)\|_{\mathrm{TV}}\geq \pi_\kappa(\{v\in\R^d:G(v)\geq y\})-\P_u(G(\kappa_t)\geq y)\geq f(y) -\frac{g(u,t)}{y}.
$$
Since $F(y)\to\infty$ (as $y\to\infty$) and $t\mapsto g(u,t)$ is increasing and continuous, for all $y\in[F^{-1}(2g(u,0)),\infty)$ (where $F^{-1}$ is the inverse of the increasing function $F$
defined in~(a)), there exists 
a unique $t\in\RP$
satisfying
$F(y)=2g(u,t)\in[1,\infty)$. 
Differently put, for every $t\in\RP$, there exists
$y_t\in[F^{-1}(1),\infty)\subset [1,\infty)$
satisfying $y_t=F^{-1}(2g(u,t))$.
Thus, for every $t\in\RP$, we have 
$f(y_t)-g(u,t)/y_t=f(F^{-1}(2g(u,t)))/2$.
\end{proof}

\section{Asymptotically oscillating domain}
\label{app:oscillating_b}
\begin{lem}
\label{lem:oscilating_domain}
Let $b:\RP \to (0,\infty)$ be a $C^2$ function with $b(0)=0$, satisfying
$$b(x) = \log \log x(1 + (\log \log x)^{-2} + \sin\log \log x) \quad \text{ for } x > \exp(\re+1).$$
Then $\limsup_{x\to\infty}b(x) = \infty$ and $\liminf_{x \to \infty}b(x) = 0$. Moreover the function $b$ satisfies assumption \aref{ass:domain2} with $\beta=0$, i.e.~$\lim_{x\to\infty}xb'(x)/b(x)=\beta=0$, and $\lim_{x\to\infty}b'(x)=\lim_{x\to\infty}b''(x)=0$
\end{lem}
\begin{proof}
To show $\liminf_{x\to\infty}b(x) = 0$, consider $\ell_k := \exp(\exp(-\pi/2 + 2k\pi))$, for any $k\in\N$. It follows that 
$$
\lim_{k\to \infty}b(\ell_k) = \lim_{k\to\infty}(\log\log \ell_k)^{-1}= \lim_{k\to\infty}(-\pi/2 + 2k\pi)^{-1} = 0.
$$
Similarly, to show $\limsup_{x\to\infty}b(x)=\infty$, consider $\tilde \ell_k := \exp(\exp(2k\pi))$, for any $k\in\N$. We obtain
$$
\lim_{k\to \infty}b(\tilde\ell_k) = \lim_{k\to\infty}(\tilde \ell_k + (\log\log \tilde\ell_k)^{-1})= \lim_{k\to\infty}2k\pi + (2k\pi)^{-1} =\infty.
$$
The first two derivatives of $b$ on $x > \exp(\re+1)$ take the form
\begin{align*}
    b'(x) &= (x\log x)^{-1}(1-(\log\log x)^{-1} + \sin \log\log x + \cos\log\log x),\\
    b''(x) &= (x\log x)^{-2}((1+\log x)(1 -(\log\log x)^{-2} + \sin\log\log x+\cos\log\log x) + 2(\log\log x)^{-3}),
\end{align*}
implying that $\lim_{x\to\infty} b'(x) =\lim_{x\to\infty} b''(x) = 0$. For the result about $\beta$ it is enough to show $\limsup_{x\to\infty} |xb'(x)|/b(x) \leq 0$. We estimate
\begin{align*}
    \limsup_{x\to\infty}\frac{xb'(x)}{b(x)} &=  \limsup_{x\to\infty}\frac{(\log x)^{-1}\lvert1-(\log\log x)^{-1} + \sin \log\log x + \cos\log\log x)}{\log \log x(1 + (\log \log x)^{-2} + \sin\log \log x)} \\
    &\leq \limsup_{x\to\infty} \frac{\log\log x(3 + (\log\log x)^{-1})}{\log x} = 0,\\
\end{align*}
this concludes the proof.
\end{proof}

\section*{Acknowledgements}
%\addcontentsline{toc}{section}{Acknowledgements}

MB is funded by the CDT in Mathematics and Statistics at The University of Warwick.
AM is supported by EPSRC under grants EP/W006227/1 \& EP/V009478/1 and by The Alan Turing
Institute under the EPSRC grant EP/X03870X/1.
The work of AW is supported by EPSRC grant EP/W00657X/1.
We thank Krzysztof Burdzy for drawing our attention to the literature on the uniform ergodicity for normally reflected Brownian motion. 

\bibliography{RefDiff}
\bibliographystyle{amsplain}

\end{document}